\documentclass{article}
\usepackage{psfrag}
\usepackage[parfill]{parskip}
\usepackage{float, graphicx}
\usepackage[]{epsfig}
\usepackage{amsmath, amsthm, amssymb,}
\usepackage{epsfig}
\usepackage{verbatim}
\usepackage{multicol}
\usepackage{url}
\usepackage{latexsym}
\usepackage{mathrsfs}
\usepackage[colorlinks, bookmarks=true]{hyperref}
\usepackage{graphicx}
\usepackage{amsmath}
\usepackage{enumerate}
\usepackage[normalem]{ulem}
\usepackage{bm}
\usepackage{dsfont}
\usepackage{stmaryrd}
\usepackage{enumitem}
\usepackage{tikz}
\usepackage{mathtools}
\usepackage[noadjust]{cite}
\usepackage{color}
\usepackage{xcolor}
\usepackage{hyperref}

\makeatletter

\def\@settitle{\begin{center}%
    \bfseries
 \normalfont\LARGE\@title
  \end{center}%
}
\def\@setauthors{\begin{center}%
 \normalsize\@author
  \end{center}%
}

\makeatother

\numberwithin{equation}{section}

\renewcommand{\cal}{\mathcal}
\newcommand\cA{{\mathcal A}}
\newcommand\cB{{\mathcal B}}
\newcommand{\cC}{{\cal C}}

\newcommand{\cE}{{\cal E}}

\newcommand{\cG}{{\cal G}}

\newcommand{\cL}{{\cal L}}
\newcommand{\cM}{{\cal M}}

\newcommand{\cS}{{\mathcal S}}

\newcommand{ \sfx }{{\mathsf x}}

\newcommand{ \sft }{{\mathsf t}}

\newcommand{\sfW}{{\mathsf W}}
\newcommand{\sfG}{{\mathsf G}}

\newcommand{\sfK}{{\mathsf K}}

\newcommand{\sfE}{{\mathsf E}}

\newcommand{\sfP}{{\mathsf P}}
\newcommand{\sfH}{{\mathsf H}}

\newcommand{\fa}{{\mathfrak a}}
\newcommand{\fb}{{\mathfrak b}}
\newcommand{\fc}{{\mathfrak c}}
\newcommand{\fC}{{\mathfrak C}}

\newcommand{\fM}{{\mathfrak M}}
\newcommand{\fP}{{\mathfrak P}}

\newcommand{\bms}{\bm{s}}
\newcommand{\bme}{{\bm{e}}}

\newcommand{\bmv}{{\bm{v}}}

\newcommand{\bmx}{{\bm{x}}}

\newcommand{\rd}{{\rm d}}

\newcommand{\ri}{\mathrm{i}}

\newcommand{\bC}{{\mathbb C}}

\newcommand{\bE}{\mathbb{E}}

\newcommand{\bP}{\mathbb{P}}

\newcommand{\bR}{{\mathbb R}}

\newcommand{\bZ}{\mathbb{Z}}

\DeclareMathOperator{\supp}{supp}
\DeclareMathOperator{\dist}{dist}

\DeclareMathOperator{\cov}{cov}

\DeclareMathOperator{\OO}{O}
\DeclareMathOperator{\oo}{o}
\DeclareMathOperator{\argmax}{argmax}

\renewcommand{\Im}{\mathop{\mathrm{Im}}}

\newcommand{\deq}{\mathrel{\mathop:}=} 
 
\renewcommand{\leq}{\leqslant}
\renewcommand{\geq}{\geqslant}

\newcommand{\del}{\partial}

\newcommand{\beq}{\begin{equation}}
\newcommand{\eeq}{\end{equation}}

\theoremstyle{plain} 
\newtheorem{theorem}{Theorem}[section]
\newtheorem*{theorem*}{Theorem}
\newtheorem{lemma}[theorem]{Lemma}
\newtheorem*{lemma*}{Lemma}
\newtheorem{corollary}[theorem]{Corollary}
\newtheorem*{corollary*}{Corollary}
\newtheorem{proposition}[theorem]{Proposition}
\newtheorem*{proposition*}{Proposition}
\newtheorem{assumption}[theorem]{Assumption}
\newtheorem*{assumption*}{Assumption}

\newtheorem{definition}[theorem]{Definition}
\newtheorem*{definition*}{Definition}

\newtheorem*{example*}{Example}
\newtheorem{remark}[theorem]{Remark}

\newtheorem*{remark*}{Remark}
\newtheorem*{remarks*}{Remarks}

\newtheorem{ansatz}[theorem]{Ansatz}

\def\author#1{\par
    {\centering{\authorfont#1}\par\vspace*{0.05in}}
}

\def\titlefont{\fontsize{13}{15}\bfseries\boldmath\selectfont\centering{}}
\def\authorfont{\fontsize{13}{15}}

\let\affiliationfont\rhfont

\def\address#1{\par
    {\centering{\affiliationfont#1\par}}\par\vspace*{11pt}
}

\def\body{
\setcounter{footnote}{0}
\def\thefootnote{\alph{footnote}}
\def\@makefnmark{{$^{\rm \@thefnmark}$}}
}

\def\title#1{
    \thispagestyle{plain}
    \vspace*{-14pt}
    \vskip 79pt
    {\centering{\titlefont #1\par}}%
    \vskip 1em
}

\newcommand{\cout}{{\omega_+}}
\newcommand{\cin}{{\omega_-}}

\newcommand{\up}[1]{%
\begin{tikzpicture}[#1]%
\draw (0,0) -- (0ex,1ex);%
\draw (0,0) -- (1ex,0ex);%
\draw (0ex,1ex) -- (1ex,1ex);%
\draw (1ex,0ex) -- (1ex,1ex);%
\end{tikzpicture}%
}
\newcommand{\rig}[1]{%
\begin{tikzpicture}[#1]%
\draw (0,0) -- (1ex,0ex);%
\draw (0,0) -- (1ex,1ex);%
\draw (1ex,1ex) -- (2ex,1ex);%
\draw (1ex,0ex) -- (2ex,1ex);%
\end{tikzpicture}%
}
\newcommand{\emp}[1]{%
\begin{tikzpicture}[#1]%
\draw (0,0) -- (0ex,1ex);%
\draw (0,0) -- (1ex,1ex);%
\draw (1ex,1ex) -- (1ex,2ex);%
\draw (0ex,1ex) -- (1ex,2ex);%
\end{tikzpicture}%
}

\begin{document}

\title{Height  Fluctuations of Random Lozenge Tilings Through Nonintersecting Random Walks}

\vspace{1.2cm}

 \author{Jiaoyang Huang}
\address{New York University\\
   E-mail: jh4427@nyu.edu}

~\vspace{0.3cm}

\begin{abstract}
In this paper we study  height fluctuations of random lozenge tilings of polygonal domains on the triangular lattice through nonintersecting Bernoulli random walks.  For a large class of polygons which have exactly one horizontal upper boundary edge,
we show that these random height functions converge to a Gaussian Free Field as predicted by Kenyon and Okounkov \cite{MR2358053}. A key ingredient of our proof is a dynamical version of the discrete loop equations as introduced by Borodin, Guionnet and Gorin \cite{MR3668648},
which might be of independent interest.

\end{abstract}


\section{Introduction}
A dimer covering, or perfect matching, of a finite graph
is a set of edges covering all the vertices exactly once. The dimer model is the study
of a uniformly chosen dimer covering of a graph. 
It is a classical model of statistical physics, going back to works of Kasteleyn \cite{kasteleyn1961statistics} and Temperley--Fisher \cite{temperley1961dimer}, who computed its partition function as the determinant of a signed adjacency matrix of the graph, the Kasteleyn matrix. 
Since then extensive research has been devoted to various facets of dimer models; we refer the reader to \cite{MR2523460, VG2020} for discussions of recent progress. 

A Dimer covering of bipartite graphs such as the
honeycomb graph or square grid can be viewed as
a Lipschitz function in $\bR^3$, called the discrete height function \cite{thurston1990conway}
which contains all the information about the dimer configuration. 
Hence the dimer model can be
viewed as a model of random Lipschitz functions. A key question
in the dimer model concerns the large-scale behavior of this height function. It is widely believed
that under very general assumptions, the fluctuations of height functions
are described by a conformally invariant process, the Gaussian free
field.

In this paper we study height functions of random lozenge tilings of polygonal domains, corresponding to the dimer model on the honeycomb lattice. 
In this case, a lozenge tiling can be viewed as the projection of its height function $\sfH(\sfx,\sft)$ (a $3$-dimensional stepped sruface) onto the $(\sfx,\sft)$ plane.
To be concrete, we use the standard square grid, with all sides parallel either to the coordinate axes or the vector $(1,1)$, as shown in Figure \ref{f:lattice}, and the lozenges become
\begin{align*}
\up{scale=2}, \quad \rig{scale=2}, \quad \emp{scale=2}.
\end{align*}
We will denote the horizontal and the
vertical integer coordinates on the square grid by $\sfx$ and $\sft$, respectively.
We make the convention that lozenges of type $\emp{scale=1}$ correspond to horizontal planes, that is, planes where the height function is constant. 
We make $\sfH(\sfx,\sft)$ a continuous function by linear interpolation of the lattice points. 
We also
require that $\sfH(\sfx,\sft)$ is zero at the lower left corner of the polygon. The height function along the boundary of the polygonal domain is determined in the obvious way, i.e. the height does not change along boundaries 
with slopes $1,\infty$, and increases with rate $1$ along boundaries
with slope $0$.
After ignoring lozenges of shape $\emp{scale=1}$, we can also identify lozenge tilings with nonintersecting Bernoulli random walks, as shown in Figure \ref{f:heart_shape}. The random walk interpretation  plays an important role in several recent studies of lozenge tilings \cite{gorin2019universality, aggarwal2019universality}.

\begin{figure}
\begin{center}
 \includegraphics[scale=0.4,trim={0cm 8cm 0 7cm},clip]{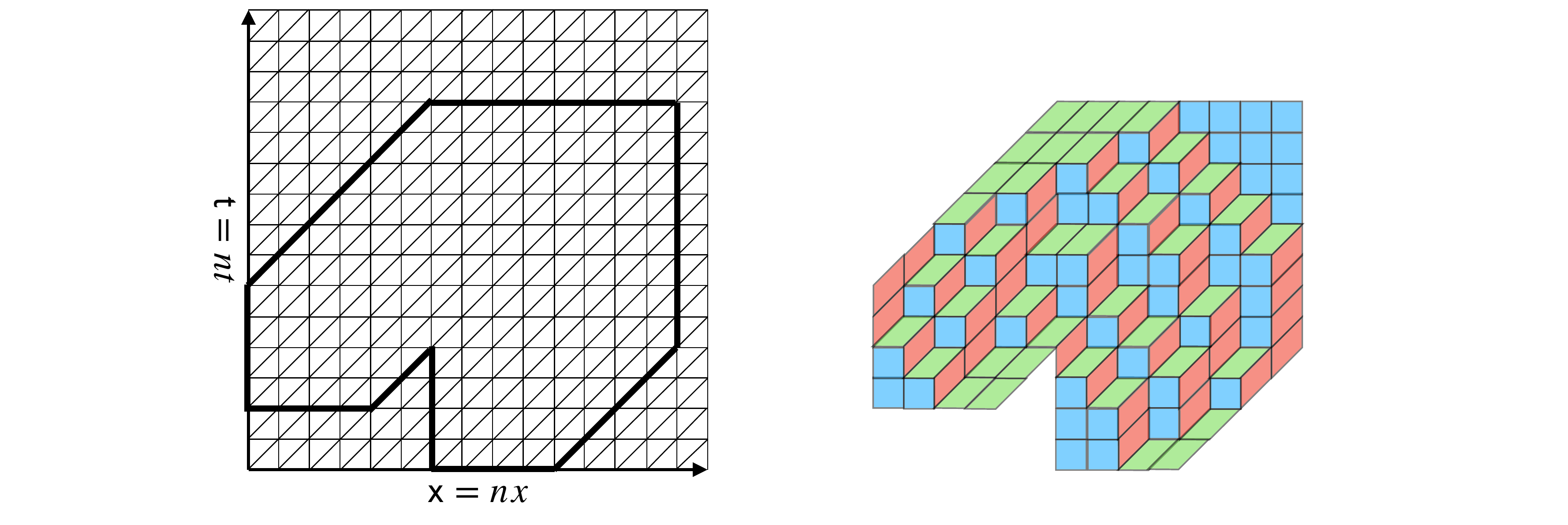}
 \caption{Lozenge tiling of polygonal domains on the standard square grid, where all sides are parallel either to the coordinate axes or the vector $(1,1)$.}
 \label{f:lattice}
 \end{center}
 \end{figure}
 

Let ${\mathfrak P}$ be a polygonal domain in the horizontal strip $t\in [0,T]$, formed by $3d$ segments
with slopes $0,1,\infty$, cyclically repeated as we follow the boundary in the
counterclockwise direction, and $\sfP$ is obtained from ${\mathfrak P}$ after rescaling by a factor $n$. It is proven in
\cite{MR1815214}, 
 height functions $\sfH(\sfx,\sft)$ for random lozenge tilings of $\sfP$ converge in probability,
\begin{align}\label{e:hlimit0}
\frac{\sfH(n x , n t )}{n}\rightarrow  h^*(x ,  t ),\quad ( x ,  t )\in {\mathfrak P}.
\end{align}
The limiting height function $h^*(x ,  t )$ is the surface tension minimizer among all Lipschitz functions $h$ taking given values on the boundary of the
domain $\fP$, given explicitly by a variational problem:
\begin{align}\label{e:vp0}
 h^*=\argmax_{ h\in {\mathfrak H({\mathfrak P})}}\iint \sigma(\nabla  h)\rd  x \rd  t .
\end{align}
 We will discuss more on the variational problem in Section \ref{s:VP}.
 
One feature of the limit shapes of random lozenge tilings is the presence of frozen regions which contain only one type of lozenges and the liquid regions which contain all three types of lozenges. The fluctuations of height functions are very different in the frozen and liquid region. In the frozen region, the height function is flat, and there are essentially no fluctuations. The situation is more interesting in the liquid region, where it is predicted in \cite{MR2358053},  height functions converge to the Gaussian Free Field  with zero boundary condition.
The main result of this paper verifies this prediction for a large class of polygons
which have exactly one horizontal upper boundary edge,
e.g. the heart-shaped polygon in Figure \ref{f:lattice}.
Those polygons correspond to nonintersecting Bernoulli random walks with tightly packed ending configurations. We refer to Section \ref{s:rw} for more precise definition of such domains.


\begin{theorem}\label{c:GFF}
Fix a polygonal domain ${\mathfrak P}$ which has exactly one horizontal upper boundary edge,
see Definition \ref{def:oneend} for a formal definition.
Let $\sfH(\sfx,\sft)$ be the height function of random lozenge tilings of the domain $\sfP$ obtained from ${\mathfrak P}$  by rescaling a factor $n$, then as $n$ goes to infinity, its fluctuation converges to a Gaussian Free Field on the liquid region $\Omega(\fP)$ with zero boundary condition,
\begin{align*}
\sqrt{\pi}\left(\sfH(n  x ,n t )-\bE[\sfH(n x , n t )]\right)\rightarrow {\rm{GFF}},\quad ( x ,  t )\in \Omega({\mathfrak P}).
\end{align*}
\end{theorem}


{\begin{figure}
\begin{center}
 \includegraphics[scale=0.4,trim={0cm 8cm 0 7cm},clip]{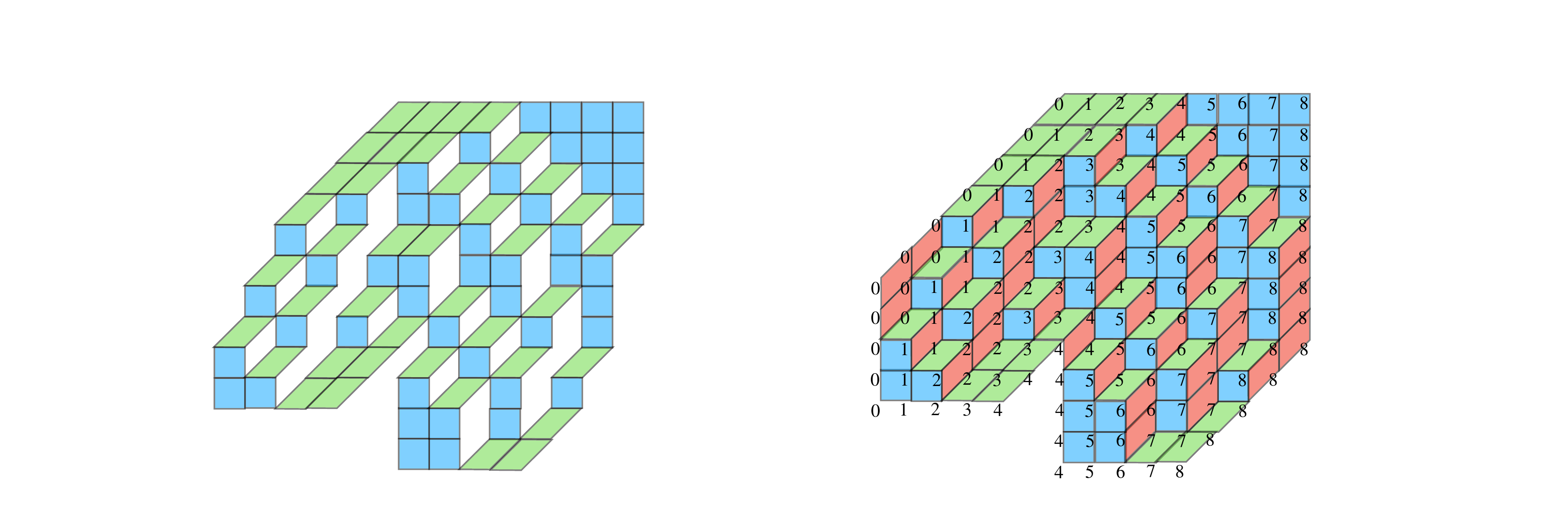}
 \caption{In the first subplot, we identify the lozenge tilings as nonintersecting Bernoulli walks; In the second subplot, we view lozenge tilings as projections of $3$-dimensional stepped surfaces onto the $(\sfx,\sft)$ plane.}
 \label{f:heart_shape}
 \end{center}
 \end{figure}}

 There are several approaches, as summarized in \cite{VG2020},  to the proof of convergence of height functions to the Gaussian Free Field for particular classes of domains.
One approach was suggested in \cite{kenyon2000conformal, kenyon2001dominos, kenyon2008height}.
The correlation functions of tilings can be computed using the inverse Kasteleyn matrix, which can be thought of as a discrete harmonic function. In this way the study of random tilings can be reduced to the study of the convergence of discrete harmonic functions.
This approach works for large families of domains which have no frozen regions \cite{russkikh2018dominos, russkikh2018dimers, berestycki2016universality}. We refer to  \cite{chelkak2020dimer, chelkak2020fluctuations} for most recent developments in this direction.
For trapezoidal domains,  its inverse Kasteleyn matrix can be expressed as a double contour integral, which makes it possible to study the asymptotics \cite{borodin2014anisotropic, petrov2015asymptotics}.
In some cases, the marginal distributions of random tilings can be identified
with discrete log-gases, and can be analyzed using discrete loop equations \cite{MR3668648, dimitrov2019log} and tools from discrete orthogonal polynomials \cite{duits2018global}.
For domains for which the asymptotics for the partition functions are known, the information of height functions can be extracted by applying differential operators to the partition functions, see \cite{bufetov2018fluctuations, bufetov2019fourier} for tilings of non-simply
connected domains.

In this paper we take a dynamical approach to study random lozenge tilings of polygonal domains. In the remaining of this introduction section, we give an outline of the proof. We identify lozenge tilings with nonintersecting Bernoulli  walks, and denote $N_t(x_1, x_2, \cdots, x_{m(t)})$ the number of nonintersecting Bernoulli walks staying in $\fP$ and starting from particle configuration $\bmx=(x_1,x_2,\cdots ,x_{m(t)})$ at time $t$,
where $m(t)$ is the number of particles at time $t$ and it is uniquely determined by $\fP$.
 Then it satisfies the following discrete heat equation
\begin{align}\label{e:Wt+recur0}
N_{t}(\bmx)
=\sum_{\bme=(e_1,e_2,\cdots, e_{m(t)})\in\{0,1\}^{m(t)}}
N_{t+1/n}(\bmx+\bme/n).
\end{align}
We can use the partition function \eqref{e:Wt+recur0} to define a nonintersecting Bernoulli random walk with transition probability given by
\begin{align}\label{e:randomwalk0}
\bP(\bmx(t+1/n)=\bmx+\bme/n|\bmx(t)=\bmx)=\frac{N_{t+1/n}(\bmx+\bme/n)}{N_{t}(\bmx)},
\end{align}
for any $t\in[0,T)\cap \bZ_n$, and $\bme=(e_1, e_2, \cdots, e_{m(t)})\in \{0,1\}^{m(t)}$, .
Then $\{\bmx(t)\}_{0\leq t\leq T}$ with $\bmx(0)$ corresponding to the bottom boundary of $\fP$,  is uniformly distributed among all possible nonintersecting Bernoulli walks. Especially, by our identification of lozenge tilings and nonintersecting Bernoulli  walks, it has the same law as uniform random lozenge tilings of the domain $\fP$.

The limit shapes of random lozenge tilings are characterized by the variational problem \eqref{e:vp0}. We can use it to define a measure valued  Hamiltonian system: Fix time $s\in [0,T)\cap \bZ_n$, and a particle configuration $\bmx=(x_1,x_2,\cdots ,x_{m(s)})$ at time $s$. Let $\beta$ be the height function corresponding to the particle configuration $\bmx$ and define
\begin{align}\label{e:varWt0}
W_ s ( \beta )=\sup_{ h\in H({\mathfrak P};\beta, s)}\int_ s ^T\int_\bR\sigma(\nabla  h)\rd x \rd t .
\end{align}
where the sup is over Lipschitz functions $h$ which take given values on the boundary of the domain $\fP\cap \bR\times [s, T]$, and coincide with $\beta$ on the bottom boundary.
There is a Riemann surface associated with the variational problem \eqref{e:varWt0}.  Properties of $W_s(\beta)$ can be understood using tools for Riemann surfaces, i.e. Rauch variational formula.

The  Hamilton's principal function $W_t(\beta)$ \eqref{e:varWt0} should give a good approximation for the partition function $N_t(\bmx)$. We make the following ansatz
\begin{align}\label{e:ansatz0}
N_t(\bmx)=e^{n^2 W_t(\beta)}E_t(\bmx),
\end{align}
where $\beta$ corresponds to the particle configuration $\bmx$. By plugging \eqref{e:ansatz0}, we can reformulate \eqref{e:randomwalk0} as
\begin{align}\label{e:Eratio}
\bP(\bmx(t+1/n)=\bmx+\bme/n|\bmx(t)=\bmx)\approx \frac{1}{Z_t(\bmx)} \frac{V(\bmx+\bme/n)}{V(\bmx)} \prod_{i=1}^m g_t^{e_i}(x_i;\beta,t) \frac{E_{t+1/n}(\bmx+\bme/n)}{E_{t}(\bmx)},
\end{align}
where $V(\bmx)$ is the Vandermonde determinant, $\ln g_t(x;\beta,t)=\del_\beta W_t(\beta)-H(\del_x\beta)$, and $H(\del_x\beta)$ is the Hilbert transform of $\del_x \beta$.
\eqref{e:ansatz0} and \eqref{e:Eratio} are only for illustration, for a rigorous proof, we need to keep track the higher order error from discretization, the change for the number of particles, and the boundary effects. The expressions are more involved, but basic ideas are the same.
Then we can solve for the correction terms $E_t(\bmx)$ in \eqref{e:Eratio} by the Feynman-Kac formula, using a nonintersecting Bernoulli random walk with drift $g_t(x;\beta,t)$
\begin{align}\label{e:form0}
\bP(\bmx(t+1/n)=\bmx+\bme/n|\bmx(t)=\bmx)\propto  \frac{V(\bmx+\bme/n)}{V(\bmx)} \prod_{i=1}^m g_t^{e_i}(x_i;\beta,t).
\end{align}
The boundary condition of $\fP$ is encoded in the drift $g_t(x;\beta,t)$, i.e. it has a pole at left boundaries of $\fP$ with slope $1$, and a zero at right boundaries of $\fP$ with slope $0$. Thus for nonintersecting Bernoulli random walks in the form \eqref{e:form0}, particles are constrained inside $\fP$.

The loop (or Dyson-Schwinger) equation is an important tool to study fluctuations of interacting particle systems.
They were introduced to the mathematical community by Johansson \cite{MR1487983} to derive macroscopic central limit theorems for general $\beta$-ensembles, see also \cite{MR3010191, borot-guionnet2, KrSh}.
For discrete $\beta$-ensembles,  the macroscopic central limit theorems were proved by Borodin, Guionnet and Gorin \cite{MR3668648}. The proof relies on Newrasov's equations, which  originated in the work of Nekrasov and his collaborators \cite{Nekrasov, Nek_PS,Nek_Pes}. These equations can be analyzed  in a spirit very close to the uses of loop equations in the continuous setting.
To analyze nonintersecting Bernoulli random walks in the form \eqref{e:form0}, 
we develop a dynamical version (discrete time and discrete space) of the discrete loop equations. Another dynamical version (continuous time and discrete space) was previously used to study the $\beta$-nonintersecting Poisson random walks in \cite{huang2017beta}.
Discrete multi-level analogues of loop equations were developed in \cite{dimitrov2019asymptotics} to study discrete $\beta$-corners processes. 
Unfortunately, in our setting, for the analysis of the discrete loop equations, we need that $g_t(x;\beta,t)$ is analytic in  a neighborhood of the support of the particle configuration $\bmx$. This is where we need the assumption that the domain $\fP$ has exactly one horizontal upper boundary edge (as in Definition \ref{def:oneend}). Otherwise, $g_t(x;\beta,t)$ may have branch points close to the support of the particle configuration $\bmx$.

Using the discrete loop equations, we then show that the fluctuations for nonintersecting Bernoulli random walks in the form \eqref{e:form0} satisfy a stochastic differential equation. As a consequence, the fluctuations are asymptotically Gaussian. Using a heuristic given by Gorin \cite[Lecture 12]{VG2020}, we identify these Gaussian fluctuations as a Gaussian Free Field on the liquid region, as predicted by Kenyon and Okounkov \cite{MR2358053}. In the last step, we solve for the correction terms $E_t(\bmx)$  in \eqref{e:Eratio}  by Feynman-Kac formula, using the nonintersecting Bernoulli random walk \eqref{e:form0}. It turns out the leading order term of the ratio $E_{t+1/n}(\bmx+\bme/n)/E_{t}(\bmx)$ in \eqref{e:Eratio} splits, and can be absorbed in $g_t(x;\beta,t)$. As a consequence, the nonintersecting Bernoulli random walk corresponding to uniform random lozenge tiling is also in the form of \eqref{e:form0}. Then we conclude that the fluctuations of height functions for random lozenge tilings converge to a Gaussian Free Field.



We now outline the organization for the rest of the paper. In Section \ref{s:VP}, we recall the variational problem which characterizes the limiting height function for random lozenge tilings from \cite{MR1815214}. 
In Section \ref{s:rw}, we identify lozenge tilings with nonintersecting Bernoulli walks, and introduce the assumption on our polygonal domains.
In Section \ref{s:burger}, we recall from \cite{MR2358053}, the complex Burgers equation and analytic curves associated with the variational problem. 
We use the rate function of the variational problem to define a measure valued Hamiltonian system, and study its properties using tools from Riemann surfaces, i.e. Rauch variational formula.  
In Section \ref{s:ansatz}, we make an ansatz for the partition function of the nonintersecting Bernoulli random walk, which corresponds to uniform random lozenge tilings.
In Section \ref{s:eq}, we express the correction term in the ansatz using Feynman-Kac formula, and a nonintersecting Bernoulli random walk with drift.
In Section \ref{s:loopeq}, we introduce a dynamical version of the  discrete loop equations \cite{MR3668648}, and study its asymptotic properties. The discrete loop equations are used to study a family of nonintersecting Bernoulli random walks with drift, in Section \ref{s:Gauss}. We prove the fluctuation of their height functions are asymptotically Gaussian.
In Section \ref{s:solve}, we solve the correction term in the ansatz, and show the nonintersecting Bernoulli random walk corresponding to uniform random lozenge tilings belongs to the family, which have been studied using the discrete loop equations.
In Section \ref{s:GFF}, we identify these fluctuations as a Gaussian Free Field on the liquid region, as predicted by Kenyon and Okounkov \cite{MR2358053}, following a heuristic given by Gorin \cite[Lecture 12]{VG2020}.

\paragraph{Notations}

We denote $\cM(\bR)$ the set of probability measures over $\bR$, and ${\mathcal{C}}([0,1],\cM(\bR))$ the set of continuous measure valued process over $[0,1]$.
We use $\fC$ to represent large universal constant, and $\fc$ a small universal constant,
which
may be different from line by line. We write that $X = \OO(Y )$ if there exists some universal constant such
that $|X| \leq \fC Y$ . We write $X = \oo(Y )$, or $X \ll Y$ if the ratio $|X|/Y\rightarrow \infty$ as $n$ goes to infinity. We write
$X\asymp Y$ if there exist universal constants such that $\fc Y \leq |X| \leq  \fC Y$.
We denote the set $\bZ_n=\{k/n: k\in \bZ\}$.

\paragraph{Acknowledgements}
The author thanks Vadim Gorin for helpful comments on the draft of this paper.
The research of J.H. is supported by the Simons Foundation as a Junior Fellow
at the Simons Society of Fellows.

\section{Variational Principle}\label{s:VP}
Let ${\mathfrak P}$ be a polygonal domain in the horizontal strip $t\in[0,T]$, formed by $3d$ segments
with slopes $0,1,\infty$, cyclically repeated as we follow the boundary in the
counterclockwise direction, and $\sfP$ is obtained from ${\mathfrak P}$ by rescaling  a factor $n$. It is proven in
\cite{MR1815214}, 
the height function $\sfH(\sfx,\sft)$ for random lozenge tilings of $\sfP$ converges in probability in sup norm,
\begin{align}\label{e:hlimit}
\frac{\sfH(n x , n t )}{n}\rightarrow  h^*(x ,  t ),\quad ( x ,  t )\in {\mathfrak P},
\end{align}
which is characterized by a variational problem. We remark the result in \cite{MR1815214} is more general, and works for simply connected domains of arbitrary shape, not necessary polygonal domain.

Before stating the results in \cite{MR1815214}, we recall the Lobachevsky function  $L(\cdot)$: 
\begin{align*}
L(\theta)=-\int_0^\theta \ln(2\sin t)\rd t,
\end{align*}
and the surface tension associated to a triple $(p_{\up{scale=0.8}}, p_{\rig{scale=0.8}}, p_{\emp{scale=0.8}})$ of densities with $p_{\up{scale=0.8}}+ p_{\rig{scale=0.8}}+ p_{\emp{scale=0.8}}=1$ and $p_{\up{scale=0.8}}, p_{\rig{scale=0.8}}, p_{\emp{scale=0.8}}\geq 0$ given by
\begin{align}\label{e:surfacetension}
\sigma(p_{\up{scale=0.8}}, p_{\rig{scale=0.8}}, p_{\emp{scale=0.8}})
=\frac{1}{\pi}(L(\pi p_{\up{scale=0.8}})+L(\pi p_{\rig{scale=0.8}})+L(\pi p_{\emp{scale=0.8}})).
\end{align}

Let ${\mathfrak P}$ be a polygonal domain, the height of its boundary, $\beta^{{\mathfrak P}}:\del{\mathfrak P}\mapsto \bR$, is determined in the following way:
$\beta^{{\mathfrak P}}=0$ at the lower left corner of ${\mathfrak P}$. $\beta^{{\mathfrak P}}$ does not change along boundaries of ${\mathfrak P}$ with slopes $1, \infty$, and increases with rate $1$ along boundaries with slopes $0$. 
If $ h:{\mathfrak P}\mapsto \bR$ is a possible limit of \eqref{e:hlimit}, by our convention of height functions, it is necessary that on the boundary of ${\mathfrak P}$,  $ h$ coincides with the height of the boundary of ${\mathfrak P}$, i.e. $ h(x,t)=\beta^{{\mathfrak P}} (x,t)$ for $ (x,t)\in \del {\mathfrak P}$.  Moreover, we can associate for $ h$ a local density triple
$(p_{\up{scale=0.8}}, p_{\rig{scale=0.8}}, p_{\emp{scale=0.8}})$ at any point $ (x,t)$ in the interior of ${\mathfrak P}$:
\begin{align}\label{e:ldensity}
\del_{x}  h(x,t)=p_{\up{scale=0.8}}+p_{\rig{scale=0.8}}=1-p_{\emp{scale=0.8}},\quad
\del_{t}  h(x,t)=-p_{\rig{scale=0.8}}.
\end{align}
Naturally, we have $p_{\up{scale=0.8}}+ p_{\rig{scale=0.8}}+ p_{\emp{scale=0.8}}=1$. For the local densities to be nonnegative $p_{\up{scale=0.8}}, p_{\rig{scale=0.8}}, p_{\emp{scale=0.8}}\geq 0$, we need that
\begin{align}\label{e:lip}
0\leq \del_{x}  h(x,t), -\del_{t}  h(x,t), \del_{x} h(x,t)+\del_t  h(x,t)\leq 1.
\end{align} 

\begin{definition}\label{def:HO}
For any polygonal domain $\mathfrak P$, We denote by ${\mathfrak H({\mathfrak P})}$ the set of functions
$ h:{\mathfrak P}\mapsto \bR$,
such that:
\begin{enumerate}
\item 
On the boundary of ${\mathfrak P}$,  $ h$ coincides with the height of the boundary of ${\mathfrak P}$, i.e. $ h(x,t)=\beta^{{\mathfrak P}} (x,t)$ for $ (x,t)\in \del {\mathfrak P}$
\item
$ h$ is Lipschitz and satisfies \eqref{e:lip} at  points where it is differentiable.
\end{enumerate}
\end{definition}

Given $ h \in {\mathfrak H}({\mathfrak P})$. For a point $ (x,t)$ in the interior
of ${\mathfrak P}$, using \eqref{e:ldensity} we identify the gradient $\nabla  h$ with the three local densities:
\begin{align*}
p_{\up{scale=0.8}}=\del_{x} h(x,t)+\del_t  h(x,t)
,\quad  p_{\rig{scale=0.8}}=-\del_t  h(x,t), \quad p_{\emp{scale=0.8}}=1-\del_{x} h(x,t).
\end{align*}
In the rest of the paper, we will just
write $\sigma(\nabla  h)=\sigma(p_{\up{scale=0.8}}, p_{\rig{scale=0.8}}, p_{\emp{scale=0.8}})$ for height functions $ h$, with the above convention.

\begin{theorem}[\cite{MR1815214}]
Let ${\mathfrak P}$ be a polygonal domain in the plane, and $\sfP$ is obtained from ${\mathfrak P}$ after rescaling by a factor $n$. The height function $\sfH(\sfx,\sft)$ of random lozenge tilings of the domain $\sfP$ converges in probability in sup norm, as $n$ goes to infinity,
\begin{align*}
\frac{H(n x , n t )}{n}\rightarrow  h^* ( x , t),\quad  ( x , t)\in {\mathfrak P},
\end{align*}
where $ h^*$ is the unique minimizer of the variational problem
The original variational principle:
\begin{align}\label{e:vp}
 h^*=\argmax_{ h\in {\mathfrak H({\mathfrak P})}}\iint \sigma(\nabla  h)\rd  x \rd  t ,
\end{align}
Moreover, we have
\begin{align*}
\frac{1}{n^2}\ln |\{\text{lozenge tilings of $\sfP$}\}| =\iint \sigma(\nabla  h^*)\rd  x \rd  t +\oo(1).
\end{align*}
\end{theorem}

\section{Random Walk Representation}\label{s:rw}
In this section, we formally identify lozenge tilings of polygonal domains with nonintersecting Bernoulli walks, and give the formal definition for polygonal domains which have exactly one horizontal upper boundary edge, i.e. the assumption in our main Theorem \ref{c:GFF}.

 Let ${\mathfrak P}$ be a polygonal domain in the plane, formed by $3d$ segments
with slopes $0,1,\infty$, cyclically repeated as we follow the boundary in the
counterclockwise direction. Without loss of generality, we take $\mathfrak P$ in the horizontal strip $0\leq t\leq T$, with horizontal boundary edges at times $0=T_0<T_1<T_2<\cdots<T_p=T$. Especially, the lowest horizontal boundary of $\mathfrak P$ is at $t=0$, and its highest horizontal boundary is at $t=T$. 

We recall from Definition \ref{def:HO} the set $\mathfrak H({\mathfrak P})$ of possible height functions over ${\mathfrak P}$. For any lozenge tiling of $\sfP$, its height function $\sfH(\sfx,\sft)$ after rescaling
\begin{align*}
h(x,t)\deq \frac{\sfH(nx, nt)}{n}, \quad (x, t)\in {\mathfrak P},
\end{align*}
gives a height function over ${\mathfrak P}$, it is easy to see that $h(x,t)
\in \mathfrak H({\mathfrak P})$. Especially, it follows that $\sfH(nx, nt)=n\beta^{{\mathfrak P}}(x,t)$ for $(x, t)\in \del {\mathfrak P}$ on the boundary of ${\mathfrak P}$.

For any time $t$, we denote the lower boundary of ${\mathfrak P}\cap\bR\times [t,T]$ as
\begin{align}\label{e:lb}
\{ x :  (x,t+)\in {\mathfrak P}\}=[\fa_1(t), \fb_1(t)]\cup[{\mathfrak a}_2(t), \fb_2(t)]\cup\cdots \cup [t_{r(t)}(t),\fb_{r(t)}(t)].
\end{align}
We remark that if ${\mathfrak P}$ contains horizontal edges at time $t$ i.e. $t\in\{T_0, T_1, T_2, \cdots, T_p\}$,  the time slice $\{ x :  (x,t)\in {\mathfrak P}\}$ is ambiguous. In \eqref{e:lb}, we define it as the lower boundary of ${\mathfrak P}\cap\bR\times [t,T]$. So a horizontal boundary edge of ${\mathfrak P}$ is not included in \eqref{e:lb} if it is an upper boundary edge of ${\mathfrak P}$. Later, to make statements precise, we also need the notation of the 
upper boundary of
${\mathfrak P}\cap\bR\times [0,t]$:
\begin{align}\label{e:upb}
\{x :  (x,t-)\in {\mathfrak P}\}=[\fa_1(t-), \fb_1(t-)]\cup[\fa_2(t-), \fb_2(t-)]\cup\cdots \cup [\fa_{r(t-)}(t-),\fb_{r(t-)}(t-)].
\end{align}
If there are no horizontal boundary edges of ${\mathfrak P}$ at time $t$, i.e. $t\neq T_0, T_1, T_2, \cdots, T_p$,  then \eqref{e:lb} and \eqref{e:upb} are the same. Otherwise, a horizontal boundary edge is not included in \eqref{e:lb} if it is a upper boundary edge of ${\mathfrak P}$; it is not included in \eqref{e:upb} if it is a lower boundary edge of ${\mathfrak P}$. See Figure \ref{f:timte_t_slice} for an illustration.
\begin{figure}
 \includegraphics[scale=0.48,trim={0cm 14cm 0 7cm},clip]{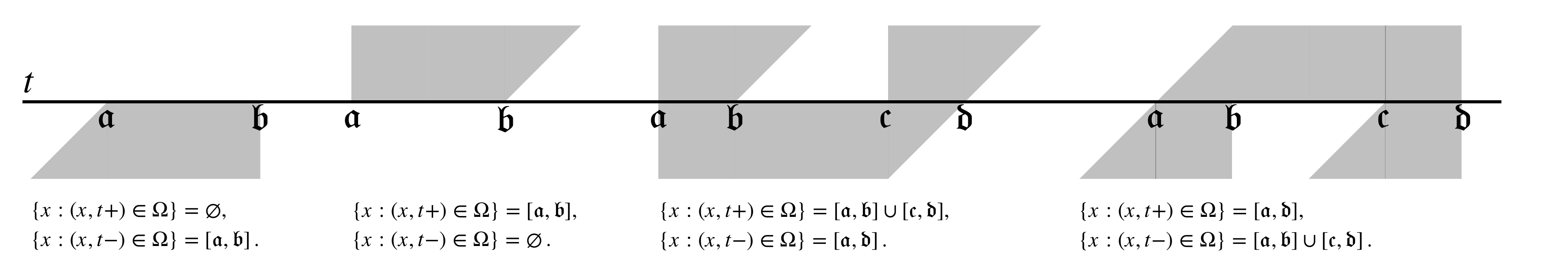}
 \caption{The slice of the polygonal domain ${\mathfrak P}$ at time $t$.}
 \label{f:timte_t_slice}
 \end{figure}

For any $0\leq t<T$, we can view $\del_{x}  h(x,t)$ as a positive measure on $[\fa_1(t), \fb_1(t)]\cup[\fa_2(t), \fb_2(t)]\cup\cdots \cup [\fa_{r(t)}(t),\fb_{r(t)}(t)]$. We denote it 
\begin{align}\label{e:defrhot+}
{ \rho}_{t}( x)\deq\del_{x}  h(x,t),\quad (x, t)\in {\mathfrak P}.
\end{align}
Then it has bounded density $0\leq \rho_{ t}(x )\leq 1$. If $[\fa',\fb']\subset [\fa_i(t), \fb_i(t)]$ is a horizontal boundary edge of ${\mathfrak P}$, then \begin{align}\label{e:cond1}
\rho_{ t }( x )=1,\quad x \in [\fa', \fb'].
\end{align} 
Moreover, the integrals of $\rho_{ t }( x )$ over the intervals $[\fa_i(t), \fb_i(t)]$ are determined by the boundary condition of $ h(x,t)$: 
\begin{align}\label{e:cond2}
\int_{\fa_i(t)}^{\fb_i(t)} \rho_{ t }( x )\rd x = \beta^{{\mathfrak P}}(\fb_i(t), t )- \beta^{{\mathfrak P}}(\fa_i(t),  t ).
\end{align}
We can recover the height function $ h(x,t)$ from ${ \rho}_{t}(x )$: for any $ x \in [\fa_i(t), \fb_i(t)]$, 
\begin{align}\label{e:recoh}
 h(x,t)= \beta^{{\mathfrak P}}(\fa_i(t),  t )+\int_{\fa_i(t)}^ x \rho_{ t }( x )\rd  x .
\end{align}

For any lozenge tiling of $\sfP$, with height function $\sfH(\sfx,\sft)$, we
can identify it with nonintersecting Bernoulli random walks, as shown in Figure \ref{f:heart_shape}. More precisely, let $h(x,t)=\sfH(nx,nt)/n$ for $0\leq t<T$. As a positive measure on $[\fa_1(t), \fb_1(t)]\cup[\fa_2(t), \fb_2(t)]\cup\cdots \cup [\fa_{r(t)}(t),\fb_{r(t)}(t)]$, $\del_{x}  h(x,t)$ is in the following form
\begin{align}\label{e:dht}
\del_{x}  h(x,t)=\sum_{i=1}^{m(t)} \bm1(x\in [x_i(t), x_i(t)+1/n]),
\end{align}
which is also $\rho_{t}(x)$ as in \eqref{e:defrhot+}. If we represent a particle at $x$ by the measure $\bm1([x, x+1/n])$, then we can view \eqref{e:dht} as the empirical measure of $m(t)$ particles $(x_1(t), x_2(t), \cdots, x_{m(t)}(t))$.
Since $\rho_{t}(x)$ satisfies \eqref{e:cond2}, $m(t)$, the number of particles at time $t$, is determined by
\begin{align*}
m(t)=n\left((\beta^{{\mathfrak P}}(\fb_1(t),t)-\beta^{{\mathfrak P}}(\fa_1(t),t))+\cdots +(\beta^{{\mathfrak P}}(\fb_{r(t)}(t),t)-,\beta^{{\mathfrak P}}(\fa_{r(t)}(t),t))\right).
\end{align*}

Similarly, for $0< t\leq T$, we can view $\del_{x}  h(x,t)$ as a positive measure on $[\fa_1(t-), \fb_1(t-)]\cup[\fa_2(t-), \fb_2(t-)]\cup\cdots \cup [\fa_{r(t-)}(t-),\fb_{r(t-)}(t-)]$. It is in the following form
\begin{align}\label{e:dht-}
\del_{x}  h(x,t)=\sum_{i=1}^{m(t-)} \bm1(x\in [x_i(t-), x_i(t-)+1/n]),
\end{align}
 where
 \begin{align*}
 m(t-)=n\left((\beta^{{\mathfrak P}}(\fb_1(t-),t)-\beta^{{\mathfrak P}}(\fa_1(t-),t))+\cdots +(\beta^{{\mathfrak P}}(\fb_{r(t-)}(t-),t)-,\beta^{{\mathfrak P}}(\fa_{r(t-)}(t-),t))\right).
 \end{align*}
Similarly, \eqref{e:dht-}  can be viewed as the empirical measure of $m(t-)$ particles $(x_1(t-), x_2(t-), \cdots, x_{m(t-)}(t-))$,

We remark if there are no horizontal boundary edges of ${\mathfrak P}$ at time $t$, i.e. $t\neq T_0, T_1,T_2,\cdots, T_p$, then \eqref{e:dht} and \eqref{e:dht-} are the same. 
For $t\in \{T_0, T_1,T_2,\cdots, T_p\}$, $(x_1(t-), x_2(t-), \cdots, x_{m(t-)}(t-))$ and $(x_1(t), x_2(t), \cdots, x_{m(t)}(t))$ also uniquely determine each other.
For any horizontal boundary edge $[\fa',\fb']$ of ${\mathfrak P}$ at time $t$, if it is a lower boundary edge of ${\mathfrak P}$, particles are created at $[\fa',\fb')\cap \bZ_n$;
 if it is an upper boundary edge of ${\mathfrak P}$, particles at $[\fa',\fb')\cap \bZ_n$ are annihilated. See Figure \ref{f:particle} for an illustration. In the rest of this paper, for any time $t\in(0,T)\cap \bZ_n$ and particle configuration $\bmx(t)=(x_1(t), x_2(t), \cdots, x_{m(t)}(t))$, we denote $\bmx(t-)=(x_1(t-), x_2(t-), \cdots, x_{m(t-)}(t-))$ the corresponding particle configuration at time $t-$.
\begin{figure}
 \includegraphics[scale=0.48,trim={0cm 14cm 0 7cm},clip]{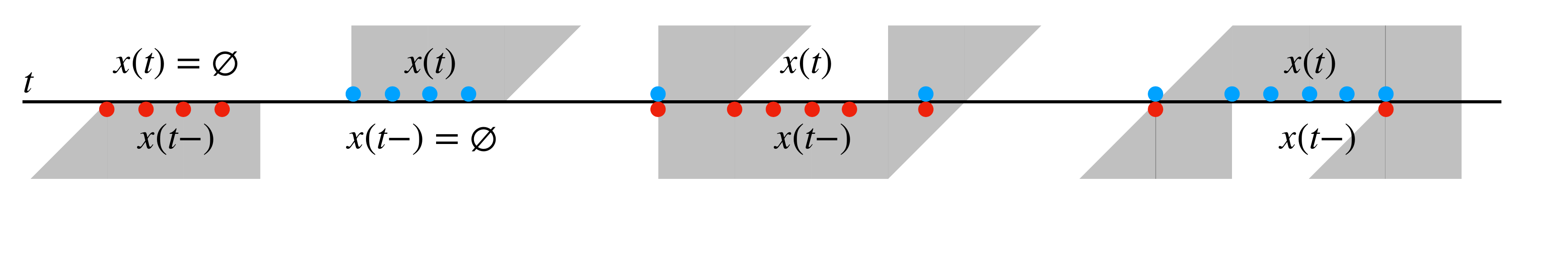}
 \caption{For any horizontal boundary edge $[\fa',\fb']$ of ${\mathfrak P}$ at time $t$, if it is a lower boundary edge of ${\mathfrak P}$, particles are created at $[\fa',\fb')\cap \bZ_n$;
 if it is an upper boundary edge of ${\mathfrak P}$, particles at $[\fa',\fb')\cap \bZ_n$ are annihilated. .}
 \label{f:particle}
 \end{figure}

\begin{definition}\label{def:Mt}
For any polygonal domain ${\mathfrak P}$, and any time $t\in [0,T)\cap \bZ_n$, we denote ${\mathfrak M}_{t}({\mathfrak P})$ the set of particle configurations
$(x_1(t)< x_2(t)<\cdots, x_{m(t)}(t))\in \bZ_n^{m(t)}$, such that its empirical measure 
\begin{align*}
\sum_{i=1}^{m(t)} \bm1\left(x\in \left[x_i(t), x_i(t)+1/n\right]\right)
\end{align*}
is a 
positive measure on
$[\fa_1(t), \fb_1(t)]\cup[\fa_2(t), \fb_2(t)]\cup\cdots \cup [\fa_{r(t)}(t),\fb_{r(t)}(t)]$, and satisfies \eqref{e:cond1} and \eqref{e:cond2}.
\end{definition}

%
%

%
%
%

%

In this way we have identified lozenge tilings of domain $\sfP$, with nonintersecting Bernoulli random walk $\{x_1(t), x_2(t), \cdots, x_{m(t)}(t)\}_{t\in[0,T]\cap \bZ_n}$ with $(x_1(t),x_2(t), \cdots, x_{m(t)}(t))\in \mathfrak M_t({\mathfrak P})$. In the rest of this paper, we restrict ourselves on the set of polygonal domains as defined below, which have exactly one horizontal upper boundary edge, and correspond to nonintersecting Bernoulli walks with tightly packed ending configurations.
\begin{definition}\label{def:oneend}
A polygonal domain $\fP$ has exactly one horizontal upper boundary edge, if
\begin{enumerate}
\item the boundary edge of $\fP$ at $t=T$ is a single interval, i.e. $\{x: (x,T-)\in \fP\}=[\fa(T-), \fb(T-)]$,
\item all the horizontal boundary edges of $\fP$ are lower boundary edges except at time $t=T$, i.e. $\{x: (x,t-)\in \fP\}\subset \{x: (x,t+)\in \fP\}$ for $t\in\{T_0, T_1,\cdots, T_{p-1}\}$.
\end{enumerate}
\end{definition}

We refer to Figure \ref{f:shapes} for some examples of domains described by Definition \ref{def:oneend}. Given a polygonal domain $\fP$ as in Definition \ref{def:oneend}, the corresponding nonintersecting Bernoulli walks all end on the interval $[\fa(T-), \fb(T-)]$ at time $t=T$. Especially no particles are annihilated at time $t\in [0,T)$, but might be created.
\begin{figure}
 \includegraphics[width=\textwidth,trim={0cm 8cm 0 8cm},clip]{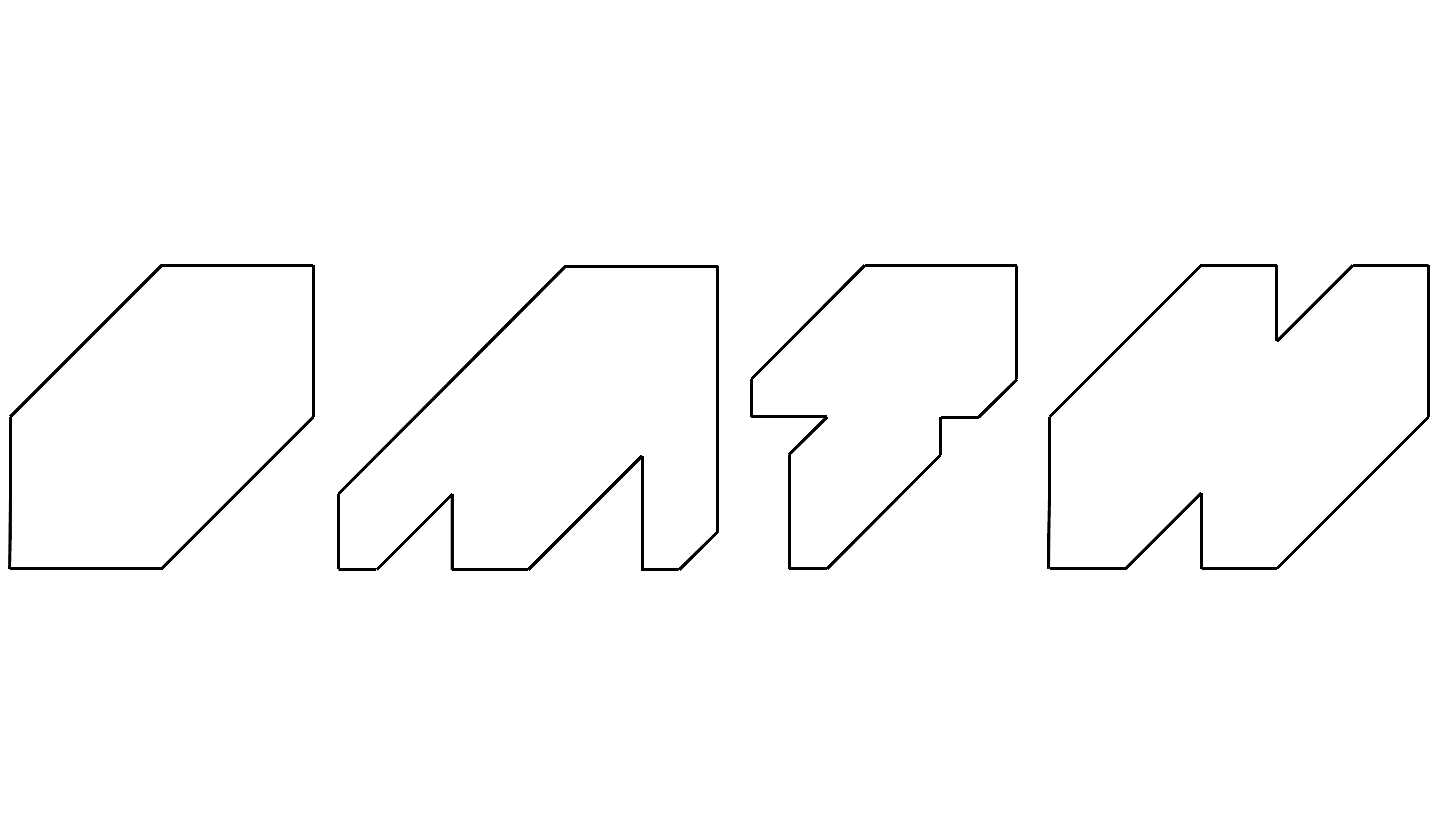}
 \caption{Some examples of domains described by Definition \ref{def:oneend} are hexagon ($1$-st subplot), trapezoidal domain ($2$-nd subplot), polygon with general left and right boundaries ($3$-rd subplot). The polygon in the $4$-th subplot does not satisfy Definition \eqref{def:oneend}, because after identifying with nonintersecting Bernoulli walks, the ending configuration consists of two components. }
 \label{f:shapes}
\end{figure}
 As we will see in Section \ref{s:loopeq}, for such domains, the corresponding discrete loop equations do not have singularity.

For any time $t\in [0,T)\cap \bZ_n$, 
and any particle configuration $(x_1,x_2,\cdots,x_{m})\in \mathfrak M_{t}({\mathfrak P})$, where
\begin{align*}
m=n\left((h(\fb_1(t),t)-h(\fa_1(t),t))
+\cdots +(h(\fb_{r(t)}(t),t)-h(\fa_{r(t)}(t),t))\right),
\end{align*}
we denote $N_{t}(x_1, x_2, \cdots, x_m)$ the number of nonintersecting Bernoulli walks starting from particle configuration $(x_1,x_2,\cdots,x_{m})$ at time $t$:
\begin{align}\label{e:defwalk}
N_{t}(x_1, x_2, \cdots, x_m)=\#\{\text{nonintersecting Bernoulli walks starting from $(x_1,x_2,\cdots,x_{m})$ at time $t$}\}.
\end{align}
And for $(x_1, x_2, \cdots, x_m)\not\in \mathfrak M_{t}({\mathfrak P})$, we simply set
$N_{t}(x_1, x_2, \cdots, x_m)=0$.
For any $t\in (0,T)\cap \bZ_n$, and any particle configuration $(x_1,x_2,\cdots,x_{m})\in \mathfrak M_{t}({\mathfrak P})$, we denote  $(x'_1,x'_2,\cdots,x'_{m'})$ the corresponding particle configuration at $t-$ and set
\begin{align}\label{e:defNt-}
N_{t-}(x'_1,x'_2,\cdots,x'_{m'})=N_{t}(x_1, x_2, \cdots, x_m).
\end{align}
Again for $t\neq\{T_0, T_1, T_2,\cdots, T_p\}$, \eqref{e:defNt-} reduces to 
$N_{t-}(x_1,x_2,\cdots,x_{m})=N_{t}(x_1, x_2, \cdots, x_m)$.
For $t=T$, we set $W_{t-}(x'_1,x'_2,\cdots,x'_{m'})=1$, if $(x'_1,x'_2,\cdots,x'_{m'})$ matches with the boundary of ${\mathfrak P}$, i.e. it equals
$(\fa(T-), \fa(T-)+1/n, \cdots, \fb(T-)-1/n)$.

We remark that if we denote the empirical measure of the particle configuration $\bmx=(x_1, x_2, \cdots, x_m)\in \fM_t(\fP)$ as
\begin{align}\label{e:defrho}
\rho(x; \bmx)=\sum_{i=1}^m \bm1\left(x\in \left[x_i, x_i+1/n\right]\right).
\end{align}
Then this particle configuration gives a height function on the bottom boundary $[\fa_1(t), \fb_1(t)]\cup[\fa_2(t), \fb_2(t)]\cup\cdots \cup [\fa_{r(t)}(t),\fb_{r(t)}(t)]$ of ${\mathfrak P}\cap \bR\times[t, T]$:
\begin{align}\label{e:defbeta0}
\beta(x;\bmx)= \left(\beta^{{\mathfrak P}}(\fa_i(t),  t )+\int_{\fa_i(t)}^ x  \rho( x;\bmx )\rd  x \right),
\end{align}
for any $ x \in [\fa_i(t), \fb_i(t)]$. From the point view of lozenge tilings,  $N_{t}(x_1, x_2, \cdots, x_m)$ (as in \eqref{e:defwalk}) is also the number of lozenge tilings of the domain $\sfP\cap \bR\times[nt, nT]$, with the bottom height function given by $nb(nx;\bmx)$.

%
%
%
%
%

We have the following recursion for the number of lozenge tilings: for any time $t\in [0,T)\cap \bZ_n$,
\begin{align}\label{e:Wt+recur}
N_{t}(x_1,x_2,\cdots,x_m)
=\sum_{(e_1,e_2,\cdots, e_m)\in\{0,1\}^m}
N_{(t+1/n)-}(x_1+e_1/n,x_2+e_2/n,\cdots,x_m+e_m/n).
\end{align}
We can use the partition function \eqref{e:defwalk} to define a nonintersecting Bernoulli random walk: Let  $\{\bmx(t)=(x_1(t), x_2(t),\cdots, x_{m(t)}(t))\}_{t\in[0,T]\cap \bZ_n}$ be the following nonintersecting Bernoulli random walk, with transition probability
\begin{align}\label{e:randomwalk}
\bP(\bmx((t+1/n)-)=\bmx+\bme/n|\bmx(t)=\bmx)=\frac{N_{(t+1)-}(\bmx+\bme/n)}{N_{t}(\bmx)},
\end{align}
for $\bme=(e_1, e_2, \cdots, e_m)\in \{0,1\}^m$.
It is easy to see that the joint law of $\{\bmx(t)=(x_1(t), x_2(t),\cdots, x_{m(t)}(t))\}_{{t\in[0,T]\cap \bZ_n}}$ is uniform among all possible nonintersecting Bernoulli walks, and we have the following proposition.

\begin{proposition}
The joint law of $\{\bmx(t)=(x_1(t), x_2(t),\cdots, x_{m(t)}(t))\}_{{t\in[0,T]\cap \bZ_n}}$, as defined in \eqref{e:randomwalk}, with initial data 
\begin{align*}
\bmx(0)=(\fa_1(0), \fa_1(0)+1/n, \cdots, \fb_1(0)-1/n, \cdots, \fa_{r(0)}(0), \fa_{r(0)}(0)+1/n, \cdots, \fb_{r(0)}(0)-1/n),
\end{align*}
is the same as  random lozenge tilings of domain $\sfP$, which is obtained from rescaling $\fP$ by a factor $n$.
\end{proposition}


\section{Complex Burger Equation}\label{s:burger}
Following \cite{MR2358053}, we can encode the local density triple
$(p_{\up{scale=0.8}}, p_{\rig{scale=0.8}}, p_{\emp{scale=0.8}})$  by the so called complex slope $ f\in \bC_-$ in the lower half plane, as in Figure \ref{f:z}. The complex slope $ f\in \bC_-$ (in the lower half complex plane) is uniquely determined by 
\begin{align*}
\arg^*( f)=\pi(p_{\emp{scale=0.8}}-1),\quad \arg^*( f+1)=-\pi p_{\rig{scale=0.8}},
\end{align*}
where we use the convention that $\arg^*(\cdot)\in[-\pi,0]$.
We recall that in our convention  $\del_{x}  h(x,t)=p_{\up{scale=0.8}}+p_{\rig{scale=0.8}}=1-p_{\emp{scale=0.8}}$ and $\del_{t}  h(x,t)=-p_{\rig{scale=0.8}}$.
(Our convention of the height function is slightly different from that in \cite{MR2358053}, that is why we need to take the complex slope in the lower half plane.)
Then the surface tension functional \eqref{e:surfacetension} satisfies
\begin{align*}
\frac{\del \sigma}{\del (\del_{x} h)}=-\frac{\del \sigma}{\del p_{\emp{scale=0.8}}}=\ln| f+1|,\quad
\frac{\del \sigma}{\del (\del_{t} h)}=-\frac{\del \sigma}{\del p_{\rig{scale=0.8}}}=\ln | f|.
\end{align*}
%
%
\begin{figure}
 \includegraphics[scale=0.48,trim={0cm 8cm 0 8cm},clip]{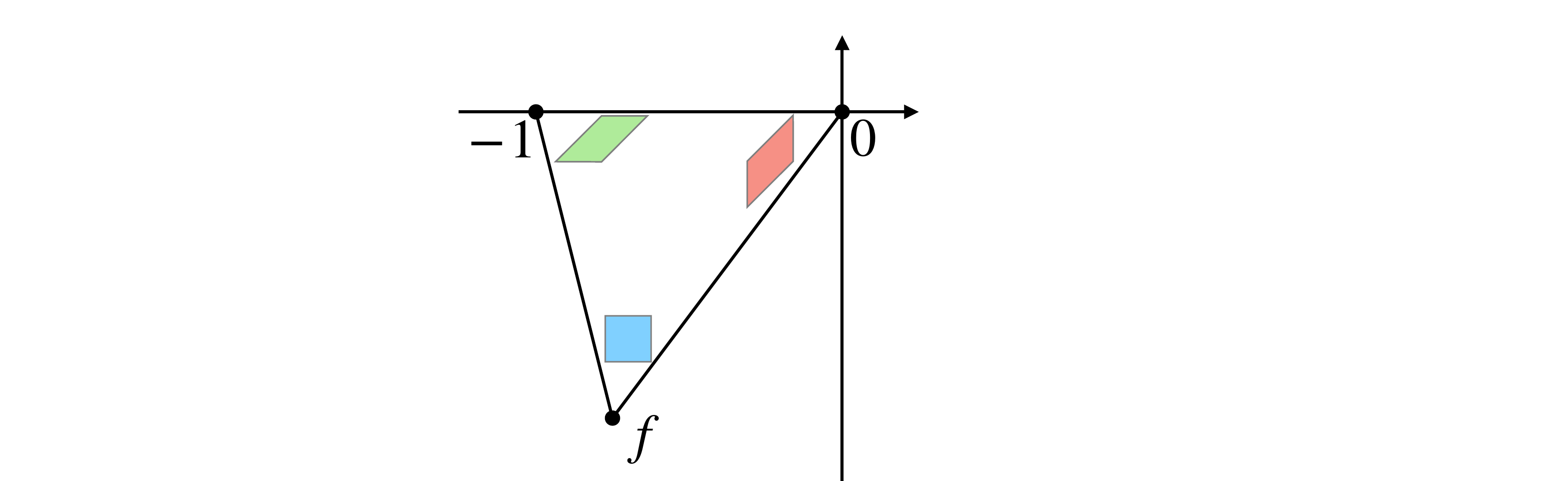}
 \caption{Triangle defining the complex slope $f$.}
 \label{f:z}
\end{figure}

We denote the liquid region,  i.e. where
$p_{\up{scale=0.8}}, p_{\rig{scale=0.8}}, p_{\emp{scale=0.8}}> 0$,
\begin{align}\label{e:liquid}
\Omega({\mathfrak P})\deq\{(x,t)\in {\mathfrak P}: 0<\del_{x}  h^*(x,t)<1\}.
\end{align}
The remaining part ${\mathfrak P}\setminus\Omega({\mathfrak P})$ is called the frozen region, where only one type of lozenges is observed. The boundary curves separate liquid region and frozen region are referred to as ``arctic curves".

\begin{theorem}[\cite{MR2358053}]
Let $ h^*$ be the maximizer of the variational problem \ref{e:vp}. In the liquid region $\Omega({\mathfrak P})$, there exists a function $ f_t(x)$ taking values in the 
lower half complex plane such that
\begin{align}\label{e:hfrelation}
(\del_{x}  h^*, \del_{t}  h^*)=\frac{1}{\pi}\left(-\arg^*( f_t(x)),\arg^*( f_t(x)+1)\right),
\end{align}
and $ f_t(x)$ satisfies the complex Burgers equation
\begin{align}\label{e:cb}
 \frac{\del_{t}  f_t(x)}{ f_t(x)}+\frac{\del_{x}  f_t(x)}{ f_t(x)+1}=0.
\end{align}
\end{theorem}
\begin{proof}
The Euler-Lagrange equations for the variational problem \eqref{e:vp} give
\begin{align*}
\frac{\del}{\del x}\frac{\del \sigma}{\del (\del_{x} h^*)}+\frac{\del}{\del t}\frac{\del \sigma}{\del (\del_{t} h^*)}=\del_{t} \ln | f|+\del_{x} \ln| f+1|=0.
\end{align*} 
Moreover, we also have that $\arg^*( f)=\pi(p_{\emp{scale=0.8}}-1)$ and $\arg^*( f+1)=-\pi p_{\rig{scale=0.8}}$. Then it follows that 
\begin{align*}
\pi\del_{tx} h^*=-\del_{x}\arg^*( f)= \del_{t} \arg^*( f+1).
\end{align*}
Thus we have 
\begin{align*}
\del_{t} \ln  f+\del_{x} \ln( f+1)=0,
\end{align*}
and \eqref{e:cb} follows.
\end{proof}

In physics
literature, the complex Burgers equation first appeared in the
work of Matytsin \cite{MR1257846}, to
describe the asymptotics of Harish-Chandra-Itzykson-Zuber integral formula. Its rigorous mathematical study appears in the work
of  Guionnet \cite{MR2034487}.

The complex Burgers equation can be solved using characteristic method as in  \cite[Corollary 1]{MR2358053}, there exists an analytic function $Q$ of two variables such that the solution of \eqref{e:cb} satisfies
\begin{align}\label{e:defQ}
Q\left( f_t(x), x- \frac{t  f_t(x)}{ f_t(x)+1}\right)=0.
\end{align}
When ${\mathfrak P}$ is  a polygonal domain in the plane, formed by $3d$ segments
with slopes $1,0,\infty$, cyclically repeated as we follow the boundary in the
counterclockwise direction, it is proven in \cite[Theorem 2]{MR2358053} that $Q$ is an algebraic curve of degree $d$ with genus zero. In this case, at the boundary of the liquid region, the so-called ``arctic curve", the solutions $( f,z)$ and $(\bar f,\bar z)$ of $Q=0$ become identical. Thus, the arctic curve corresponds to points where the equation \eqref{e:defQ}
has a double root, and is itself an analytic curve in the $(x,t)$-plane.

At a generic point, the arctic curve is smooth and \eqref{e:defQ} has exactly
one real double root. The function $ f_t(x)$ has a square-root
singularity at the arctic curve. Hence, $ f$ grows
like the square root of the distance as one moves from the arctic curve
inside the liquid region. At $(x,t)$ of the arctic curve, its tangent vector has slope 
\begin{align}\label{e:slope}
\frac{ f_t(x)+1}{ f_t(x)},
\end{align}
and locally, the liquid region lies completely on one side of its tangent lines. As we move along arctic curve counterclockwise, its tangent vector also rotates counterclockwise, and should rotate $d$ times in total. Moreover the arctic curve is tangent to all sides of the polygonal domain ${\mathfrak P}$ that we tile. We note that near a non-convex
corner, the arctic curve may be tangent to the line forming a side of ${\mathfrak P}$ but
not to the side itself. Those tangent points give the corresponding values of $ f_t(x)$, this can be used to fix $Q$.

Inside the liquid region $\Omega({\mathfrak P})$ as defined in \eqref{e:liquid}, we have $0<\del_{x} h^*(x,t)<1$. From \eqref{e:hfrelation}, it follows that $\Im[ f_t(x)]<0$ inside the liquid region. We can recover $ (x,t)$ from $( f_t(x), x-t f_t(x)/( f_t(x)+1))$, by noticing that $t=-\Im[x-t f_t(x)/( f_t(x)+1)]/\Im[ f_t(x)/( f_t(x)+1)]$.
Thus, the map 
\begin{align}\label{e:cover}
\pi_Q:  (x,t)\in \Omega({\mathfrak P})\mapsto ( f_t(x),x-t f_t(x)/( f_t(x)+1))
\end{align}
is an orientation preserving diffeomorphism from $\Omega({\mathfrak P})$ onto its image.
In fact, we can glue
$ f_t(x)$ and its complex conjugate
along the boundary of the liquid region $\Omega({\mathfrak P})$. This gives a map from the double of
$\overline {\Omega({\mathfrak P})}$ to the curve $Q( f,z)=0$. This map is orientation preserving and
unramified, and so a diffeomorphism of $Q( f,z)=0$. Especially, $Q( f,z)=0$ is of the same genus as the liquid region, both have genus zero. 

\begin{figure}
 \includegraphics[scale=0.48,trim={0cm 6cm 0 8cm},clip]{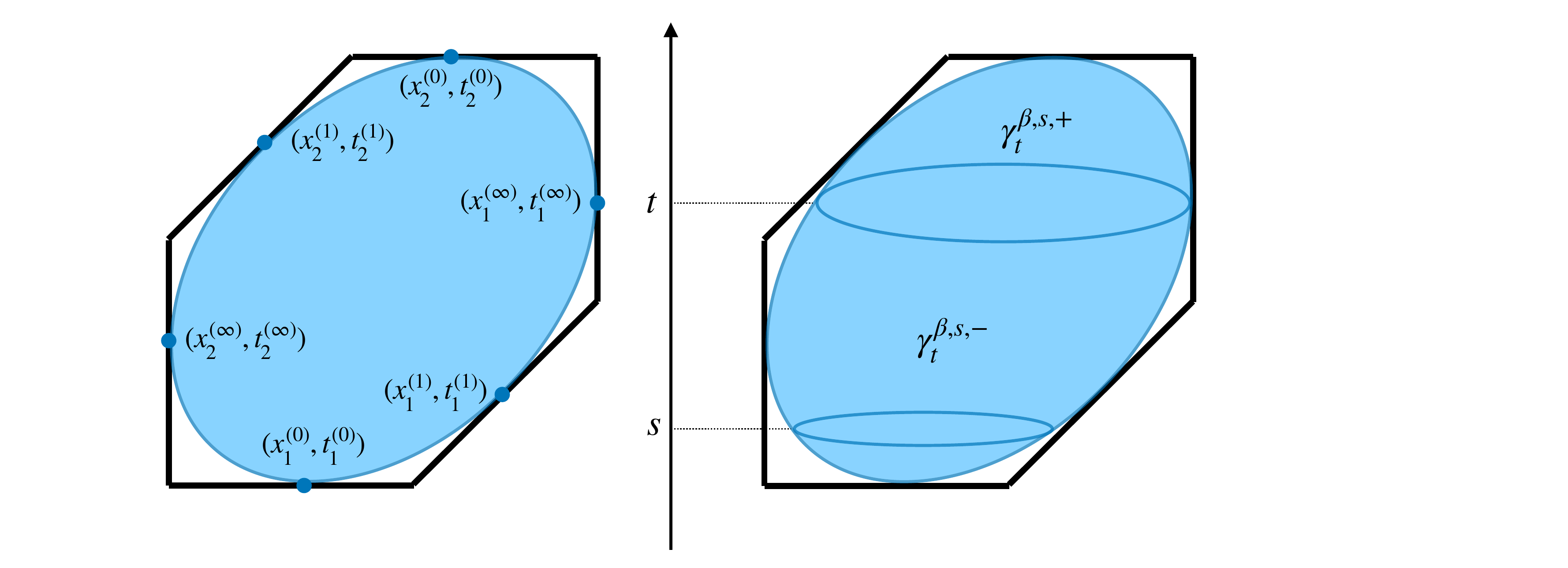}
 \caption{In the first subplot, the arctic curve is tangent to boundary of ${\mathfrak P}$  in counterclockwise order at $(x_1^{(0)}, t_1^{(0)}), (x_1^{(1)}, t_1^{(1)}), (x_1^{(\infty)}, t_1^{(\infty)}), (x_2^{(0)}, t_2^{(0)}), (x_2^{(1)}, t_2^{(1)}), (x_2^{(\infty)}, t_2^{(\infty)})$; In the second subplot, the image of the segment $\{ (x,t): (x,t)\in\Omega({\mathfrak P};\beta,s)\}$ under the map $\pi^t_{Q^{ \beta , s }}$, and its complex conjugate  cuts the Riemann surface $\gamma_t^{ \beta , s }$ into two parts:  $\gamma_t^{ \beta , s  ,-}, \gamma_t^{ \beta , s  ,+}$. }
 \label{f:circle}
\end{figure}

As we move along the boundary of ${\mathfrak P}$ counterclockwise, the boundary of the liquid region is tangent to the boundary of ${\mathfrak P}$ at $(x_1^{(0)}, t_1^{(0)}), (x_1^{(1)}, t_1^{(1)}), (x_1^{(\infty)}, t_1^{(\infty)}), \cdots, (x_d^{(0)}, t_d^{(0)}), (x_d^{(1)}, t_d^{(1)}), (x_d^{(\infty)}, t_d^{(\infty)})$, sequentially, where the points $(x_i^{(0)}, t_i^{(0)}), (x_i^{(1)}, t_i^{(1)}), (x_i^{(\infty)}, t_i^{(\infty)})$ are tangent points on boundary edges of ${\mathfrak P}$ with slopes $0, 1, \infty$ respectively,  as in Figure \ref{f:circle}.
For $1\leq i\leq d$, using the relation \eqref{e:slope}
\begin{align}\label{e:bcond}
 f_{t_i^{(0)}}(x_i^{(0)})=-1,\quad
 f_{t_i^{(1)}}(x_i^{(1)})=\infty,\quad
 f_{t_i^{(\infty)}}(x_i^{(\infty)})=0.
\end{align}
 Using the map \eqref{e:cover}, those tangent points correspond to points on $Q$, for $1\leq i\leq d$
  \begin{align}\label{e:x0}
 X_i^{(0)}=( f_{t_i^{(0)}}(x_i^{(0)}), x_i^{(0)}-t_i^{(0)} f_{t_i^{(0)}}(x_i^{(0)})/( f_{t_i^{(0)}}(x_i^{(0)})+1))
 =
 \left\{
 \begin{array}{cc}
 (-1, x_i^{(0)}), &t_i^{(0)}=0,\\
  (-1, \infty), &t_i^{(0)}>0,\\
 \end{array}
 \right.
 \end{align}
  and 
  \begin{align}\begin{split}\label{e:x1inf}
 &X_i^{(1)}=( f_{t_i^{(1)}}(x_i^{(1)}), x_i^{(1)}-t_i^{(1)} f_{t_i^{(1)}}(x_i^{(1)})/( f_{t_i^{(1)}}(x_i^{(1)})+1))
 =(\infty, x_i^{(1)}-t_i^{(1)}),\\
  &X_i^{(\infty)}=( f_{t_i^{(\infty)}}(x_i^{(\infty)}), x_i^{(\infty)}-t_i^{(\infty)} f_{t_i^{(\infty)}}(x_i^{(\infty)})/( f_{t_i^{(\infty)}}(x_i^{(\infty)})+1))
 =(0, x_i^{(\infty)}).
 \end{split}\end{align}
We notice that  $(x_i^{(1)},t_i^{(1)})$ belongs to an edge of ${\mathfrak P}$ with slope $1$, the difference $x_i^{(1)}-t_i^{(1)}$ depends only on that edge. Similarly $(x_i^{(\infty)},t_i^{(\infty)})$ belongs to an edge of ${\mathfrak P}$ with slope $\infty$, $x_i^{(\infty)}$ depends only on that edge. Therefore, we can directly compute the coordinates of those points $X_1^{(1)}, X_1^{(\infty)}, \cdots, X_d^{(1)}, X_d^{(\infty)}$ from the polygonal domain ${\mathfrak P}$. Moreover, these relations \eqref{e:x1inf} give local behavior of $Q$ at points $X_1^{(1)}, X_1^{(\infty)},X_2^{(1)}, X_2^{(\infty)},\cdots, X_d^{(1)}, X_d^{(\infty)}$. The algebraic curve $Q( f,z)=0$ behaves like:
 $ f$ has a pole at $( f,z)=X_i^{(1)}$, for $1\leq i\leq d$;
$ f$ has a zero at $( f,z)=X_i^{(\infty)}$, for $1\leq i\leq d$;
and $\sf$ does not have other zeros or poles.

\begin{remark}
The prescribed $d$ zeros and $d$ poles, given by the above conditions, uniquely determines the degree $d$ algebraic curve $Q( f,z)=0$. However, it is not easy to directly write down the expression of $ f$.
\end{remark}

With the algebraic curve $Q( f,z)=0$, we can extend the function \eqref{e:hfrelation} $ f_t(x)$ from $\{ x :  (x,t)\in\Omega({\mathfrak P})\}$ to a meromorphic function on certain algebraic curve.
At time $ t $, the equation $\{( f,z)\in (\bC\cup\infty)^2:  Q( f, z- t  f/( f+1))=0\}$ defines an algebraic curve $\gamma_t$. We notice $\gamma_{0}$ is simply the algebraic curve $Q( f,z)=0$, and $\gamma_t$ can  be obtained from $\gamma_0$ by the characteristic flow:
\begin{align}\label{e:flow}
( f,z)\in \gamma_{0}\mapsto ( f, z+t f/( f+1))\in \gamma_t.
\end{align}
As a consequence, the family of algebraic curves $\{\gamma_t\}_{0\leq t\leq T}$ are homeomorphic to each other.
Comparing with \eqref{e:defQ}, this gives a natural extension for $ f_t(x)$ from $\{ x :  (x,t)\in\Omega({\mathfrak P})\}$ to the algebraic curve $\gamma_t$. In this way $ f_t: Z=( f,z)\in \gamma_t\mapsto  f\in \bC\bP^1$. If the context is clear we will simply write $ f_t(Z)= f_t(( f,z))$ as $ f_t(z)$ for $Z=( f,z)\in \gamma_t$.


\subsection{More General Boundary Condition}\label{s:generalb}

Slightly more general than \eqref{e:vp}, we can consider the variational problem 
on 
\begin{align}\label{e:omegas}
{\mathfrak P}^{ s }={\mathfrak P}\cap\{(x,t):  s\leq t\leq T \},
\end{align}
with general boundary condition at time $ s $. As in \eqref{e:lb}, we denote the  lower boundary of ${\mathfrak P}^{ s }$ as
\begin{align}\label{e:lbcopy}
\{ x :  (x,s+)\in {\mathfrak P}\}=[\fa_1(s), \fb_1(s)]\cup[\fa_2(s), \fb_2(s)]\cup\cdots \cup [\fa_{r(s)}(s),\fb_{r(s)}(s)].
\end{align}
We further denote  $ \beta : [\fa_1(s), \fb_1(s)]\cup[\fa_2(s), \fb_2(s)]\cup\cdots \cup [\fa_{r(s)}(s),\fb_{r(s)}(s)]\mapsto \bR$ any nondecreasing Lipschitz function with Lipschitz constant $1$: $0\leq \del_{x}  \beta ( x )\leq 1$.
Moreover on the boundary of ${\mathfrak P}$, $ \beta $ coincides with the height of the boundary of ${\mathfrak P}$, i.e. $ \beta ( x )=\beta^{{\mathfrak P}}(x , s )$ for $(x , s )\in \del {\mathfrak P}$.
We denote $H({\mathfrak P};  \beta, s)$ the set of height functions over ${\mathfrak P}^{ s }$ with the height at time $ s $ given by $ \beta $.

\begin{definition}\label{def:Hpb}
For any polygonal domain ${\mathfrak P}$, we denote by $H({\mathfrak P}; \beta, s)$ the set of functions
$ h:{\mathfrak P}^s\mapsto \bR$,
such that:
\begin{enumerate}
\item 
$ h(x,t)=\beta^{{\mathfrak P}} (x,t)$ on $\del{\mathfrak P}\cap\{(x,t):  s<t\leq T \}$ and 
$ h( x , s )= \beta(x)$ on $\{ x : ( x , s+ )\in {\mathfrak P}\}$.
\item
$ h$ is Lipschitz and satisfies \eqref{e:lip} at  points where it is differentiable.
\end{enumerate}
\end{definition}

In this section, we consider the following variational problem 
on ${\mathfrak P}^{ s }$, with general boundary condition at time $ s $ given by $ \beta $,
\begin{align}\label{e:varWt}
W_ s ( \beta )=\sup_{ h\in H({\mathfrak P};\beta, s)}\int_ s ^T\int_\bR\sigma(\nabla  h)\rd x \rd t .
\end{align}

The same as \eqref{e:hfrelation},
if we denote the minimizer of \eqref{e:varWt} as $ h_t^*(x;\beta,s)$ and the corresponding complex slope $ f_t( x ;\beta, s)$
\begin{align}\label{e:ft2}
(\del_{x}  h^*, \del_{t}  h^*)=\frac{1}{\pi}\left(-\arg^*( f_t( x ;\beta, s)),\arg^*( f_t( x ;\beta, s)+1)\right),
\end{align}
then $ f_t( x ;\beta, s)$ satisfies the complex Burgers equation
\begin{align}\label{e:burgereq22}
 \frac{\del_{t}  f_t( x ;\beta, s)}{ f_t( x ;\beta, s)}+\frac{\del_{x}  f_t( x ;\beta, s)}{ f_t( x ;\beta, s)+1}=0,
\end{align}
for $(x,t)\in\Omega({\mathfrak P};\beta,s)$ in the liquid region:
\begin{align}\label{e:defliquid}
\Omega({\mathfrak P};\beta,s)\deq \{ (x,t): 0<\del_{x}  h^*<1, s<t<T\}.
\end{align}
There exists an analytic function $Q^{ \beta , s }$ of two variables such that the solution of \eqref{e:burgereq22} satisfies
\begin{align}\label{e:defQatt}
Q^{ \beta , s }\left( f_t( x ;\beta, s),  x - \frac{ t   f_t( x ;\beta, s)}{ f_t( x ;\beta, s)+1}\right)=0.
\end{align}
 With the analytic function $Q^{ \beta , s }$, we can define the Riemann surface:
$\gamma^{ \beta , s }_ t =\{( f,z)\in (\bC\cup\infty)^2: Q^{ \beta , s }( f,z- t  f/( f+1))=0\}$. Similarly to \eqref{e:flow}, 
 for $ s \leq  t \leq T$, $\gamma^{ \beta , s }_ t $ can  be obtained from $\gamma^{ \beta , s }_ s $ by the characteristic flow:
\begin{align}\label{e:flowatt}
Z=( f,z)\in \gamma^{ \beta , s }_ s \mapsto Z_t=( f, z+( t - s ) f/( f+1))\in \gamma^{ \beta , s }_ t .
\end{align}
As a consequence, the family of algebraic curves $\{\gamma_t^{ \beta , s }\}_{s\leq t\leq T}$ are homeomorphic to each other.
Comparing  \eqref{e:defQatt} with \eqref{e:flowatt}, this gives a natural extension for $ f_t(x;\beta,s)$ from $\{ x :  (x,t)\in\Omega({\mathfrak P};\beta,s)\}$ to the Riemann surface $\gamma_t^{ \beta , s }$. In this way 
\begin{align}\label{e:extend}
f_t(\cdot; \beta,s): Z=( f,z)\in \gamma^{ \beta , s }_t\mapsto  f\in \bC\bP^1.
\end{align} 
If the context is clear we will simply write $ f_t(Z;\beta,s)$ as $ f_t(z;\beta,s)$ for $Z=( f,z)\in \gamma_t^{ \beta , s }$. The complex Burgers equation \eqref{e:burgereq22} can be extended to $\gamma^{ \beta , s }_t$,
\begin{align}\label{e:burgereq3}
 \frac{\del_{t}  f_t( Z ;\beta, s)}{ f_t( Z ;\beta, s)}+\frac{\del_{z}  f_t( Z ;\beta, s)}{ f_t( Z ;\beta, s)+1}=0,\quad Z\in \gamma^{ \beta , s }_t.
\end{align}
Using the characteristic flow \eqref{e:flowatt}, we can rewrite above complex Burgers equation in the following simple form
\begin{align}\label{e:burgernew}
\del_tf_t(Z_t;\beta,s)=0,\quad \del_t z(Z_t)=\frac{f_t(Z_t;\beta,s)}{f_t(Z_t;\beta,s)+1}.
\end{align}

Using the map 
\begin{align}\label{e:cover2}
\pi_{Q^{ \beta , s }}: ( x , r)\in\Omega({\mathfrak P};\beta,s) \mapsto (f_r(x;\beta, s),x-(r- s )f_r(x; \beta, s)/(f_r(x; \beta, s)+1))\in \gamma_ s ^{ \beta , s },
\end{align}
we can identify $\gamma_ s ^{ \beta , s }$ with two copies of $\Omega({\mathfrak P};\beta,s) $ gluing along the boundary of $\Omega({\mathfrak P};\beta,s)$.
Moreover, using the characteristic flow \eqref{e:flowatt}, we can also map $\Omega({\mathfrak P};\beta,s)$ to $\gamma_t^{ \beta , s }$, for any $ s \leq  t \leq T$:
\begin{align}\label{e:coveratr}
\pi^t_{Q^{ \beta , s }}: ( x , r)\in \Omega({\mathfrak P};\beta,s)\mapsto ( f_ r( x ; \beta, s), x -( r- t )f_ r(x;\beta, s))\in \gamma_t^{ \beta , s },
\end{align}
and identify $\gamma_t^{ \beta , s }$ as the gluing of two copies of $ \Omega({\mathfrak P};\beta,s)$ along the boundary.

Using the above identification, the points $(f,z)\in \gamma_t^{\beta,s}$ with $z$ real, i.e. $z\in \bR\bP^1$, correspond to the boundary of $\Omega(\fP;\beta,s)$ where $f_t(x;\beta,s)\in \bR\bP^1$, and the horizontal segments $\{(x,t):(x,t)\in \Omega(\fP;\beta,s) \}$. The image of $\{(x,t):(x,t)\in \Omega(\fP;\beta,s) \}$ under the map \eqref{e:coveratr} $\{( f_t(x;\beta,s),x): (x,t)\in\Omega(\fP;\beta,s)\}$ and its 
complex conjugate $ \{(\overline{ f_t(x;\beta,s)},x): (x,t)\in\Omega(\fP;\beta,s)\}$ glue together to cycles on the Riemann surface $\gamma_t^{\beta,s}$. They
 cut the Riemann surface $\gamma_t^{\beta,s}$ into two parts:  $\gamma_t^{\beta,s,-}, \gamma_t^{\beta,s,+}$. $\gamma_t^{ \beta , s  ,-}$ corresponds to $\{( x , r): ( x , r)\in\Omega({\mathfrak P};\beta,s),  s \leq  r\leq  t \}$, and $\gamma_t^{ \beta , s  ,+}$ corresponds to $\{( x , r): ( x , r)\in\Omega({\mathfrak P};\beta,s),  t \leq  r\leq T\}$, as shown in Figure \ref{f:circle}.
If we consider the covering map, 
\begin{align}\label{e:zmatr}
( f,z)\in \gamma^{ \beta , s,+}_ t \mapsto z\in \bC\bP^1,
\end{align}
it maps the cut $\{(  f_t( x ;\beta, s), x ):  (x,t)\in\Omega({\mathfrak P};\beta,s)\}\cup \{(\overline{  f_t( x ;\beta, s)}, x ):  (x,t)\in\Omega({\mathfrak P};\beta,s)\}$ to the segment $\{ (x,t): (x,t)\in\Omega({\mathfrak P};\beta,s)\}$ twice. We remark that in the special case when $t=s$, we have $\gamma_s^{\beta, s,+}=\gamma_s^{\beta,s}$, and $\gamma_s^{\beta, s,-}=\emptyset$.

Fix any $s\leq t\leq T$, it follows from the same discussion as for \eqref{e:slope},  at a generic point $(x,r)$ on the boundary of the liquid region in time interval $[t,T]$,
\begin{align*}
\Omega^t({\mathfrak P};\beta,s)\deq \{ (x,r): (x,r)\in\Omega({\mathfrak P};\beta,s), t\leq r\leq T\}, 
\end{align*}
its tangent vector has slope 
\begin{align}\label{e:slope2}
\frac{ f_r(x;\beta,s)+1}{ f_r(x;\beta,s)}.
\end{align}
As we move along the boundary of $\Omega^t({\mathfrak P};\beta,s)$ counterclockwise, its tangent vector also rotates counterclockwise. Moreover the boundary of $\Omega^t({\mathfrak P};\beta,s)$ is tangent to all sides of ${\mathfrak P}^t$ (as in \eqref{e:omegas}) that we tile, except possibly for edges leaving the bottom of ${\mathfrak P}^t$, which have slopes $1,\infty$. 
We denote $(x_1^{(0)}, t_1^{(0)}), (x_2^{(0)}, t_2^{(0)}), \cdots, (x_{d_0}^{(0)}, t_{d_0}^{(0)})$ the tangent points on boundary edges of ${\mathfrak P}^t$ with slope $0$. The number $d_0$ is the number of horizontal edges of ${\mathfrak P}$ above $t$. We remark it depends only on the polygonal domain ${\mathfrak P}$ and the time $t$, and is independent of the bottom height function $\beta$. Then from \eqref{e:slope2},
for $1\leq i\leq d_0$,
$
 f_{t_i^{(0)}}(x_i^{(0)})=-1.
$
Those tangent points are mapped by \eqref{e:coveratr} to points on $\gamma_t^{\beta,s,+}$, for $1\leq i\leq d_0$
  \begin{align}\label{e:x02}
 X_i^{(0)}=\left( f_{t_i^{(0)}}(x_i^{(0)}), x_i^{(0)}-(t_i^{(0)}-t)\frac{ f_{t_i^{(0)}}(x_i^{(0)})}{( f_{t_i^{(0)}}(x_i^{(0)})+1)}\right)
 =
  (-1, \infty).
 \end{align}
It follows that  the covering map \eqref{e:zmatr}
hits $\infty$ only at $\{X_i^{(0)}: 1\leq i\leq d_0\}$, and it has degree $d_0$, which depends only on the polygonal domain ${\mathfrak P}$.

The boundary of $\Omega^t({\mathfrak P};\beta,s)$ is also tangent to all sides of ${\mathfrak P}^t$ with slopes $1,\infty$, which are above $t$. We denote them $(x_1^{(1)}, t_1^{(1)}), (x_2^{(1)}, t_2^{(1)}), \cdots, (x_{d_1}^{(1)}, t_{d_1}^{(1)})$
and
$(x_1^{(\infty)}, t_1^{(\infty)}), (x_2^{(\infty)}, t_2^{(\infty)}), \cdots, (x_{d_\infty}^{(\infty)}, t_{d_\infty}^{(\infty)})$, where $d_1, d_\infty$ are the number of sides of ${\mathfrak P}^t$ with slopes $1,\infty$ above $t$.
Then from \eqref{e:slope2},
for $1\leq i\leq d_1$,
$
 f_{t_i^{(1)}}(x_i^{(1)})=\infty,
$
and 
for $1\leq i\leq d_\infty$,
$
 f_{t_i^{(\infty)}}(x_i^{(\infty)})=0.
$
Those tangent points are mapped by \eqref{e:coveratr} to points on $\gamma_t^{\beta,s,+}$, for $1\leq i\leq d_1$
  \begin{align}\label{e:x0d1}
 X_i^{(1)}=\left( f_{t_i^{(1)}}(x_i^{(1)}), x_i^{(1)}-(t_i^{(1)}-t)\frac{ f_{t_i^{(1)}}(x_i^{(1)})}{( f_{t_i^{(1)}}(x_i^{(1)})+1)}\right)
 =
  (\infty, x_i^{(1)}-(t_i^{(1)}-t)),
 \end{align}
and for $1\leq i\leq d_\infty$, let 
   \begin{align}\label{e:x0dinfty}
 X_i^{(\infty)}=\left( f_{t_i^{(\infty)}}(x_i^{(\infty)}), x_i^{(\infty)}-(t_i^{(\infty)}-t)\frac{ f_{t_i^{(\infty)}}(x_i^{(\infty)})}{( f_{t_i^{(\infty)}}(x_i^{(\infty)})+1)}\right)
 =
  (0, x_i^{(\infty)}).
 \end{align}
 We notice that $(x_i^{(1)}, t_i^{(1)})$ belongs to an edge of ${\mathfrak P}$ with slope $1$, the difference $x_i^{(1)}-(t_i^{(1)}-t)$ depends only on that 
edge; similarly $(x_i^{(\infty)}, t_i^{(\infty)})$ belongs to an edge of ${\mathfrak P}$ with slope $\infty$,  $x_i^{(\infty)}$ depends only on that 
edge. Therefore $X_i^{(1)}, X_i^{(\infty)}$ depend only on the polygonal domain ${\mathfrak P}$.

For the edges leaving the bottom of ${\mathfrak P}^t$, the boundary of the liquid region $\Omega^t({\mathfrak P};\beta,s)$ may or may not be tangent to them, depending on the height function $\beta$ of the bottom boundary, as shown in Figure \ref{f:lbcase}. 
For example, in the first subplot of Figure \ref{f:lbcase}, the frozen region consists of lozenges of shape $\emp{scale=1}$. Then on the boundary of the liquid region, the complex slope $f$ is nonnegative. Using \eqref{e:slope2}, we conclude that the slope of the tangent vector is between $1$ and $\infty$. Therefore, if we move along the boundary of $\Omega^t({\mathfrak P};\beta,s)$ clockwise, the slope will decrease until it hits $1$, that is where the boundary edge of ${\mathfrak P}^t$ and the boundary of the liquid region $\Omega^t({\mathfrak P};\beta,s)$ (the shadow) are tangent to each other. In the sixth subplot of Figure \ref{f:lbcase}, the frozen region consists of lozengs of shape $\rig{scale=1}$, so on the boundary of the liquid region, the complex slope $f$ is less or equal than $-1$. Using \eqref{e:slope2}, we conclude that the slope of the tangent vector is between $0$ and $1$. Therefore, if we move along the boundary of $\Omega^t({\mathfrak P}; \beta,s)$ clockwise, the slope will keep decreasing, and will not hit $1$. In this case, the boundary edge of ${\mathfrak P}^t$ and the boundary of the liquid region $\Omega^t({\mathfrak P};\beta,s)$ are not tangent to each other. 

\begin{figure}
 \includegraphics[scale=0.48,trim={0cm 6cm 0 5cm},clip]{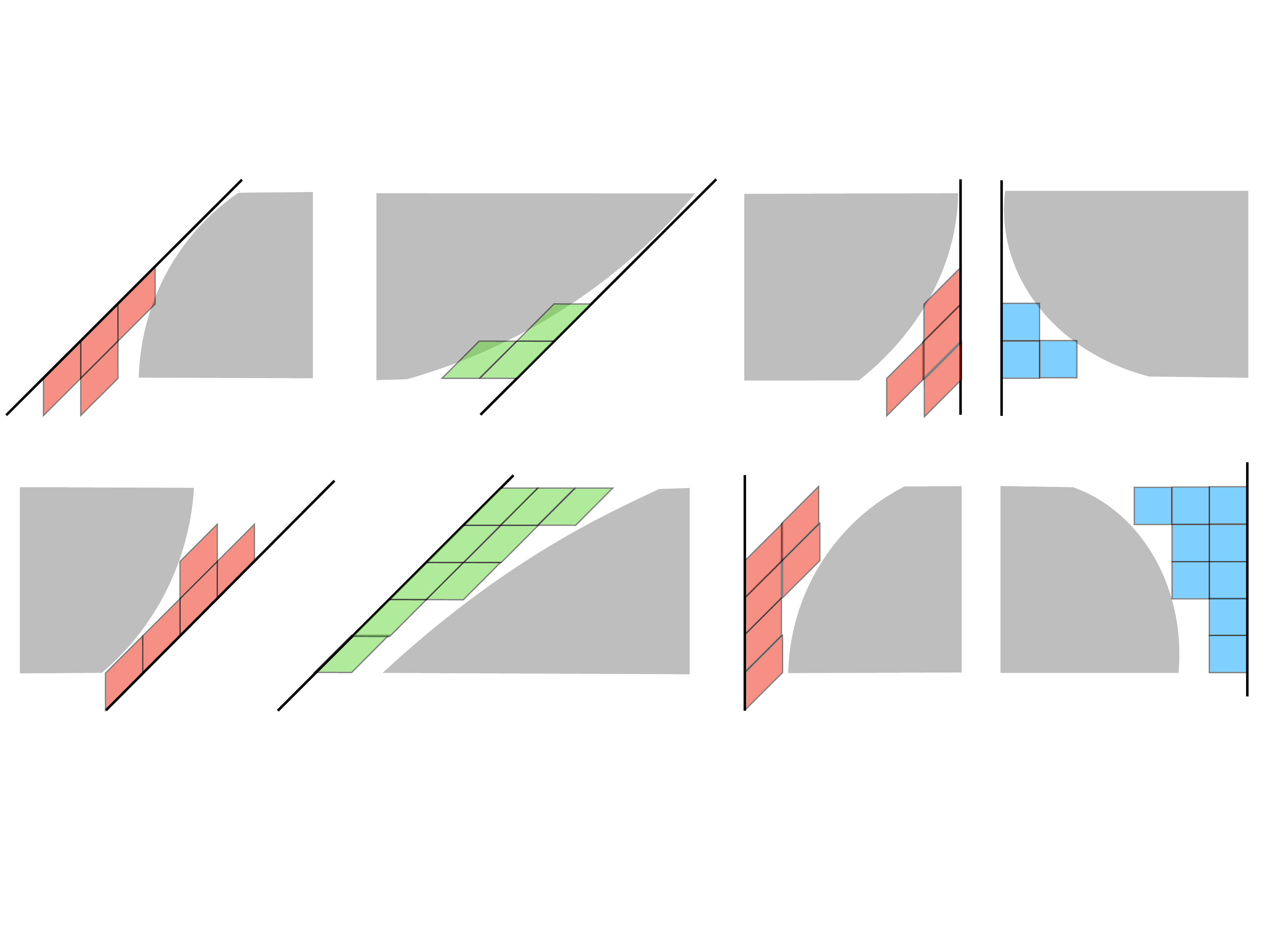}
 \caption{In the first row, the boundary edge of ${\mathfrak P}^t$ and the boundary of the liquid region (the shadow) are tangent to each other, in the second row they are not.}
 \label{f:lbcase}
\end{figure}

We recall from \eqref{e:extend} that the solution $f_t(\cdot;\beta,s)$ of the complex Burgers equation \eqref{e:burgereq22} at time $t$ can be extended from $\{ x :  (x,t)\in\Omega({\mathfrak P};\beta,s)\}$
to a meromorphic function on $\gamma_t^{\beta,s,+}$.  In the rest of this section we describe the zeros and poles of $ f_t(Z;\beta,s)$. From the discussion above, using $\pi^t_{Q^{ \beta , s }}$ from \eqref{e:coveratr},  $\gamma_t^{ \beta , s  ,+}$ can be identified as two copies of $\Omega^t({\mathfrak P};\beta,s)=\{( x , r): ( x , r)\in\Omega({\mathfrak P};\beta,s),  t \leq  r\leq T\}$ gluing together. In the interior of each copy, either $ f_t(Z;\beta,s)$ is in the lower half complex plane, or in the upper half complex plane. So $ f_t(Z;\beta,s)$ does not have zeros or poles in the interior of those two copies. Thus its zeros and poles appear only on the boundary. In the following we analyze $\arg^*f_t(Z;\beta,s)$ when $Z$ moves along the  boundary of $\Omega^t({\mathfrak P};\beta,s)$ under $\pi^t_{Q^{ \beta , s }}$ counterclockwise, which will give a full description of zeros and poles of $f_t(Z;\beta,s)$.

%
%
%

As in \eqref{e:lb}, we denote the slice of ${\mathfrak P}$ at time $t$ as
\begin{align}\label{e:lbcopy}
\{ x :  (x,t+)\in {\mathfrak P}\}=[\fa_1(t), \fb_1(t)]\cup[\fa_2(t), \fb_2(t)]\cup\cdots \cup [\fa_{r(t)}(t),\fb_{r(t)}(t)],
\end{align}
Then we can identify the spacial derivative of the height function $h^*$, $\del_x h^*_t(x;\beta,s)$ as a positive measure on 
$ [\fa_1(t), \fb_1(t)]\cup[\fa_2(t), \fb_2(t)]\cup\cdots \cup [\fa_{r(t)}(t),\fb_{r(t)}(t)]$, and it  satisfies \eqref{e:cond1} and \eqref{e:cond2}.

We decompose the boundary of the liquid region $\Omega^t({\mathfrak P};\beta,s)$
into two parts: 
\begin{align*}
\del\Omega^t({\mathfrak P};\beta,s)=\{(x,r)\in \del \Omega^t({\mathfrak P};\beta,s): r>t\}\cup \{(x,r)\in \del \Omega^t({\mathfrak P};\beta,s): r=t\}.
\end{align*}
When we move along the
boundary of $\Omega^t({\mathfrak P};\beta,s)$ corresponding to the part $\{(x,r)\in \del \Omega^t({\mathfrak P};\beta,s): r>t\}$ counterclockwise, its tangent vector also rotates counterclockwise. Using the relation \eqref{e:slope2},  $f_t(Z;\beta,s)$ takes value in $\bR\bP^1$ and continuously moves in negative direction, i.e. once it hits $-\infty$ then immediately starts from $+\infty$ again.  It follows our discussion before \eqref{e:x02}, it hits $-1$ exactly $d_0$ times, at $X_1^{(0)}, X_2^{(0)}, \cdots, X_{d_0}^{(0)}$.

%
%
%
For the part of the boundary $\Omega^t({\mathfrak P};\beta,s)$ corresponding to $\{(x,r)\in \del \Omega^t({\mathfrak P};\beta,s): r=t\}$, the argument of the complex slope $f_t(Z;\beta,s)$ is given explicitly by the defining relation \eqref{e:ft2}. More precisely,
for $Z$ corresponding to the liquid region $\{ x :  (x,t)\in\Omega^t({\mathfrak P};\beta,s)\}\cap (\fa_i(t), \fb_i(t))$, i.e. $Z=(f_t(x;\beta,s), x)\in \gamma_t^{\beta,s,+}$ with $z(Z)\in \{ x :  (x,t)\in\Omega^t({\mathfrak P};\beta,s)\}\cap (\fa_i(t), \fb_i(t))$, the definition of the complex slope \eqref{e:ft2} gives,
\begin{align}\label{e:argft}
\arg^* f_t(Z;\beta,s)=\arg^* f_t(x;\beta,s)=-\pi \del_x h_t^*(x;\beta,s).
\end{align}
It turns out the above relation \eqref{e:argft} holds on the whole interval $(\fa_i(t), \fb_i(t))$. For $Z\in \gamma_t^{\beta,s,+}$ with $z(Z)\in (\fa_i(t), \fb_i(t))\setminus \{ x :  (x,t)\in\Omega^t({\mathfrak P};\beta,s)\}$, either $\del_x h_t^*(z(Z);\beta,s)=0
$ or $\del_x h_t^*(z(Z);\beta,s)=1$. In this case, by our construction of the extension $f_t(Z;\beta,s)$, there exists some point $(x,r)\in \overline{\Omega^t({\mathfrak P};\beta,s)} $, such that
\begin{align}\label{e:xrpair}
f_t(Z;\beta,s)=f_r(x;\beta,s),\quad
z(Z)-\frac{tf_t(Z;\beta,s)}{f_t(Z;\beta,s)+1}=x-\frac{rf_r(x;\beta,s)}{f_r(x;\beta,s)+1}.
\end{align}
Since both $z(Z)$ and $x$ are real, it follows that $f_t(Z;\beta,s)=f_r(x;\beta,s)$ are real, and $(x,r)$ is on the boundary of the liquid region $\Omega^t({\mathfrak P};\beta,s)$. We can rewrite the second relation in \eqref{e:xrpair} as
\begin{align}\label{e:xrpair2}
\frac{r-t}{x-z(Z)}=\frac{f_r(x;\beta,s)+1}{f_r(x;\beta,s)}.
\end{align}
Comparing with \eqref{e:slope2}, the righthand side of \eqref{e:xrpair2} is the slope of the tangent vector of the liquid region $\Omega^t({\mathfrak P};\beta,s)$ at $(x,r)$. In other words, the line from $(z(Z),t)$ to $(x,r)$ is tangent to the boundary of the liquid region $\Omega^t({\mathfrak P};\beta,s)$. Since the segment from $(z(Z),t)$ to $(x,r)$ stays in the frozen region, either it is tiled with lozenges of shapes $\up{scale=1},\rig{scale=1}$, then $\del_x h_t^*(z(Z);\beta,s)=1$ and $f_t(Z;\beta,s)=f_r(x;\beta,s)\leq 0$;
or it is tiled with lozenges of shape $\emp{scale=1}$, then $\del_x h_t^*(z(Z);\beta,s)=0$ and $f_t(Z;\beta,s)=f_r(x;\beta,s)\geq 0$,
as illustrated in Figure \ref{f:tangent}. In both cases we have
\begin{align}\label{e:argft2}
\arg^* f_t(Z;\beta,s)=\arg^* f_r(x;\beta,s)=-\pi \del_x h_t^*(z(Z);\beta,s).
\end{align}

\begin{figure}
 \includegraphics[scale=0.5,trim={0cm 10cm 5cm 9cm},clip]{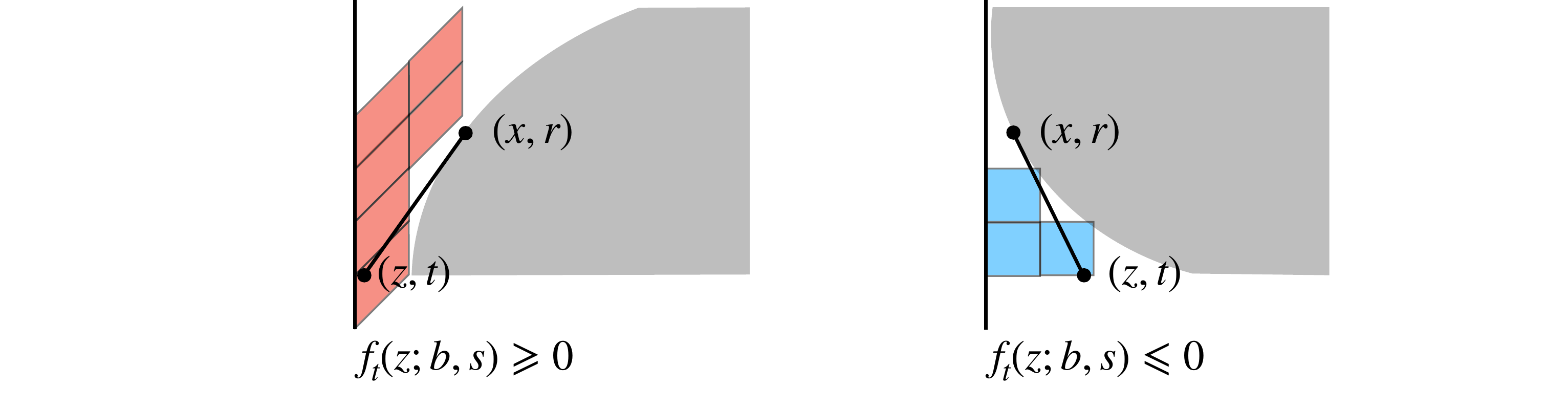}
\caption{The sign of $f_t(z;\beta,s)$ is determined by the type of lozenges in frozen region. }
\label{f:tangent}
 \end{figure}

By the same argument, we can also study $\arg^* f_t(Z;\beta,s)$ in small neighborhoods of the boundary points $\fa_i(t), \fb_i(t)$. There are several cases: for $z(Z)\in (\fa_i(t)-\varepsilon, \fa_i(t))$
\begin{align}\label{e:argf1}
\arg^* f_t(Z;\beta,s)
=\left\{
\begin{array}{cc}
0, &\text{if the boundary edge of ${\mathfrak P}$ containing $\fa_i(t)$ has slope $\infty$, }\\
-\pi, &\text{if the boundary edge of ${\mathfrak P}$ containing $\fa_i(t)$ has slope $1$. }
\end{array}
\right.
\end{align}
And for $z(Z)\in (\fb_i(t), \fb_i(t)+\varepsilon)$
\begin{align}\label{e:argf2}
\arg^* f_t(Z;\beta,s)
=\left\{
\begin{array}{cc}
-\pi, &\text{if the boundary edge of ${\mathfrak P}$ containing $\fb_i(t)$ has slope $\infty$, }\\
0, &\text{if the boundary edge of ${\mathfrak P}$ containing $\fb_i(t)$ has slope $1$. }
\end{array}
\right.
\end{align}
If we extend $\del_x h_t^*(x;\beta,s)$ such that for $x\in (\fa_i(t)-\varepsilon, \fa_i(t))$
\begin{align*}
\del_x h_t^*(x;\beta,s)
=\left\{
\begin{array}{cc}
0, &\text{if the boundary edge of ${\mathfrak P}$ containing $\fa_i(t)$ has slope $\infty$, }\\
1, &\text{if the boundary edge of ${\mathfrak P}$ containing $\fa_i(t)$ has slope $1$. }
\end{array}
\right.
\end{align*}
and for $x\in (\fb_i(t), \fb_i(t)+\varepsilon)$
\begin{align*}
\del_x h_t^*(z;\beta,s)
=\left\{
\begin{array}{cc}
1, &\text{if the boundary edge of ${\mathfrak P}$ containing $\fb_i(t)$ has slope $\infty$, }\\
0, &\text{if the boundary edge of ${\mathfrak P}$ containing $\fb_i(t)$ has slope $1$. }
\end{array}
\right.
\end{align*}
Then the relation \eqref{e:argft2} can be extended to the whole interval $(-\fa_i(t)-\varepsilon, \fb_i(t)+\varepsilon)$:
\begin{align}
\label{e:argftall}
\arg^* f_t(Z;\beta,s)=-\pi \del_x h_t^*(z(Z);\beta,s),\quad z(Z)\in(-\fa_i(t)-\varepsilon, \fb_i(t)+\varepsilon).
\end{align}

The above discussions give a complete description of $\arg^* f_t(Z;\beta,s)$ when $Z$ moves along the
boundary of $\Omega^t({\mathfrak P};\beta,s)$ under $\pi^t_{Q^{ \beta , s }}$ counterclockwise.

Let $m_{ t }(z;\beta, s)$ be the Stieltjes transform of the measure $\del_{x}  h_t^*(x;\beta,s) $, 
\begin{align*}
m_{ t }(z;\beta, s)=\int_{\{ x :  (x,t+)\in {\mathfrak P}\}} \frac{\del_{x} h_ t ^*( x ; \beta, s)}{z- x }\rd  x .
\end{align*}
As a meromorphic function over $\bC\bP^1$, we can lift $m_{ t }(z;\beta, s)$ to a meromorphic function over $\gamma_t^{ \beta , s  ,+}$ using the covering map \eqref{e:zmatr}. 
We notice that on $x\in (\fa_i(t), \fb_i(t))$ for any $1\leq i\leq r(t)$,
\begin{align*}\begin{split}
&\lim_{\eta\rightarrow 0+}\Im[m_{ t }( x +\ri\eta;\beta, s)]=-\ri\pi \del_{x} h_ t ^*( x ; \beta, s)=\arg^*  f_t( Z ;\beta, s),\\ 
&\lim_{\eta\rightarrow 0-}\Im[m_{ t }( x +\ri\eta;\beta, s)]=\ri\pi \del_{x} h_ t ^*( x ; \beta, s)=-\arg^* \overline{ f_t( Z ;\beta, s)},
\end{split}\end{align*}
where $z(Z)=x\in (\fa_i(t),\fb_i(t))$.
Therefore, if we take the ratio of $e^{m_{ t }(z;\beta, s)}$ and $ f_t(z;\beta, s)$, the remaining can be glued along the cut $\{(  f_t( x ;\beta, s), x ):  (x,t)\in\Omega^t({\mathfrak P};\beta,s)\}\cup \{(\overline{  f_t( x ;\beta, s)}, x ):  (x,t)\in\Omega^t({\mathfrak P};\beta,s)\}$ i.e. identifying $( f_t( x ;\beta, s), x )$ with $(\overline{ f_t( x ;\beta, s)}, x )$, to a meromorphic function, which does not have zeros or poles on $(\fa_i(t), \fb_i(t))$.
More precisely,
let
\begin{align}\label{e:gszmut}
 f_t(Z;\beta, s)=e^{m_{ t }(z;\beta, s)}g_ t (Z;\beta, s), \quad z=z(Z),
\end{align}
then $g_ t (Z;\beta, s)$, as a meromorphic function over $\gamma^{ \beta , s  ,+}_ t $, can be glued along the cut $\{(  f_t( x ;\beta, s), x ):  (x,t)\in\Omega^t({\mathfrak P};\beta,s)\}\cup \{(\overline{  f_t( x ;\beta, s)}, x ):  (x,t)\in\Omega^t({\mathfrak P};\beta,s)\}$, to a meromorphic function, which does not have zeros or poles on $(\fa_i(t), \fb_i(t))$. Moreover, for $x\in (\fa_i(t)-\varepsilon, \fa_i(t))$, we have 
\begin{align}\label{e:mtx}
\lim_{\eta\rightarrow 0}\Im[m_{ t }( x +\ri\eta;\beta, s)]=0.
\end{align}
If the boundary edge of ${\mathfrak P}$ containing $\fa_i(t)$ has slope $\infty$, then \eqref{e:argf1} implies that $\arg^* f_t(Z;\beta,s)=0$, where $z(Z)=x$.  It matches with \eqref{e:mtx}. We conclude that $g_t(Z;\beta,s)$ does not have zeros or poles for $z(Z)\in (\fa_i(t)-\varepsilon, \fb_i(t))$.
If the boundary edge of ${\mathfrak P}$ containing $\fa_i(t)$ has slope $1$, then \eqref{e:argf1} implies that  $\arg^* f_t(x;\beta,s)=-\pi$. Then $e^{m_t(x;\beta,s)}$ and $f_t(Z;\beta,s)$ differ by a sign. We conclude that  $g_t(Z;\beta,s)$ has a pole at $z(Z)=\fa_i(t)$. Similarly, using \eqref{e:argf2},  if the boundary edge of ${\mathfrak P}$ containing $\fb_i(t)$ has slope $1$, then $g_t(Z;\beta,s)$ does not have zeros or poles on $z(Z)\in (\fa_i(t), \fb_i(t)+\varepsilon)$;
If the boundary edge of ${\mathfrak P}$ containing $\fb_i(t)$ has slope $\infty$, then  $g_t(Z;\beta,s)$ has a zero at $z(Z)=\fb_i(t)$. In summary, we have the following description of zeros of poles for $g_t(Z;\beta,s)$ for $Z$ in a neighborhood of $z^{-1}([\fa_i(t), \fb_i(t)])\subset \gamma_t^{\beta,s,+}$, for any $1\leq i\leq r(t)$,
\begin{enumerate}
\item $g_t(Z;\beta,s)$ has a pole at $z(Z)=\fa_i(t)$ if only if the boundary edge of ${\mathfrak P}$ containing $\fa_i(t)$ has slope $1$.
\item $g_t(Z;\beta,s)$ has a zero at $z(Z)=\fb_i(t)$ if only if the boundary edge of ${\mathfrak P}$ containing $\fa_i(t)$ has slope $\infty$.
\item $g_t(Z;\beta,s)$ does not have other poles or zeros.
\end{enumerate}
Using the decomposition \eqref{e:gszmut}, we can rewrite the complex Burgers equation \eqref{e:burgereq3} as
\begin{align}\label{e:burgereq4}
\del_{t} m_{ t }(z;\beta, s)+\frac{\del_{t}g_ t (Z;\beta, s)}{g_ t (Z;\beta, s)}
=-\frac{\del_{z}f_t(Z;\beta, s)}{f_t(Z;\beta, s)+1},\quad z=z(Z),\quad Z\in \gamma^{\beta,s,+}_t.
\end{align}

\subsection{Hurwitz Space and Branch Points}
We recall that Hurwitz spaces are moduli spaces of covers of $\bC\bP^{1}$
 with fixed monodromy group. Two covers $( f,z)\in \gamma\mapsto \bC\bP^{1}$ and $( f',z)\in \gamma'\mapsto \bC\bP^{1}$ 
are equivalent if there
exists an isomorphism $\pi : \gamma\mapsto \gamma'$ such that $\pi(f,z)=(f',z)$.

We recall the Riemann surface $\gamma_t^{ \beta , s  ,+}$ from \eqref{e:zmatr},
and $d_0$ is the number of horizontal boundary edges of ${\mathfrak P}$ above $t$. It follows from \eqref{e:x02},
the map \eqref{e:zmatr}
\begin{align}\label{e:zmattcopy}
(f,z)\in \gamma_t^{ \beta , s  ,+}\mapsto z\in \bC\bP^1,
\end{align}
makes $\gamma_t^{ \beta , s  ,+}$ a $d_0$-sheeted branched covering of $\bC\bP^1$. Since $Q^{ \beta , s }$ is real, the branch points of \eqref{e:zmattcopy} come in pairs. We denote the branch points of \eqref{e:zmattcopy} as $(f_1, z_1), (\bar{f}_1, \bar{z}_1), (f_2, z_2), (\bar{f}_2, \bar{z}_2), \cdots, (f_w, z_w), (\bar{f}_w, \bar{z}_w)$ with ramification indices $r_1, r_1, r_2,r_2,\cdots, r_w, r_w$ respectively. 
In a neighborhood $V_i$ of the branch point $(f_i, z_i)$, $f$
is given by:
\begin{align*}
f_t(Z;\beta, s)= f_i+c_i(z(Z)-z_i)^{1/r_i}+\cdots, \quad Z=(f_t(Z; \beta, s),z(Z))\in V_i,
\end{align*}
and analogous expression holds in a neighborhood of $(\bar f_i, \bar z_i)$.
Moreover $\gamma_t^{ \beta , s  ,+}$ has genus $0$, the Riemann--Hurwitz formula gives that
\begin{align*}
\sum_{i=1}^w (r_i-1)=d_0-\#\{\text{connected components of $\gamma_t^{ \beta , s  ,+}$}\}=d_0-1.
\end{align*}

For general polygonal domain $\fP$, we are lack of a good description of the locations of those  branch points $(f_1, z_1), (\bar{f}_1, \bar{z}_1), (f_2, z_2), (\bar{f}_2, \bar{z}_2), \cdots, (f_w, z_w), (\bar{f}_w, \bar{z}_w)$. However, if $\fP$ is as in Definition \ref{def:oneend}, we do know something about those branch points. 
We recall from our discussion around \eqref{e:zmatr}, $\gamma_t^{\beta,s,+}$ corresponds to $\Omega^t(\fP;\beta,s)\deq \{( x , r): ( x , r)\in\Omega({\mathfrak P};\beta,s),  t \leq  r\leq T\}$, and can be identified as the gluing of two copies of $\Omega^t(\fP;\beta,s)$ along the boundary using the map \eqref{e:coveratr}.  The points which are mapped to $\bR\bP^1$ under the covering map \eqref{e:zmattcopy} are the boundary of $\Omega^t({\mathfrak P};\beta,s)$. 
The preimage of $\{ x :  (x,t+)\in {\mathfrak P}\}$ under the covering map \eqref{e:zmattcopy} may consist of several copies. By slightly abuse of notation, in the rest of this paper, we denote  $ z^{-1}(\{ x :  (x,t+)\in {\mathfrak P}\}$, the specific copy which includes the bottom boundary of $\Omega^t({\mathfrak P};\beta,s)$.
The following proposition states that if $\fP$ is as in Definition \ref{def:oneend},  those branch points can not be too close to $ z^{-1}(\{ x :  (x,t+)\in {\mathfrak P}\}$.

\begin{proposition}\label{p:branchloc}
If the polygonal domain $\fP$ has exactly one horizontal upper boundary edge (as in Definition \ref{def:oneend}), then the covering map \eqref{e:zmattcopy} does not have a branch point in a neighborhood of ${ z^{-1}(\{ x :  (x,t+)\in {\mathfrak P}\})}$ over $\gamma_t^{\beta,s,+}$.
\end{proposition}

\begin{proof}
We identify $\gamma_t^{\beta,s,+}$ with $\Omega^t(\fP;\beta,s)$, then
$ z^{-1}(\{ x :  (x,t+)\in {\mathfrak P}\}$ is the copy which includes the bottom boundary of $\Omega^t({\mathfrak P};\beta,s)$.
If there is a branch point in a small neighborhood of $ z^{-1}(\{ x :  (x,t+)\in {\mathfrak P}\})$, then there is another copy of the preimage of $\{ x :  (x,t+)\in {\mathfrak P}\}$ under the covering map \eqref{e:zmattcopy}  which is close to $ z^{-1}(\{ x :  (x,t+)\in {\mathfrak P}\})$.
 There is another part of the boundary of $\Omega^t({\mathfrak P};\beta,s)$ which is close to its bottom boundary. This happens only if the boundary of $\Omega^t({\mathfrak P};\beta,s)$ has a cusp facing upwards. Say $(x,r)$ is a cusp point of $\Omega^t(\fP;\beta,s)$, such that the cusp at $(x,r)$ faces upwards.  Then the horizontal slice of $\Omega^t(\fP;\beta,s)$ at time $r$ splits at $(x,r)$.
By our assumption on the domain $\fP$ (as in Definition \ref{def:oneend}), the corresponding nonintersecting Bernoulli walks all end on the interval $[\fa(T-), \fb(T-)]$ at time $T$. Especially the liquid region $\Omega^t(\fP;\beta,s)\cap \bR\times[r,T]$ is connected. In fact it is path-connected to the point where $\Omega^t(\fP;\beta,s)$ is tangent to $[\fa(T-), \fb(T-)]\times T$. If the horizontal slice of $\Omega^t(\fP;\beta,s)$ at time $r$ splits at $(r,s)$, we conclude that $\Omega^t(\fP;\beta,s)$ contains a hole. However, this contradicts to the fact that 
$\Omega^t(\fP;\beta,s)$ is simply connected.
This finishes the proof of Proposition \ref{p:branchloc}.
\end{proof}

\subsection{Rauch Variational Formula}
In Sections \ref{s:Gauss} and \ref{s:solve}, we need to estimate the difference of two solutions of the complex Burgers equation \eqref{e:ft2} with different boundary data: $\del_z \ln f_s(Z^1; \beta^1,s)-\del_z \ln f_s(Z^0;\beta^0,s)$ where $z(Z^0)=z(Z^1)=z$. To estimate it we interpolate the boundary data 
\begin{align*}
\beta^\theta=(1-\theta)\beta^0+\theta \beta^1, \quad 0\leq \theta\leq 1,
\end{align*}
and write 
\begin{align*}
\del_z \ln f_s(Z^1; \beta^1,s)-\del_z \ln f_s(Z^0;\beta^0,s)=\int_0^1 \del_\theta \del_z \ln f_s(Z^\theta;\beta^\theta,s)\rd \theta,\quad Z_\theta\in \gamma_s^{\beta^\theta,s}, \quad z(Z^\theta)=z.
\end{align*}
For any $0\leq \theta\leq 1$,  we have that
\begin{align*}
(f,z)\in \gamma_s^{\beta^\theta,s}\mapsto z\in \bC\bP^1
\end{align*}
is a $d_0$-sheeted branched covering of the Riemann sphere $\bC\bP^1$, where $d_0$ is the number of horizontal boundary edges of ${\mathfrak P}$ above $s$.
Similarly we denote its branch points as $Z_1^\theta=(f_1^\theta, z^\theta_1), \bar Z_1^\theta=(\bar{f}^\theta_1, \bar{z}^\theta_1), Z_2^\theta=(f^\theta_2, z^\theta_2), \bar Z_2^\theta=(\bar{f}^\theta_2, \bar{z}^\theta_2), \cdots, Z_w^\theta=(f^\theta_w, z^\theta_w), \bar Z_w^\theta=(\bar{f}^\theta_w, \bar{z}^\theta_w)$ with ramification indices $r_1, r_1, r_2,r_2,\cdots, r_w, r_w$ respectively. 

%

In the rest of this section, we recall the Rauch variational formula \cite{MR110798,MR110799,MR2131384}, which can be used to describe the changes of the Schiffer kernel \cite{MR39812,BE2018} on $\gamma^{\beta^\theta,s}_s$. Then the change $\del_\theta 
\del_z f_s(Z^\theta;\beta^\theta,s)$ of $\del_z f_s(Z^\theta;\beta^\theta,s)$ can be expressed as an integral of the  Schiffer kernel. 

We recall the Schiffer kernel $B(Z,Z';\beta^\theta,s)$ from \cite{MR39812,BE2018} on $(Z,Z')=((f,z), (f',z'))\in\gamma_s^{\beta^\theta,s}\times \gamma_s^{\beta^\theta,s}$:  it is a symmetric meromorphic bilinear differential (meromorphic $1$-form of $z$, tensored by
a meromorphic $1$-form of $z'$, $B(Z,Z';\beta^\theta,s)=B(Z',Z;\beta^\theta,s)$); it has  double poles on the diagonal, such that in a small neighborhood $V$ of $(f,z)\in \gamma^{\beta^\theta,s}_s$, for $(Z,Z')=((f,z), (f',z'))\in V\times V$,
\begin{align}\label{e:Schiffer}
B(Z,Z';\beta^\theta,s)=\frac{1}{(z-z')^2}\rd z\rd z'
+ \OO(1). 
\end{align}
The integral of the Schiffer kernel
\begin{align}\label{e:third}
Q_{Z_0, Z_1}(Z;\beta^\theta,s)=\int_{Z'=Z_0}^{Z'=Z_1}B(Z,Z';\beta^\theta,s),
\end{align}
is a meromorphic $1$-form (called the third kind form) over  $\gamma_s^{\beta^\theta,s}$, having only poles at $Z_0, Z_1$ with residual $-1,1$.

The Rauch variational formula describes the change of the Schiffer kernel $B(Z,Z';\beta^\theta,s)\rd z\rd z'$ along $\theta$ in terms of the change of the branch points $Z_1^\theta, \bar Z_1^\theta, \cdots, Z_w^\theta, \bar Z_w^\theta$.
\begin{theorem}\label{thm:Rauch}
The derivative of the Schiffer kernel $B(Z,Z';\beta^\theta,s)$ with respect to the branch point $Z_i^\theta$ is given by
\begin{align}\begin{split}\label{e:Rauch}
\del_{Z_i^\theta} B(Z,Z';\beta^\theta,s)
= \frac{1}{2\pi\ri}\oint_{\cS_i} B(Z,Z'';\beta^\theta,s)B(Z'',Z';\beta^\theta,s)/\rd z(Z''), \end{split}\end{align}
where the contour $\cS_i$ encloses $Z_i^\theta$. 
The derivative of the third kind kernel $Q_{Z_1, Z_0}(Z; \beta^\theta, s)$ with respect to the branch point $Z_i^\theta$ is given by
\begin{align}\begin{split}\label{e:Rauch}
\del_{Z_i^\theta} Q_{Z_1, Z_0}(Z;\beta^\theta,s)
= \frac{1}{2\pi\ri}\oint_{\cS_i} B(Z,Z'';\beta^\theta,s)Q_{Z_1, Z_0}(Z'';\beta^\theta,s)/\rd z(Z''), \end{split}\end{align}
where the contour $\cS_i$ encloses $Z_i^\theta$. The same statements hold for other branch points.

\end{theorem}

We recall from \eqref{e:gszmut}, that we can decompose $f_s(Z;\beta^\theta,s)$ as
\begin{align}\label{e:fgtt0}
f_s(Z;\beta^\theta,s)=e^{m_s(z;\beta^\theta,s)}g_s(Z;\beta^\theta,s), \quad z=z(Z).
\end{align} 
As a meromorphic function over $\gamma_s^{\beta^\theta,s}$, the zeros and poles of $g_s(Z;\beta^\theta,s)$ are described at the end of Section \ref{s:generalb}. We recall the bottom boundary of ${\mathfrak P}^s$ from \eqref{e:lbcopy}
\begin{align}\label{e:lbcopy2}
\{ x :  (x,s+)\in {\mathfrak P}\}=[\fa_1(s), \fb_1(s)]\cup[\fa_2(s), \fb_2(s)]\cup\cdots \cup [\fa_{r(s)}(s),\fb_{r(s)}(s)].
\end{align}
There are $d_0$ boundary edges of ${\mathfrak P}^s$ with slope $0$.
The poles of $g_s(Z;\beta^\theta,s)$ correspond to edges of ${\mathfrak P}^s$ with slope $1$:
either the edge is not adjacent to the bottom boundary of ${\mathfrak P}^s$, or the edge is a left boundary, i.e. containing $\fa_i(s)$ for some $1\leq i\leq r(s)$. In total there are $d_0$ such edges. We denote those poles as $X_1^{(1)}, X_2^{(1)}, \cdots, X_{d_1}^{(1)}$ and $X_{d_1+1}^{(1)}=A_1, X_{d_1+2}^{(1)}=A_2, \cdots, X_{d_0}^{(1)}=A_{d_0-d_1}$.
The zeros of $g_s(Z;\beta^\theta,s)$ correspond to edges of ${\mathfrak P}^s$ with slope $\infty$:
either the edge is not adjacent to the bottom boundary of ${\mathfrak P}^s$, or the edge is a right boundary, i.e. containing $\fb_i(s)$ for some $1\leq i\leq r(s)$. In total there are $d_0$ such edges. We denote those poles as $X_1^{(\infty)}, X_2^{(\infty)}, \cdots, X_{d_\infty}^{(\infty)}$ and $X_{d_\infty+1}^{(\infty)}=B_1, X_{d_1+2}^{(\infty)}=B_2, \cdots, X_{d_0}^{(\infty)}=B_{d_0-d_\infty}$.
We also remark that inside the bottom boundary of $\fP^s$, for  $Z\in z^{-1}((\fa_1(s), \fb_1(s))\cup(\fa_2(s), \fb_2(s))\cup\cdots \cup (\fa_{r(s)}(s),\fb_{r(s)}(s)))\subset \gamma_s^{\beta^\theta,s}$, $g_s(Z;\beta^\theta,s)$ is real and positive.

From the discussion above, as a meromorphic function $ \ln f_s(Z;\beta^\theta,s)$ over $\gamma_s^{\beta^\theta,s}$, it has residuals along the bottom boundary of ${\mathfrak P}^s$ given by the measure $\del_x \beta^\theta(x)$,  
behaves like $-\ln(z(Z)-x_i^{(\infty)})$ in neighborhoods of  $X_1^{(\infty)}, X_2^{(\infty)},\cdots,  X_{d_0}^{(\infty)}$, and behaves like $\ln(z(Z)-x_i^{(1)})$ in neighborhoods of   $X_1^{(1)}, X_2^{(1)},\cdots,  X_{d_0}^{(1)}$.   We can write down $\ln f_s(Z;\beta^\theta,s)$ explicitly as a linear combination of the Schiffer kernel \eqref{e:Schiffer} and the third kind form \eqref{e:third}: for $Z=(f,z)\in \gamma_s^{\beta^\theta,s}$
\begin{align}\begin{split}\label{e:ftexp}
\ln f_s(Z;\beta^\theta,s)\rd z(Z)
&=-\int B(Z, X^\theta(x);\beta^\theta,s) \beta^\theta (x)\rd x\\
&+
\sum_{i=1}^{d_0}\int^{X_i^{(\infty)}}_{Z_0}Q_{W, Z_0}(Z;\beta^\theta,s)\rd W-\sum_{i=1}^{d_0}\int^{X_i^{(1)}}_{Z_0}Q_{W, Z_0}(Z;\beta^\theta,s)\rd W.
\end{split}\end{align}
We remark that the second line does not have a logarithmic singularity at $Z_0$, and is in fact independent of $Z_0$.
Using the Rauch variational formula Theorem \ref{thm:Rauch}, the derivative of $\ln f_s(Z;\beta^\theta,s)$ with respect to $\theta$ can be written down explicitly using the Schiffer kernel.
\begin{proposition}\label{p:derft}
The derivative of $ \ln f_s(Z;\beta^\theta,s)$ with respect to $\theta$ is given by
\begin{align}\label{e:dlogf}
\del_\theta \ln f_s(Z;\beta^\theta,s) \rd z(Z) =-\int B(Z,X^\theta(x);\beta^\theta,s)( \beta^1(x)- \beta^0(x)).
\end{align}
\end{proposition}
\begin{proof}
The derivative of $\ln f_s(Z;\beta^\theta,s)$ with respect to $\theta$ consists of two parts, either the derivative hits $ \beta^\theta(x)$ in the first term of \eqref{e:ftexp}, or the derivative hits one of those Schiffer kernels or third kind kernels. In the first case, $\del_\theta  \beta^\theta(x)=\beta^1(x)- \beta^0(x)$, and we get the term
\begin{align*}
\int B(Z,X^\theta(x);\beta^\theta,s)( \beta^1(x)- \beta^0(x)).
\end{align*} 
The derivative corresponds to the second case turns out to be zero. As in Theorem \ref{thm:Rauch}, the derivative of the Schiffer kernel $B(Z,Z';\beta^\theta,s)$, and the third kind kernel $Q_{Z_1, Z_0}(Z;\beta^\theta, s)$ with respect to $\theta$ breaks down into several contour integrals. Each contour integral corresponds to a branch point. The contribution corresponding to the branch point $Z_i^\theta$ is given by
\begin{align*}\begin{split}
&\phantom{{}={}}\frac{\del_\theta Z_i^\theta}{2\pi\ri} 
\oint_{\cS_i} \frac{B(Z,Z'';\beta^\theta,s)}{\rd z(Z'')}\left(-\int   B(Z'',X^\theta(x);\beta^\theta,s)\beta^\theta (x)\right.\\
&\phantom{{}\frac{\del_\theta Z_i^\theta}{2\pi\ri} 
\oint_{\cS_i} \frac{B(Z,Z'';\beta^\theta,s)}{\rd z(Z'')}=={}}\left.+
\sum_{i=1}^{d_0}\int^{X_i^{(\infty)}}_{Z_0}Q_{W, Z_0}(Z'';\beta^\theta,s)\rd W-\sum_{i=1}^{d_0}\int^{X_i^{(1)}}_{Z_0}Q_{W, Z_0}(Z'';\beta^\theta,s)\rd W\right)\\
&=\frac{\del_\theta Z_i^\theta}{2\pi\ri} \oint_{\cS_i} \frac{B(Z,Z'';\beta^\theta,s)}{\rd z(Z'')}\ln f_s(Z'';\beta^\theta,s)\rd z(Z'')\\
&=\frac{\del_\theta Z_i^\theta}{2\pi\ri} \oint_{\cS_i} B(Z,Z'';\beta^\theta,s)\ln f_s(Z'';\beta^\theta,s)=0,
\end{split}\end{align*}
where in the last line the integral vanishes, because the integrand is holomorphic at the 
branch point $Z_i^\theta$. Similarly the contribution from other branch points also vanishes. This finishes the proof of Proposition \ref{p:derft}.
\end{proof}

For any boundary data $\beta$, we can rewrite \eqref{e:dlogf} as the derivative with respect to the boundary data $\beta$,
\begin{align}\label{e:ddlogf}
\frac{ \del \ln f_s(Z;\beta,s)}{\del \beta}(Z') =-\frac{B(Z,Z';\beta,s)}{\rd z(Z)\rd z(Z')},\quad Z'\in \gamma_s^{\beta,s,+},
\end{align}
which is  the Schiffer kernel \eqref{e:Schiffer}. 
By plugging \eqref{e:ddlogf} into \eqref{e:fgtt0}, and noticing $\del_\beta m_t(z;\beta,s)=1/(z-x)^2$, we get
\begin{align}\label{e:ddlogg}
K(Z,Z';\beta,s)\deq \frac{ \del \ln g_s(Z;\beta,s)}{\del \beta}(Z') =\frac{1}{(z(Z)- z(Z'))^2}-\frac{B(Z,Z';\beta,s)}{\rd z(Z)\rd z(Z')},\quad Z'\in \gamma_s^{\beta,s,+},
\end{align}
which no longer have double poles along the diagonal.

We recall from Proposition \ref{p:branchloc}, in a  neighborhood of $ z^{-1}(\{ x :  (x,s+)\in {\mathfrak P}\})$, 
$f_s(Z;\beta,s)$
does not have a branch point. So does $\del_\beta \ln g_s(Z;\beta,s)$. Especially, combining with the discussion above, we conclude that $\del_\beta g_s(Z;\beta,s)$ is analytic. We can take further derivatives with respect to $Z$ and $\beta$, the results are still analytic in a small neighborhood of $ z^{-1}(\{ x :  (x,s+)\in {\mathfrak P}\})$. We collect this fact in the following Proposition, which will be used in Sections \ref{s:solve}.

\begin{proposition}\label{p:derg}
If the polygonal domain $\fP$ has exactly one horizontal upper boundary edge (as in Definition \ref{def:oneend}), for any time $0\leq s\leq T$ and boundary data $\beta$, as a function of $Z$, $\del_\beta \ln g_s(Z;\beta,s)$ and its derivatives with respect to $Z$ and $\beta$ are analytic  in a small neighborhood of $ z^{-1}(\{ x :  (x,s+)\in {\mathfrak P}\})$ over $\gamma_s^{\beta,s,+}$.
%

\end{proposition}

\section{An Ansatz}\label{s:ansatz}

In Section \ref{s:rw},
we have identified the lozenge tilings of $\fP$, with nonintersecting Bernoulli random walks $\{x_1(t), x_2(t), \cdots, x_{m(s)}(s)\}_{s\in[0,T]\cap \bZ_n}$ with $(x_1(s),x_2(s), \cdots, x_{m(s)}(s))\in \mathfrak M_s(\fP)$.
For any time $s\in [0,T)\cap \bZ_n$, 
and any particle configuration $(x_1,x_2,\cdots,x_{m})\in \mathfrak M_{s}(\fP)$, where
\begin{align*}
m=n\left((\beta^\fP(\fb_1(s),s)-\beta^\fP(\fa_1(s),s))
+\cdots +(\beta^\fP(\fb_{r(s)}(s),s)-\beta^\fP(\fa_{r(s)}(s),s))\right),
\end{align*}
we recall that $N_{s}(x_1, x_2, \cdots, x_m)$ is the number of nonintersecting Bernoulli walks starting from particle configuration $(x_1,x_2,\cdots,x_{m})$ at time $s$ as defined in \eqref{e:defwalk}. From the particle configuration $(x_1,x_2,\cdots,x_{m})\in \mathfrak M_{s}({\mathfrak P})$, by removing particles created at time $s$, we get a particle configuration  $(x'_1,x'_2,\cdots,x'_{m'})$, and
\begin{align*}
N_{s-}(x'_1,x'_2,\cdots,x'_{m'})=N_{s}(x_1, x_2, \cdots, x_m).
\end{align*}

We denote the empirical measure of the particle configuration $\bmx=(x_1, x_2, \cdots, x_m)$ as
\begin{align*}
\rho(x; \bmx)=\sum_{i=1}^m \bm1(x\in [x_i, x_i+1/n]).
\end{align*}
Then it gives a height function $\beta(x;\bmx)$ on the bottom boundary of $\fP\cap \bR\times[s, T]$, $[\fa_1(s), \fb_1(s)]\cup[\fa_2(s), \fb_2(s)]\cup\cdots \cup [\fa_{r(s)}(s),\fb_{r(s)}(s)]$ :
\begin{align*}
\beta(x;\bmx)= \left(\beta^{{\mathfrak P}}(\fa_i(s), s)+\int_{\fa_i(s)}^ x  \rho( x;\bmx )\rd  x \right),
\end{align*}
for any $ x \in [\fa_i(s), \fb_i(s)]$. Then $N_{t}(x_1, x_2, \cdots, x_m)$ is also the number of lozenge tilings of the domain $\fP^s=\fP\cap \bR\times[s, T]$ (after rescaling by a factor $n$), with bottom height function given by $\beta(x;\bmx)$.

We recall the region $\fP^s$ from \eqref{e:omegas}, and the minimizer $\{h_t^*(x;\beta,s)\}_{s\leq t\leq T}$ of the variational problem \ref{e:varWt}.  We denote the corresponding complex slope as $f_t(Z;\beta,s)$, which is a meromorphic function on the Riemann surface $Z\in \gamma_t^{\beta, s,+}$. It has the following decomposition 
\begin{align}\label{e:fgtt}
f_t(Z;\beta,s)=e^{m_t(z;\beta,s)}g_t(Z;\beta,s), \quad z=z(Z),\quad Z\in \gamma_t^{\beta, s,+}.
\end{align} 
As a meromorphic function over $\gamma_t^{\beta,s,+}$, the zeros and poles of $g_t(Z;\beta,s)$ are described at the end of Section \ref{s:generalb}. We recall the bottom boundary of ${\mathfrak P}^t$ from \eqref{e:lbcopy}
\begin{align}\label{e:lbcopy2}
\{ x :  (x,t+)\in {\mathfrak P}\}=[\fa_1(t), \fb_1(t)]\cup[\fa_2(t), \fb_2(t)]\cup\cdots \cup [\fa_{r(t)}(t),\fb_{r(t)}(t)].
\end{align}
There are $d_0$ boundary edges of ${\mathfrak P}^t$ with slope $0$.
The poles of $g_t(Z;\beta,s)$ correspond to edges of ${\mathfrak P}^t$ with slope $1$:
either the edge is not adjacent to the bottom boundary of ${\mathfrak P}^t$, or the edge is a left boundary, i.e. containing $\fa_i(t)$ for some $1\leq i\leq r(t)$. In total there are $d_0$ such edges. We denote those poles as $X_1^{(1)}, X_2^{(1)}, \cdots, X_{d_1}^{(1)}$ and $X_{d_1+1}^{(1)}=A_1, X_{d_1+2}^{(1)}=A_2, \cdots, X_{d_0}^{(1)}=A_{d_0-d_1}$.
The zeros of $g_t(Z;\beta,s)$ correspond to edges of ${\mathfrak P}^t$ with slope $\infty$:
either the edge is not adjacent to the bottom boundary of ${\mathfrak P}^t$, or the edge is a right boundary, i.e. containing $\fb_i(t)$ for some $1\leq i\leq r(t)$. In total there are $d_0$ such edges. We denote those poles as $X_1^{(\infty)}, X_2^{(\infty)}, \cdots, X_{d_\infty}^{(\infty)}$ and $X_{d_\infty+1}^{(\infty)}=B_1, X_{d_1+2}^{(\infty)}=B_2, \cdots, X_{d_0}^{(\infty)}=B_{d_0-d_\infty}$.

For any $x\in[\fa_k(t), \fb_k(t)]$, we define the sets  $A_{t}(x), B_{t}(x)$, such that $A_{t}(x)$ collects poles of $g_t(Z;\beta,s)$, and $B_{t}(x)$ collects zeros of $g_t(Z;\beta,s)$:
\begin{align}\begin{split}\label{e:defAB}
&A_{t}(x)\deq\{z(X_i^{(1)}): 1\leq i\leq d_0, z(X_i^{(1)})\leq \fa_k(t), X_i^{(1)} \text{ corresponds to a left boundary edge} \},\\
&B_{t}(x)\deq\{z(X_i^{(\infty)}): 1\leq i\leq d_0, z(X_i^{(\infty)})\geq \fb_k(t), X_i^{(\infty)} \text{ corresponds to a right boundary edge} \}.
\end{split}\end{align}
For any $a\in A_{t}(x)$ and $b\in B_{t}(x)$, we have $a\leq x\leq b$. We remark that $A_{t}(x), B_{t}(x)$ only depend on the interval 
$[\fa_k(t),\fb_k(t)]$. And we  can extend them to a constant function in a neighborhood of $z^{-1}([\fa_k(t),\fb_k(t)])$ on $\gamma_t^{\beta,s,+}$. The poles $a=z(X_i^{(1)})\in A_t(x)$ travel at rate $1$, i.e.
$\del_t a=1$, and the zeros $b=z(X_i^{(\infty)})\in B_t(x)$ does not move, i.e. $\del_t b=0$.

If ${\mathfrak P}$ contains a left boundary edge with slope $1$ ending at time $t$, or a right boundary edge with slope $\infty$ ending at time $t$,  $A_t, B_t$ are not continuous at time $t$.
For any $x\in [\fa_1(t-), \fb_1(t-)]\cup[\fa_2(t-), \fb_2(t-)]\cup\cdots \cup [\fa_{r(t-)}(t-),\fb_{r(t-)}(t-)]$, we can similarly define $A_{t-}(x)$ and $B_{t-}(x)$. More precisely, 
$A_{t-}(x)=\lim_{r\rightarrow t-}A_r(x)$ is the same as  adding  the projection of left boundary edges with slope $1$ ending at time $t$ to $A_t(x)$; 
$B_{t-}(x)=\lim_{r\rightarrow t-}B_r(x)$ is the same as    adding  the projection of right boundary edges with slope $\infty$ ending at time $s$ to $B_t(x)$. Since the domain $\fP$ is formed by $3d$ segments with slopes $0,1,\infty$ cyclically repeated as we follow the boundary in the counterclockwise direction, a left boundary edge with slope $1$ is followed by a horizontal edge, and a right boundary edge with slope $\infty$ is preceded by a horizontal edge, as in Figure \ref{f:newp}. 
We have $A_{t-}=A_t$ and $B_{t-}=B_t$ for $t\neq \{T_0, T_1,T_2,\cdots, T_{p}\}$, when $\fP$ does not contain horizontal edges at $t$.

We also recall the functional $W_s(\beta)$ from \eqref{e:varWt}, and define
\begin{align}\begin{split}\label{e:defYt+}
&\phantom{{}={}}Y_{s}(\beta)=Y_{s}(\bmx)
\deq W_s(\beta)-\frac{1}{2}\int \ln|x-y|\rho(x;\bmx)\rho(y;\bmx)\rd x\rd y\\
&+\int \sum_{a\in A_{s}(x)}\int_{a}^x \ln(x-y)\rho(x;\bmx)\rd y\rd x 
+\int \sum_{b\in B_{s}(x)} \int_x^{b}\ln (y-x)\rho(x;\bmx)\rd y\rd x,
\end{split}\end{align}
and similarly
\begin{align}\begin{split}\label{e:defYt-}
&\phantom{{}={}}Y_{s-}(\beta)=Y_{s-}(\bmx)
\deq W_s(\beta)-\frac{1}{2}\int \ln|x-y|\rho(x;\bmx')\rho(y;\bmx')\rd x\rd y\\
&+\int \sum_{a\in A_{s-}(x)}\int_{a}^x \ln(x-y)\rho(x;\bmx)\rd y\rd x 
+\int \sum_{b\in B_{s-}(x)} \int_x^{b}\ln (y-x)\rho(x;\bmx)\rd y\rd x.
\end{split}\end{align}

%

The main result of this section is the following ansatz for the number of tilings of the domain $\fP^s$ with bottom height function given by $\beta(x;\bmx)$.
\begin{ansatz}\label{a:defE}
For any time $s\in [0,T)\cap \bZ_n$, and any particle configuration $(x_1,x_2,\cdots,x_{m})\in {\mathfrak M}_{s}({\mathfrak P})$, let  $(x'_1,x'_2,\cdots,x'_{m'})$ be the particle configuration from $\bmx$ by removing particles created at time $s$. We make the following ansatz
\begin{align*}
N_{s}(x_1, x_2, \cdots, x_m)=\frac{V(\bmx)e^{n^2Y_{s}(\beta)}E_{s}(x_1, x_2, \cdots, x_m)}{\prod_{j=1}^m\left(\prod_{a\in A_{s}(x_j)} \Gamma_n(x_j-a)\prod_{b\in B_{s}(x_j)}\Gamma_n(b-x_j-1/n)\right)},
\end{align*}
where $\Gamma_n(k/n)=(k/n)((k-1)/n)\cdots (1/n)=k!/n^k$ and
\begin{align*}
N_{s-}(x_1', x_2', \cdots, x_{m'}')=\frac{V(\bmx')e^{n^2Y_{s-}(\beta)} E_{s-}(x_1', x_2', \cdots, x_m')}{\prod_{j=1}^m\left(\prod_{a\in A_{s-}(x_j')} \Gamma_n(x_j'-a)\prod_{b\in B_{t-}(x_j')}\Gamma_n(b-x_j'-1/n)\right)},
\end{align*}
where 
\begin{align*}
\del_x \beta(x)=\sum_{i=1}^m \bm1(x\in [x_i, x_i+1/n]).
\end{align*}
\end{ansatz}

It turns out the correction terms  have no influence on the fluctuations of height functions. In Section \ref{s:eq} we derive the equation for $E_{s}$, and solve for the leading order term of $E_{s}$ in Section \ref{s:solve}.

By our definition, $N_{s}(x_1, x_2,\cdots, x_m)=N_{s-}(x_1', x_2', \cdots, x_{m'}')$, we have
the following relation between $ E_{s+}(\bmx)$ and $ E_{s-}(\bmx')$
\begin{align*}
\frac{  E_{s-}(\bmx')}{ E_{s}(\bmx)}
=\frac{V(\bmx)e^{n^2Y_{s}(\beta)}}{V(\bmx')e^{n^2Y_{s-}(\beta)}}
\frac{\prod_{j=1}^{m'}\left(\prod_{a\in A_{s-}(x_j')} \Gamma_n(x_j'-a)\prod_{b\in B_{s-}(x_j')}\Gamma_n(b-x_j'-1/n)\right)}{\prod_{j=1}^m\left(\prod_{a\in A_{s}(x_j)} \Gamma_n(x_j-a)\prod_{b\in B_{s}(x_j)}\Gamma_n(b-x_j-1/n)\right)}.
\end{align*}
If $\fP$ does not contain a horizontal boundary edge at $s$, i.e. $s\neq \{T_0, T_1,T_2,\cdots, T_p\}$, then $\bmx'=\bmx$, $A_{s-}=A_s$, $B_{s-}=B_s$ and $Y_{s}(\beta)=Y_{s-}(\beta)$. It follows  that $  E_{s-}(\bmx')/ E_{s}(\bmx)=1$. If $\fP$ contains a horizontal boundary edge at $s$, i.e. $s\in\{T_0,T_1, T_2,\cdots, T_p\}$, denoting the horizontal boundary edge as $[b,a]$, there are four cases as in Figure \ref{f:newp}. 
\begin{figure}
\begin{center}
 \includegraphics[scale=0.4,trim={0cm 16cm 0 7cm},clip]{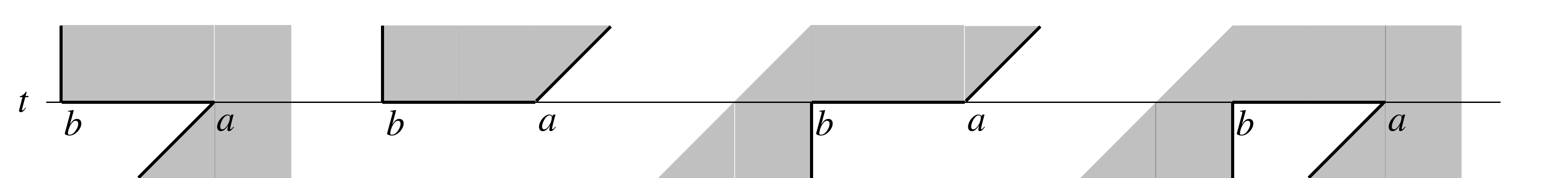}
 \caption{The polygon $\fP$ contains a horizontal edge $[b,a]$ at time $t$.}
 \label{f:newp}
 \end{center}
 \end{figure}
In this case we have
\begin{align}\label{e:etratio}
\frac{  E_{s-}(\bmx')}{ E_{s}(\bmx)}
\propto \frac{\prod_{j=1}^{m'}\left( \prod_{y\in [b,a)\cap \bZ_n}|x'_j-y| \right)\prod_{j: a\in A_{s-}(x_j')}\Gamma_n(x_j'-a) \prod_{j:  b\in B_{s-}(x_j')}\Gamma_n(b-x_j'-1/n)}
{e^{n^2\int (\int_{b}^{a}+\int_{a}^x\bm1(a\in A_{s-}(x))+ \int_{x}^{b}\bm1(b\in B_{s-}(x)))\ln|x-y| \rho(x;\bmx')\rd y\rd x   }},
\end{align}
up to a universal multiplicative factor depending only on the polygon $\fP$.
We want to emphasize that the righthand side of \eqref{e:etratio} splits
\begin{align}\label{e:split}
\prod_{j=1}^{m'}
\frac{e^{\left(\sum_{y\in [b,a)\cap \bZ_n}+\sum_{y\in [a, x_j')\cap \bZ_n}\bm1(a\in A_{s-}(x_j'))+\sum_{y\in (x_j', b)\cap \bZ_n}\bm1(b\in B_{s-}(x_j'))\right)\ln |x'_j-y| }}
{e^{n^2 \int_{x_{j'}}^{x_{j'+1/n}}(\int_{b}^{a}+\int_{a}^x\bm1(a\in A_{s-}(x))+ \int_{x}^{b}\bm1(b\in B_{s-}(x)))\rd y\ln|x-y| \rd x   }}.
\end{align}
We can further simplify the righthand side of \eqref{e:split}. For example, in the case as in the first subplot of Figure \ref{f:newp}, if  $x_j'\geq a$, then $a\in A_{s-}(x_j')$ and $b\not\in B_{s-}(x_j')$. And the term corresponding to $x_j'$ on the righthand side of \eqref{e:split} simplifies to
\begin{align}\label{e:onet}
\frac{e^{\sum_{y\in [b,x_j')\cap \bZ_n}\ln |x'_j-y| }}
{e^{n^2 \int_{x_{j'}}^{x_{j'+1/n}}\int_{b}^{x}\ln|x-y| \rd y\rd x   }}.
\end{align}
Using Stirling's formula, 
\begin{align*}\begin{split}
&\phantom{{}={}}\sum_{y\in [b,x_j')\cap \bZ_n}\ln |x'_j-y|=\ln((n(x_j'-b))!/n^{n(x_j'-b)})\\
&=n(x_j'-b)\ln(x_j'-b)-n(x_j'-b)+\frac{1}{2}\ln (n(x_j'-b))+\sqrt{2\pi}+\frac{1}{12n(x_j'-b)}+\OO(\frac{1}{n^2}).
\end{split}\end{align*}
We can also do a Taylor expansion for the denominator of \eqref{e:onet}
\begin{align*}\begin{split}
n^2 \int_{x_{j'}}^{x_{j'+1/n}}\int_{b}^{x}\ln|x-y| \rd y\rd x  
=n^2\int_{x_j'}^{x_j'+1/n}(x-b)\ln(x-b)-(x-b)\rd x\\
=n((x_j'-b)\ln(x_j'-b)-(x'-b))+\frac{1}{2}\ln(x_j'-b)+\frac{1}{6n(x_j'-b)}+\OO(\frac{1}{n^2}).
\end{split}\end{align*}
Therefore up to a multiplicative universal factor (depending only on the polygon $\fP$), the ratio \eqref{e:onet} simplifies to $\exp\{-1/(12n(x_j'-b))+\OO(1/n^2)\}$.
 It turns out, in all four cases as in Figure \ref{f:newp}, the righthand side of \eqref{e:split} is of order $\OO(1)$:
If $x_j'\geq a$, up to errors of order $\OO(1/n)$, the ratios are given by
\begin{align*}
\prod_{j=1}^{m'}e^{-\frac{1}{12 n(x_{j}'-b)} }, \quad
\prod_{j=1}^{m'}e^{-\frac{1}{12 n(x_{j}'-b)}+\frac{1}{12 n(x_{j}'-a)} },\quad
\prod_{j=1}^{m'}e^{-\frac{1}{12 n(x_{j}'-b)}+\frac{1}{12 n(x_{j}'-a)} },\quad
\prod_{j=1}^{m'}e^{-\frac{1}{12 n(x_{j}'-b)}};
\end{align*}
If $x_j'<b$, up to errors of order $\OO(1/n)$, the ratios are given by
\begin{align*}
\prod_{j=1}^{m'}e^{-\frac{1}{12 n(a-x_{j}')}+\frac{1}{12 n(b-x_{j}')} },\quad
\prod_{j=1}^{m'}e^{-\frac{1}{12 n(a-x_{j}')}+\frac{1}{12 n(b-x_{j}')} },\quad
\prod_{j=1}^{m'}e^{-\frac{1}{12 n(a-x_{j}')} }, \quad
\prod_{j=1}^{m'}e^{-\frac{1}{12 n(a-x_{j}')}}.
\end{align*}

\section{Equation of the Correction Term}\label{s:eq}
In this Section, we derive the difference equation for the correction term $E_{s}(\bmx)$ defined in Ansatz \ref{a:defE}. 
We can rewrite the recursion \eqref{e:Wt+recur} as a recursion for $E_{s}(\bmx)$: for any time $s\in[0,T)\cap \bZ_{n}$, and $\bmx=(x_1,x_2,\cdots,x_m)\in \fM_s(\fP)$
\begin{align}\label{e:cEeq}
E_{s}(\bmx)
=\sum_{\bme\in\{0,1\}^m}a_s(\bme;\bmx) E_{(s+1/n)-}(\bmx+\bme/n),
\end{align}
where
\begin{align}\label{e:defat}
a_s(\bme;\bmx)=\frac{V(\bmx+\bme/n)}{V(\bmx)}\prod_{j=1}^m\left(
\prod_{a\in A_{s}(x_j)}  (x_j-a)^{1-e_j}
\prod_{b\in B_{s}(x_j)} (b-x_j-1/n)^{e_j}\right)e^{n^2(Y_{(s+1/n)-}(\bmx+\bme/n)-Y_{s}(\bmx))}.
\end{align}
We denote the partition function as
\begin{align*}
Z_s(\bmx)\deq \sum_{\bme\in\{0,1\}^m}a_s(\bme;\bmx)
\end{align*}

The equation \eqref{e:cEeq} can be solved using Feynman-Kac formula. The polygon $\fP$ contains horizontal edges at times $0=T_0<T_1<T_2<\cdots<T_p=T$. From the discussion at the end of Section \ref{s:ansatz}, for any $s\not\in\{T_0,T_1,\cdots, T_p\}$, we have $E_{s}(\bmx)= E_{s-}(\bmx')$ and for $s\in \{T_1,T_2,\cdots, T_{p-1}\}$, we have explicit formula \eqref{e:etratio} for the ratio $ E_{s-}(\bmx')/E_{s}(\bmx')$. For any time $s\in[0,T)\cap \bZ_{n}$, to solve for $E_{s}(\bmx)$, we construct a Markov process $\{\bmx_t\}_{t\in[s, T]\cap \bZ_n}$ starting from the configuration $\bmx_s=\bmx$ with a time dependent generator given by 
\begin{align}\label{e:gener}
\cL_{t} f(\bmx)=\sum_{\bme\in \{0,1\}^{m(t)}}\sum_{\bme \in \{0,1\}^{m(t)}}\frac{a_t(\bme;\bmx)}{Z_t(\bmx)}(f(\bmx+\bme/n)-f(\bmx)),
 \quad \bmx=(x_1, x_2, \cdots, x_{m(t)})\in \fM_{t}(\fP).
\end{align}

We remark that the boundary condition of $\fP$ is encoded in $a_t(\bme;\bmx)$ as in \eqref{e:defat}. If a particle $x_j$ is at the left boundary of $\fP$ with slope $1$ corresponding to $a$, i.e. $x_j=a$. It has to jump to the right, otherwise $e_j=0$ and $a_t(\bme;\bmx)=0$. If a particle $x_j$ is at the right boundary of $\fP$ with slope $\infty$ corresponding to $b$, i.e. $x_j=b-1/n$. It has to stay, otherwise $e_j=1$ and $a_t(\bme;\bmx)=0$.
Thus for nonintersecting Bernoulli random walks in the form \eqref{e:gener}, particles are constrained inside $\fP$.

Using the Markov process $\{\bmx_t\}_{t\in[s, T]\cap \bZ_n}$ defined above, $E_{s}$ can be solved using the Feynman-Kac formula.
\begin{proposition}\label{p:FK}
For any time $s\in[0,T)\cap \bZ_n$ and $\bmx\in \fM_{s}(\fP)$, $E_{s}(\bmx)$ is given by
\begin{align}\begin{split}\label{e:FK}
E_{s}(\bmx)
&=\bE\left[\left.\prod_{t\in[s,T)\cap \bZ_n}Z_t(\bmx_t)
\prod_{k:s<T_k<T}\frac{E_{T_k-}(\bmx_{T_k}')}{ E_{T_k}(\bmx_{T_k})}
\right| \bmx_s=\bmx\right].
\end{split}\end{align}
\end{proposition}
\begin{proof}
We prove \eqref{e:FK} by induction on $r\in [s,T)\cap\bZ_n$, that
\begin{align}\begin{split}\label{e:indcE}
&\phantom{{}={}}\bE\left[\left.\prod_{t\in [s,T)\cap \bZ_n}Z_t(\bmx_t)
\prod_{k: s<T_k<T}\frac{E_{T_k-}(\bmx_{T_k}')}{ E_{T_k}(\bmx_{T_k})}
\right| \bmx_s=\bmx\right]\\
&=\bE\left[\left.\prod_{t\in [s,r]\cap\bZ_n}Z_t(\bmx_t) E_{(r+1/n)-}(\bmx_{r+1/n}')
\prod_{k: s<T_k\leq r}\frac{E_{T_k-}(\bmx_{T_k}')}{ E_{T_k}(\bmx_{T_k})}
\right| \bmx_s=\bmx\right]
\end{split}\end{align}
The statement holds trivially for $r=T-1/n$, since $ E_{T-}(\bmx_{T}')=1$. We assume the statement \eqref{e:indcE} holds for $r$ and prove it for $r-1/n$.
We can rewrite the righthand side of \eqref{e:indcE} as 
\begin{align}\begin{split}\label{e:rarr}
&\phantom{{}={}}\prod_{t\in [s,r]\cap \bZ_n}Z_t(\bmx_t) E_{(r+1/n)-}(\bmx'_{r+1/n})
\prod_{k: s<T_k\leq r}\frac{E_{T_k-}(\bmx_{T_k}')}{E_{T_k}(\bmx_{T_k})}\\
&=
\prod_{t\in [s,r-1/n]\cap \bZ_n}Z_t(\bmx_t)
\prod_{k: s<T_k\leq r}\frac{E_{T_k-}(\bmx_{T_k}')}{ E_{T_k}(\bmx_{T_k})}
\left( Z_r(\bmx_r) E_{(r+1/n)-}(\bmx_{r+1/n}')\right).
\end{split}\end{align}
To take the expectation conditioning on $\bmx_s=\bmx$, we can first take expectation conditioning on $\bmx_r$. Using \eqref{e:cEeq}, we have
\begin{align*}\begin{split}
&\phantom{{}={}}Z_r(\bmx_r) E_{(r+1/n)-}(\bmx_{r+1/n}')- E_{r}(\bmx_r)\\
&=Z_r(\bmx_r)\left( E_{(r+1/n)-}(\bmx_{r+1/n}')
-\sum_{\bme\in\{0,1\}^{m(r)}}\frac{a_{r}(\bme, \bmx_{r})}{Z_{r}(\bmx_{r})} E_{(r+1/n)-}(\bmx_{r}+\bme/n)\right),
\end{split}\end{align*}
which has mean zero conditioning on $\bmx_{r}$.
Therefore, after taking expectation conditioning on $\bmx_r$, \eqref{e:rarr} becomes
\begin{align}\label{e:rarr2}
\prod_{t\in [s,r-1/n]\cap Z_n}Z_t(\bmx_t) E_{r}(\bmx_r)
\prod_{k: s<T_k\leq r}\frac{E_{T_k-}(\bmx_{T_k}')}{ E_{T_k}(\bmx_{T_k})}.
\end{align}
If $r\not\in \{T_0, T_1,T_2,\cdots, T_p\}$, then $ E_r(\bmx_r)= E_{r-}(\bmx_r')$; if  $r\in \{T_0, T_1,T_2,\cdots, T_p\}$, then 
\begin{align*}
 E_{r}(\bmx_r)
\prod_{k: s<T_k\leq r}\frac{E_{T_k-}(\bmx_{T_k}')}{ E_{T_k}(\bmx_{T_k})}
= E_{r}(\bmx_r)
\prod_{k: s<T_k\leq r-1/n}\frac{E_{T_k-}(\bmx_{T_k}')}{ E_{T_k}(\bmx_{T_k})}\frac{ E_{r-}(\bmx_{r}')}{ E_{r}(\bmx_{r})}
= E_{r-}(\bmx'_r)
\prod_{k: s<T_k\leq r-1/n}\frac{E_{T_k-}(\bmx_{T_k}')}{ E_{T_k}(\bmx_{T_k})}
\end{align*}
In both cases, we can rewrite \eqref{e:rarr2} as
\begin{align*}
\prod_{t\in [s,r-1/n]\cap Z_n}Z_t(\bmx_t) E_{r-}(\bmx_r')
\prod_{k: s<T_k\leq r-1/n}\frac{E_{T_k-}(\bmx_{T_k}')}{ E_{T_k}(\bmx_{T_k})}.
\end{align*}
Then we can take expectation conditioning on $\bmx_s=\bmx$ to get
\begin{align*}
&\phantom{{}={}}\bE\left[\left.\prod_{t\in [s,T)\cap \bZ_n}Z_t(\bmx_t)
\prod_{k: s<T_k<T}\frac{E_{T_k-}(\bmx_{T_k}')}{E_{T_k+}(\bmx_{T_k})}
\right| \bmx_s=\bmx\right]\\
&=\bE\left[\left.\prod_{t\in [s,r-1/n]\cap \bZ_n}Z_t(\bmx_t) E_{r-}(\bmx_r')
\prod_{k: s<T_k\leq r-1/n}\frac{E_{T_k-}(\bmx_{T_k}')}{E_{T_k+}(\bmx_{T_k})}
\right| \bmx_s=\bmx\right]
\end{align*}
This gives the statement \eqref{e:indcE} for $r-1/n$ and finishes our induction. The claim \eqref{e:FK} follows by taking $r=s$ in \eqref{e:indcE}.

\end{proof}
%
%

To use Proposition \ref{p:FK} to solve for the correction terms $E_{s}$, we need to understand the Markov process $\{\bmx_t\}_{t\in [s,T]\cap \bZ_n}$ with generator $\cL_t$ as given in \eqref{e:gener}. In the weights $a_s(\bme,\bmx)$ as in \eqref{e:defat}, all the terms are explicit, except for $Y_{(s+1/n)-}(\bmx+\bme/n)-Y_s(\bmx)$. In the rest of this section, we derive an $1/n$-expansion of $Y_{(s+1/n)-}(\bmx+\bme/n)-Y_s(\bmx)$. 

We recall the variational problem from \eqref{e:varWt} 
\begin{align}\label{e:varWtcopy}
W_ s ( \beta )=\sup_{ h\in H({\mathfrak P};\beta, s)}\int_ s ^T\int_\bR\sigma(\nabla  h)\rd x \rd t ,
\end{align}
where the set of height functions $H({\mathfrak P};\beta, s)$ is defined in \eqref{def:Hpb}. We denote the minimizer of \eqref{e:varWtcopy} as $\{h_t^*(\cdot;\beta,s)\}_{s\leq t\leq T}$ and the complex slope $ f_t( x ;\beta, s)$
\begin{align}\label{e:ft3}
(\del_{x}  h^*, \del_{t}  h^*)=\frac{1}{\pi}\left(-\arg^*( f_t( x ;\beta, s)),\arg^*( f_t( x ;\beta, s)+1)\right).
\end{align}
Then $ f_t( x ;\beta, s)$ satisfies the complex Burgers equation
\begin{align}\label{e:burgereq2}
 \frac{\del_{t}  f_t( x ;\beta, s)}{ f_t( x ;\beta, s)}+\frac{\del_{x}  f_t( x ;\beta, s)}{ f_t( x ;\beta, s)+1}=0,
\end{align}
for $(x,t)\in\Omega({\mathfrak P};\beta,s)$ in the liquid region:
$
\Omega({\mathfrak P};\beta,s)\deq \{ (x,t): 0<\del_{x}  h^*<1, s<t<T\}.
$

\begin{proposition}\label{p:Ws}
The functional derivative of $W_s(\beta)$ is given by
\begin{align}\label{e:dbW}
\frac{\del W_ s ( \beta )}{\del  \beta }=-\ln | f_s( x ;\beta, s)|=-\ln (g_s(x;\beta,s)) H(\del_x \beta),
\end{align}
where $H(\del_x \beta)$ is the Hilbert transform of $\del_x \beta$.
The time derivative of the functional $W_s(\beta)$ is given by
\begin{align}\begin{split}\label{e:dsW0}
\del_s W_s(\beta)
&=\int  \ln|f_s(x;\beta,s)| \del_s h_s^*(x;\beta,s)\rd  x  - \int\sigma(\del_x \beta, \del_s h_s^*(x;\beta,s))\rd  x.
\end{split}\end{align}

\end{proposition}

\begin{proof}
We can perturb the functional $W_s$, for any other boundary height function $\beta +\delta  \beta$, we denote the corresponding minimizer of \eqref{e:varWtcopy} as $h^*+ \delta h$, then
\begin{align*}\begin{split}
W_ s ( \beta +\delta  \beta )
&=\int_ s ^T\int \sigma(\nabla  h^*+\nabla \delta h)\rd  x \rd  t \\
&=\int_ s ^T\int  \frac{\del\sigma(\nabla  h^*)}{\del(\del_{x} h^*)}\del_{x}\delta  h\rd  x \rd  t +\frac{\del\sigma(\nabla  h^*)}{\del(\del_{t} h^*)}\del_{t}\delta  h\rd  x \rd  t +\OO((\delta  h)^2)\\
&=-\int_ s ^T\int  \del_{x}\frac{\del\sigma(\nabla  h^*)}{\del(\del_{x} h^*)}\delta  h\rd  x \rd  t +\del_{t}\frac{\del\sigma(\nabla  h^*)}{\del(\del_{t} h^*)}\delta  h\rd  x \rd  t 
+\left.\int\frac{\del\sigma(\nabla  h^*)}{\del(\del_{t} h^*)}\delta  h\rd  x \right|_{ s }^{T}+\OO((\delta  h)^2)\\
&=-\int \ln| f_{ s }( x ;\beta, s)|\delta  \beta ( x )\rd  x +\OO((\delta  h)^2).
\end{split}\end{align*}
It follows that the functional derivative of $W_s(\beta)$ with respect to the boundary height function $\beta$ is given by $-\ln |f_s(x;\beta,s)|$, and \eqref{e:dbW} follows.

We recall the definition \eqref{e:ft3} of $ f_s( x ;\beta, s)$,
\begin{align}\label{e:fim}
\arg^*( f_{ s }( x ;\beta, s))=-\pi\del_{x} h^*=-\pi\del_{x}  \beta ( x ).
\end{align}
Using \eqref{e:dbW} and \eqref{e:fim}, we can write $f_s(x;\beta,s)$ explicitly in terms of $\del_x \beta(x)$ and $\del_\beta W_s(\beta)$.
\begin{align}\label{e:fnewexp}
 f_{ s }( x ;\beta, s)=e^{-\ri \pi \del_{x}  \beta ( x )-\del_\beta W_s(\beta)}
=e^{-\ri \pi \del_x \beta+H(\del_x \beta)-\del_\beta W_s(\beta)-H(\del_x \beta)},
\end{align}
where $H(\del_x \beta)$ is the Hilbert transform of the measure $\del_x \beta$.
We notice that $-\ri \pi \del_x \beta+H(\del_x \beta)$ is the Stieltjes transform of  the measure $\del_x \beta$. It follows by comparing with \eqref{e:gszmut}, we conclude that 
\begin{align}\label{e:gss}
g_s(x;\beta,s)=e^{-\del_\beta W_s(\beta)-H(\del_x \beta)},\quad x\in \supp(\del_x \beta).
\end{align}
Moreover, with this notation, it holds that
\begin{align*}
|f_s(x;\beta,s)|=g_s(x;\beta,s)e^{H(\del_x \beta)},\quad x\in \supp(\del_x \beta).
\end{align*}

In the following, we compute the time derivative of $W_s(\beta)$. Let $h^*+\delta h$ be the minimizer of \eqref{e:varWtcopy} starting at time $s+\Delta s$, then
\begin{align*}\begin{split}
W_ {s+\Delta s} ( \beta )-W_ {s} ( \beta )
&=\int_ {s+\Delta s} ^T\int \left(\sigma(\nabla  h^*+\nabla \delta h)-\sigma(\nabla  h^*)\right)\rd  x \rd  t -\int_{s}^{s+\Delta s} \int\sigma(\nabla  h^*)\rd  x \rd  t\\
&=\left. \int  \frac{\del\sigma(\nabla  h^*)}{\del(\del_{t} h^*)} \delta  h\rd  x \right |_{s+\Delta s}^T -\Delta s \left.\int\sigma(\nabla  h^*)\rd  x\right|_{t=s}+\OO((\Delta s)^2)\\
&=\Delta s\left. \int  \frac{\del\sigma(\nabla  h^*)}{\del(\del_{t} h^*)} \del_t h^*\rd  x \right |_{t=s} -\Delta s \left.\int\sigma(\nabla  h^*)\rd  x\right|_{t=s}+\OO((\Delta s)^2)\\
\end{split}\end{align*}
It follows that 
\begin{align}\begin{split}\label{e:dsW}
\del_s W_s(\beta)
&=\int  \frac{\del\sigma(\nabla  h^*)}{\del(\del_{t} h^*)} \del_s h_s^*(x;\beta,s)\rd  x  - \int\sigma(\nabla  h_s^*(x;\beta,s))\rd  x\\
&=\int  \ln|f_s(x;\beta,s)| \del_s h_s^*(x;\beta,s)\rd  x  - \int\sigma(\del_x \beta, \del_s h_s^*(x;\beta,s))\rd  x\\
&=\int  (\ln g_s(x;\beta,s)+H(\del_x \beta)) \del_s h_s^*(x;\beta,s)\rd  x  - \int\sigma(\del_x \beta, \del_s h_s^*(x;\beta,s))\rd  x.
\end{split}\end{align}
This finishes the proof of \eqref{e:dsW0}.
\end{proof}

We recall that  the poles $a=z(X_i^{(1)})\in A_s(x)$ travel at rate $1$, i.e.
$\del_s a=1$, and the zeros $b=z(X_i^{(\infty)})\in B_s(x)$ does not move, i.e. $\del_s b=0$.  From the definition of $Y_s(\bmx)$ as in \eqref{e:defYt+}, we have
\begin{align}\label{e:dsdY}
\del_s Y_s(\bmx)=\del_s W_s(\bmx)-\int \sum_{a\in A_{s}(x)}\bm1(x\geq a)  \ln(x-a)\rho(x;\bmx)\rd x.
\end{align}
For the functional derivative of $Y_s(\beta)$ with respect to $\beta$, using \eqref{e:dbW} we have
\begin{align}\begin{split}\label{e:deftg}
\frac{\del Y_s(\beta)}{\del \beta}
&=\frac{\del W_s(\beta)}{\del \beta}+H(\del_x \beta)-\sum_{a\in A_{s}(x)} \ln(x-a)+\sum_{b\in B_{s}(x)}\ln(b-x)\\
&=-\ln g_s(x;\beta,s)-\sum_{a\in A_{s}(x)} \ln(x-a)+\sum_{b\in B_{s}(x)}\ln(b-x)=:-\ln  g_s'(x;\beta,s).
\end{split}\end{align}
More generally, for any $s\leq t\leq T$, we define 
\begin{align}\label{e:deftgt}
\ln  g_t'(x;\beta,s)\deq\ln g_t(x;\beta,s)+\sum_{a\in A_{t}(x)} \ln(x-a)-\sum_{b\in B_{t}(x)}\ln(b-x).
\end{align}
$g_t'(x;\beta,s)$ extends to a meromorphic function in a neighborhood of $z^{-1}(\{x:(x,t+)\in \fP\})$ over $\gamma_t^{\beta,s,+}$. From  the discussion of zeros and poles of $g_t(Z;\beta,s)$ at the end of Section \ref{s:generalb}, and the definition of $A_t(x), B_t(x)$ as in \eqref{e:defAB}, we know that $\ln  g_t'(Z;\beta,s)$ is analytic in a neighborhood of $z^{-1}(\{x:(x,t+)\in \fP\})$ over $\gamma_t^{\beta,s,+}$. And it does not have zeros or poles.
In the specially case for $t=s$, corresponding to \eqref{e:deftg}, \eqref{e:ddlogg} gives
\begin{align}\label{e:deeKt}
\del_\beta \ln  g_s'(x;\beta,s)=\del_\beta\ln  g_s(x;\beta,s)=K(x,y;\beta,s)= \left(\frac{1}{(x-y)^2}-\frac{B(X,Y;\beta,s)}{\rd z(X)\rd z(Y)}\right),
\end{align}
where $X=z^{-1}(x)$ and $Y=z^{-1}(y)$.
The time derivative of $\ln  g_t'(x;\beta,s)$ is 
\begin{align*}
\del_t \ln  g_t'(x;\beta,s)=\del_t \ln g_t(x;\beta,s)+\sum_{a\in A_t(x)}\frac{1}{a-x},
\end{align*}
which also extends to analytic function in a neighborhood of $z^{-1}(\{x:(x,t+)\in \fP\})$ over $\gamma_t^{\beta,s,+}$.
%

\begin{proposition}\label{p:YYdif}
Let $\beta=\beta(x;\bmx)$, and the indicator function $\Lambda(x)=x$ for $0\leq x\leq 1/n$; $\Lambda(x)=2/n-x$ for $1/n\leq x\leq 2/n$; $\Lambda(x)=0$ elsewhere. Then
\begin{align}\label{e:YYdif}
n^2(Y_{(s+1/n)-}(\bmx+\bme/n)-Y_{s}(\bmx))=\sum_i e_i \chi(x_i;\beta,s)+\frac{1}{n^2}\sum_{ij} e_i e_j \kappa(x_i,x_j;\beta,s) +\zeta(\beta,s)+\OO(1/n),
 \end{align}
 where
 \begin{align}\begin{split}\label{e:defka}
\chi(x_i;\beta,s)&=
n^2\int \ln  g_s'(x;\beta,s)\Lambda(x-x_i)\rd x+n\int \del_s(\ln  g_s'(x;\beta,s))\Lambda(x-x_i)\rd x\\
\kappa(x_i,x_j;\beta,s)&=-\frac{n^4}{2}\int K(x,y;\beta,s)\Lambda(x-x_i)\Lambda(y-x_j)\rd x\rd y.
\end{split} \end{align}
the functions $ g_s'(x;\beta,s), K(x,y;\beta,s)$ are defined in \eqref{e:deftgt}, \eqref{e:deeKt}, and
 \begin{align}\label{e:defzeta}
 \zeta(\beta,s)=(n\del_s Y_s(\beta)+\frac{1}{2}\del_s^2 Y_s(\beta)).
 \end{align}
\end{proposition}
\begin{proof}

For simplification of notations, we write $\beta=\beta(x;\bmx)$ and 
\begin{align}\label{e:defLa}
\delta \beta=\beta(x;\bmx+\bme/n)-\beta(x;\bmx)
=-\sum_{i}e_i \Lambda(x-x_i). 
\end{align}
$Y_s(\beta)$ depends smoothly on $s$ and $\beta$, and the derivatives are given in \eqref{e:dsdY} and \eqref{e:deftg}.
By a Taylor expansion we get
\begin{align}\begin{split}
\label{Y:diff}
&\phantom{{}={}}Y_{(s+1/n)-}(\beta+\delta \beta)-Y_{s}(\beta)
=
\frac{1}{n}\del_sY_{s}(\beta)
+\int \del_\beta Y_s(\beta)\delta\beta \rd x\\
&+\frac{1}{2n^2}\del^2_sY_{s}(\beta)
+\frac{1}{n}\int \del_s\del_\beta Y_{s}(\beta)\delta\beta\rd x
+\frac{1}{2}\int \del^2_\beta Y_{s}(\beta)\delta\beta\delta \beta \rd x\rd y+\OO(1/n^3)
\end{split}\end{align}
In the following, we give explicit expressions of the terms on the righthand side of \eqref{Y:diff}. For $\int \del_\beta Y_s(\beta)\delta\beta$, using \eqref{e:deftg}, we have
\begin{align}
\begin{split}\label{e:yy1}
\int \del_\beta Y_s(\beta)\delta\beta\rd x
&=\sum_{i=1}^m e_i\int \ln  g_s'(x;\beta,s)\Lambda(x-x_i)\rd x\\
&=\sum_{i=1}^m \frac{e_i}{n^2}\left(\ln  g_s'(x_i;\beta,s)
+\frac{\del_x \ln  g_s'(x;\beta,s)}{n} +\OO(1/n^2)\right).
\end{split}
\end{align}
For $(1/n)\int \del_s\del_\beta Y_{s}(\beta)\delta\beta \rd x
$, we have
\begin{align}\label{e:yy2}
\frac{1}{n}\int \del_s\del_\beta Y_{s}(\beta)\delta\beta \rd x
=\sum_{i=1}^m\frac{e_i}{n}\int \del_s(\ln  g_s'(x;\beta,s))\Lambda(x-x_i)\rd x.
\end{align}
For $(1/2)\int \del^2_\beta Y_{s}(\beta)\delta\beta\delta \beta \rd x\rd y$, we have
\begin{align}\label{e:yy3}
\frac{1}{2}\int \del^2_\beta Y_{s}(\beta)\delta\beta\delta\beta
=\sum_{i,j}\iint\frac{e_ie_j}{n^2}K(x,y;\beta,s)\Lambda(x-x_i)\Lambda(y-x_j)\rd x\rd y.
\end{align}
The claim \eqref{e:YYdif} follows from plugging \eqref{e:yy1}, \eqref{e:yy2} and \eqref{e:yy3} into \eqref{e:YYdif}.
\end{proof}

Using Proposition \ref{p:YYdif}, we can rewrite $a_s(\bme;\bmx)$ in the following form
\begin{align}\label{e:newat0}
a_s(\bme;\bmx)&=\frac{V(\bmx+\bme/n)}{V(\bmx)}\prod_{j} \phi^+(x_j;\beta,s)^{e_i}\phi^-(x_j;\beta,s)^{1-e_i}
e^{\frac{1}{n^2}\sum_{i} e_i e_j \kappa(x_i,x_j;\beta,s) +\zeta(\beta,s)+\OO(1/n)},
\end{align}
where $\beta=\beta(x;\bmx)$, and 
\begin{align*}
\phi^+(x;\beta,s)=\prod_{b\in B_{s}(x)} (b-x-1/n)e^{ \chi(x;\beta,s)},\quad 
\phi^-(x;\beta,s)=\prod_{a\in A_{s}(x)}  (x-a).
\end{align*}

\section{Discrete Loop Equation}
\label{s:loopeq}

%
%
%

In this section, we study the nonintersecting random Bernoulli walk with transition probability in the following form: 
Fix $\bmx=(x_1, x_2, \cdots, x_m)\in \bZ_n^m$, and its empirical measure satisfies
\begin{align*}
\rho(x;\bmx)=\frac{1}{n}\sum_{i=1}^m \bm1(x\in [x_i, x_i+1/n]), \quad \supp(\rho(x;\bmx))\subset [\fa_1, \fb_1]\cup[\fa_2, \fb_2]\cup\cdots [\fa_r, \fb_r].
\end{align*}
There are meromorphic functions  $\phi^+(Z), \phi^-(Z), \kappa(Z,W)$  over certain Riemann surface $\gamma$ equipped with the covering map $Z=(f,z)\in\gamma\mapsto z\in\bC\bP^1$, and $\kappa(Z,W)=\kappa(W,Z)$ is symmetric. The intervals $[\fa_i, \fb_i]$ are lifted to $\gamma$, and we denote it by $z^{-1}([\fa_1, \fb_1]\cup[\fa_2, \fb_2]\cup\cdots [\fa_r, \fb_r])$. We simply write $\phi^+(x_i)=\phi^+(z^{-1}(x_i)), \phi^-(x_i)=\phi^-(z^{-1}(x_i)), \kappa(x_i,x_j)=\kappa(z^{-1}(x_i),z^{-1}(x_j))$. In the applications we will take $\bmx\in \fM_s(\fP)$ as in definition \ref{def:Mt}, then its empirical measure satisfies $\supp(\rho(x;\bmx))\subset [\fa_1(s), \fb_1(s)]\cup[\fa_2(s), \fb_2(s)]\cup\cdots [\fa_{r(s)}(s), \fb_{r(s)}(s)]$, and $\gamma$ to be $\gamma_s^{\beta,s,+}$ as defined in \eqref{e:zmatr}, 

For any $\bme=\{0,1\}^m$, the transition probability $\bP_\bme$ is given by
\begin{align}\label{e:defL}
\bP_\bme= \frac{a_\bme}{Z}\deq\frac{1}{Z}\frac{V(\bmx+\bme/n)}{V(\bmx)}\prod_{i=1}^m\phi^+(x_i)^{e_i}\phi^-(x_i)^{1-e_i}\exp\left\{\sum_{1\leq i,j\leq m}\frac{e_ie_j}{n^2}\kappa(x_i,x_j)\right\}.
\end{align}
We make the following assumptions
\begin{assumption}\label{a:weight}
We assume the weights $\phi^\pm(Z)$ satisfy
\begin{enumerate}
\item In a neighborhood $\Upsilon$ of  $z^{-1}([\fa_1, \fb_1]\cup[\fa_2, \fb_2]\cup\cdots [\fa_r, \fb_r])$ over $\gamma$, the covering map $\gamma\mapsto \bC\bP^1$ does not have a branch point, and $\phi^\pm(Z), K(Z,W)$ are analytic. 
\item $\phi^\pm(Z)$ can be decomposed as
\begin{align}\label{e:defvphi}
\phi^\pm(z)=\varphi_\pm(z)e^{\frac{1}{n}\psi^\pm(z)}.
\end{align}
where $\psi^\pm(z)$ are analytic and uniformly bounded in a neighborhood of $z^{-1}([\fa_1, \fb_1]\cup[\fa_2, \fb_2]\cup\cdots [\fa_r, \fb_r])$ over $\gamma$.
\end{enumerate}
\end{assumption}
Under Assumption \ref{a:weight}, in a neighborhood $\Upsilon$ of  $z^{-1}([\fa_1, \fb_1]\cup[\fa_2, \fb_2]\cup\cdots [\fa_r, \fb_r])$ over $\gamma$, the covering map $\gamma\mapsto \bC\bP^1$ does not have a branch point. By taking $\Upsilon$ small enough, $\gamma\mapsto \bC\bP^1$ restricted on $\Upsilon$ is a bijection. In this way we can identify $\Upsilon$ with a neighborhood of $[\fa_1, \fb_1]\cup[\fa_2, \fb_2]\cup\cdots [\fa_r, \fb_r]$ over $\bC$. In the rest of this Section, we use this identification, and view $\Upsilon$ as a subset of $\bC$.

We remark that the weights $a_s(\bme;\bmx)$ in \eqref{e:newat0} are in the form of \eqref{e:defL} with 
\begin{align*}
\phi^+(z)=e^{\chi(z;\beta,s)}\prod_{b\in B_{s}(z)} (b-z-1/n)^{e_j},\quad
\phi^-(z)=\prod_{a\in A_{s}(z)}  (z-a),\quad
\kappa(z,w)=\kappa(z,w;\beta,s),
\end{align*}
 and they satisfy Assumption \ref{a:weight}.

The main goal of this section is to understand the difference of the empirical measures $\sum_i\delta_{x_i+e_i/n}$ and $\sum_i\delta_{x_i}$ under the probability \eqref{e:defL}:
\begin{align}\label{e:mun}
\mu_n=\sum_i e_i \left(\delta_{x_i+e_i/n}-\delta_{x_i}\right).
\end{align}

Our analysis of the transition probability \eqref{e:defL} is based on the following lemma. 
\begin{lemma}\label{l:loopeq}
Let $a_\bme$ as defined in \eqref{e:defL}. Under assumption \ref{a:weight}, the following function is analytic on $\Upsilon$,
\begin{align}\label{e:sum1}
\sum_{\bme}a_\bme\left(\prod_{i=1}^m\frac{z-x_i+(1-e_i)/n}{z-x_i}\phi^+ \exp\left\{\sum_{i=1}^m \frac{e_i}{n^2}\kappa(x_i, z)\right\}+\prod_{i=1}^m\frac{z-x_i-e_i/n}{z-x_i}\phi^-(z)\right)
\end{align}
\end{lemma}

\begin{remark}
Lemma \ref{e:sum1} can be viewed as a dynamical version of the discrete loop equations introduced in \cite{MR3668648}, which is crucial for the proof of the central limit theorems of the discrete $\beta$ ensembles. 
\end{remark}

\begin{proof}
Under Assumption \ref{a:weight}, the only possible poles of \eqref{e:sum1} are $z^{-1}(x_k)$, for $1\leq k\leq m$. In the following we check that the residual at each $z^{-1}(x_k)$ is zero.

For any $1\leq k\leq m$, let $\bmx^{(k)}=(x_1, x_2,\cdots, x_{k-1}, x_{k+1},\cdots, x_m)$, $\bme^{(k)}=(e_1, e_2,\cdots, e_{k-1}, e_{k+1},\cdots, e_m)\in \{0,1\}^{m-1}$, and 
\begin{align*}
P_{\bme^{(k)}}=\frac{V(\bmx^{(k)}+\bme^{(k)}/n)}{V(\bmx^{(k)})}\prod_{i: i\neq k} \phi^+(x_i)^{e_i}\phi^-(x_i)^{1-e_i}\exp\left\{\sum_{i,j: i,j\neq k}\frac{e_ie_j}{n^2}\kappa(x_i,x_j)\right\}.
\end{align*}
The residual of \eqref{e:sum1} at $z^{-1}(x_k)$ is: for $e_k=1$ 
\begin{align*}
-\frac{1}{n}P_{\bme^{(k)}}\exp\left\{\sum_{i=1}^m \frac{e_i}{n^2}\kappa(x_i, x_k)\right\} \prod_{i:i\neq k}\frac{x_k-x_i+(1-e_i)/n}{x_k-x_i}\phi^-(x_k)\prod_{i:i\neq k}\frac{x_k-x_i-e_i/n}{x_k-x_i},
\end{align*}
and for $e_k=0$
\begin{align*}
\frac{1}{n}P_{\bme^{(k)}}\prod_{i:i\neq k}\frac{x_k-x_i-e_i/n}{x_k-x_i}\phi^+(x_k)\prod_{i:i\neq k}\frac{x_k-x_i+(1-e_i)/n}{x_k-x_i}\exp\left\{\sum_{i=1}^m \frac{e_i}{n^2}\kappa(x_i, x_k)\right\}.
\end{align*}
They cancel out. Therefore, $z^{-1}(x_k)$ is not a pole of \eqref{e:sum1}. This finishes the proof of Lemma \ref{l:loopeq}.
\end{proof}

We will use Lemma \ref{l:loopeq} to analyze the following quantities related to the transition probability as in \eqref{e:defL}. For any $z\in \Upsilon$ we define
\begin{align}\begin{split}\label{e:ABC}
\cA(z)&=\sum_{\bme}a_\bme\prod_{i}\frac{z-x_i-e_i/n}{z-x_i}=Z\bE_\bme\left[\prod_{i}\left(\frac{z-x_i-1/n}{z-x_i}\right)^{e_i}\right],\\
\cB(z)&=\cG(z)\varphi^+(z)+\varphi^-(z), \quad \quad \cG(z)=\prod_{i}\left(1+\frac{1/n}{z-x_i-1/n}\right)=\exp\left\{\int \frac{\rho(x;\bmx)}{z-x}\rd x\right\},\\
\cC(z)&=\sum_{\bme}a_\bme\left(\prod_{i=1}^m\frac{z-x_i+(1-e_i)/n}{z-x_i}\exp\left\{\sum_{i=1}^m \frac{e_i}{n^2}\kappa(x_i, z)\right\}\phi^+(z)+\prod_{i=1}^m\frac{z-x_i-e_i/n}{z-x_i}\phi^-(z)\right).
\end{split}\end{align}

Among the quantities $\cA(z), \cB(z), \cC(z)$, $\cA(z)$ encodes the information of the measure $\mu_n$, which we want to understand, $\cC(z)$ encodes the information of the partition function for the measure $\bP_\bme$, and $\cB(z)$ is known explicitly, depending only on $\bmx$. In the rest of this section, we will solve $\cA(z)$ and $\cC(z)$ explicitly in terms of $\cB(z)$, more precisely as contour integrals of $\cB(z)$.
To do it, we need to make the following assumption of the quantity $\cB$ as in \eqref{e:ABC}, which is satisfied in all our applications.
\begin{assumption}\label{a:cancelB}
We assume that the quantity $\cB(z)$ as defined in \eqref{e:ABC} satisfes
\begin{align*}
\frac{1}{2\pi \ri}\oint_{\cin}\frac{\del_w \cB(w)}{\cB(w)}\frac{\rd w}{w-z}\sim \frac{1}{z^2},\quad z\rightarrow \infty,
\end{align*}
where the contour $\cin$ encloses $[\fa_1, \fb_1]\cup[\fa_2, \fb_2]\cup\cdots [\fa_r, \fb_r]$ but not $z$.
\end{assumption}

We rewrite the first part of $\cC(z)$ as
\begin{align*}\begin{split}
\phantom{{}={}}\sum_{\bme}a_\bme\prod_{i=1}^m\frac{z-x_i+(1-e_i)/n}{z-x_i}\phi^+(z)+\sum_{\bme}a_\bme\prod_{i=1}^m\frac{z-x_i+(1-e_i)/n}{z-x_i}\left(e^{\sum_{i=1}^m \frac{e_i}{n^2}\kappa(x_i, z)}-1\right)\phi^+(z)\\
=\cG(z+1/n)\cA(z+1/n)\phi^+(z)+\sum_{\bme}a_\bme\prod_{i=1}^m\frac{z-x_i+(1-e_i)/n}{z-x_i}\left(e^{\sum_{i=1}^m \frac{e_i}{n^2}\kappa(x_i, z)}-1\right)\phi^+(z),
\end{split}\end{align*}
where the second term is of order $\OO(Z/n)$,
\begin{align*}\begin{split}
&\phantom{{}={}}\sum_{\bme}a_\bme\prod_{i=1}^m\frac{z-x_i+(1-e_i)/n}{z-x_i}\left(e^{\sum_{i=1}^m \frac{e_i}{n^2}\kappa(x_i, z)}-1\right)\phi^+(z)\\
&=\sum_{\bme}a_\bme\prod_{i=1}^m\frac{z-x_i+(1-e_i)/n}{z-x_i}\left(\sum_{i=1}^m \frac{e_i}{n^2}\kappa(x_i, z)+\OO(\frac{1}{n^2})\right)\phi^+(z)=\OO\left(\frac{Z}{n}\right).
\end{split}\end{align*}
The second term in $\cC(z)$ is simply $\cA(z)\phi^-(z)$. Then in this way we have 
\begin{align}\label{e:ABCprecise}
\cC(z)=\cA(z)\cB(z)\left(1+\frac{\cE_0(z)}{n}\right)
\end{align}
where 
\begin{align}\begin{split}\label{e:defcE}
\cE_0(z)&=
\frac{n}{\cA(z)\cB(z)}\left(\cA(z)((\cG(z)\phi^+(z)+\phi^-(z))-\cB(z))+(\cG(z+1/n)\cA(z+1/n)-\cG(z)\cA(z))\phi^+(z)\right.\\
&\left.+\sum_{\bme}a_\bme\prod_{i=1}^m\frac{z-x_i+(1-e_i)/n}{z-x_i}\left(e^{\sum_{i=1}^m \frac{e_i}{n^2}\kappa(x_i, z)}-1\right)\phi^+(z)\right)=\OO(1).
\end{split}\end{align}

\subsection{First Order Term}
In this section we derive the first order terms of the quantities $\cA(z), \cC(z)$ as defined in \eqref{e:ABC} in terms of contour integrals of $\cB(z)$.
\begin{proposition}\label{p:firstod}
We assume Assumptions \ref{a:weight} and \ref{a:cancelB}, then
\begin{align*}
\frac{\del_z \cA(z)}{\cA(z)}
=\frac{1}{2\pi \ri}\oint_{\cin}\frac{\del_w \cB(w)}{\cB(w)}\frac{\rd w}{w-z}+\OO\left(\frac{1}{n}\right),
\end{align*}
where the contour $\cin$ encloses  $[\fa_1, \fb_1]\cup[\fa_2, \fb_2]\cup\cdots [\fa_r, \fb_r]$ but not $z$, and 
\begin{align*}
\frac{\del_{z} \cC(z)}{\cC(z)}
&=\frac{1}{2\pi \ri}\oint_{\cout}\frac{\del_w\cB}{\cB}\frac{\rd w}{w-z}+\OO\left(\frac{1}{n}\right),
\end{align*}
where the contour $\cout$ encloses $z$ and $[\fa_1, \fb_1]\cup[\fa_2, \fb_2]\cup\cdots [\fa_r, \fb_r]$.
\end{proposition}

\begin{proof}
We can rewrite \eqref{e:ABCprecise} in the following linear form
\begin{align}\label{e:dzC}
\frac{\del_z \cC(z)}{\cC(z)}=\frac{\del_z \cA(z)}{\cA(z)}+\frac{\del_z\cB(z)}{\cB(z)}+\frac{\del_z\cE_0(z)/n}{1+\cE_0(z)/n}.
\end{align}
From the defining relation \eqref{e:ABC}, we have that for $z\rightarrow \infty$
\begin{align*}
\frac{\del_z \cA(z)}{\cA(z)}\sim \frac{\OO(1)}{ z^2}.
\end{align*}
Thanks to Lemma \ref{l:loopeq}, $\cC(z)$ is analytic in a neighborhood of $[\fa_1, \fb_1]\cup[\fa_2, \fb_2]\cup\cdots [\fa_r, \fb_r]$, we can use a contour integral to get rid of $\del_z\cC(z)/\cC(z)$ and recover $\del_z \cA(z)/\cA(z)$,
\begin{align*}\begin{split}
-\frac{\del_z \cA(z)}{\cA(z)}&=\frac{1}{2\pi \ri}\oint_{\cin}\frac{\del_w \cA(w)}{\cA(w)}\frac{\rd w}{w-z}=-\frac{1}{2\pi \ri}\oint_{\cin}\left(\frac{\del_w \cB(w)}{\cB(w)}+\frac{\del_w \cE_0/n}{1+\cE_0/n}-\frac{\del_w \cC(w)}{\cC(w)}\right)\frac{\rd w}{w-z}\\
&=-\frac{1}{2\pi \ri}\oint_{\cin}\frac{\del_w \cB(w)}{\cB(w)}\frac{\rd w}{w-z}-\sum_{p: \cC(p)=0,\text{$p$ inside the contour}}\frac{1}{z-p}+\OO\left(\frac{1}{n}\right),
\end{split}\end{align*}
where the contour $\cin$ encloses $[\fa_1, \fb_1]\cup[\fa_2, \fb_2]\cup\cdots [\fa_r, \fb_r]$ but not $z$.
Therefore, by our assumption \ref{a:cancelB}, the lefthand side of 
\begin{align*}
\frac{\del_z \cA(z)}{\cA(z)}
-\frac{1}{2\pi \ri}\oint_{\cin}\frac{\del_z \cB(w)}{\cB(w)}\frac{\rd w}{w-z}=\sum_{p: \cC(p)=0,\text{$p$ inside the contour}}\frac{1}{z-p}+\OO\left(\frac{1}{n}\right),
\end{align*}
behaves like $\OO(1)/z^2$ for $z$ large. This implies that $\cC(z)$ does not have zeros inside the contour, otherwise, the righthand side behaves like $\Omega(1)/z$ for $z$ large. It follows that
\begin{align}\label{e:A1}
\frac{\del_z \cA(z)}{\cA(z)}
=\frac{1}{2\pi \ri}\oint_{\cin}\frac{\del_z \cB(w)}{\cB(w)}\frac{\rd w}{w-z}+\OO\left(\frac{1}{n}\right).
\end{align}

Since $\del_z \cA(z)/\cA(z)$ behaves like $\OO(1)/z^2$ when $z$ is large, we can also do a contour integral to kill $\del_z \cA(z)/\cA(z)$ and recover $\del_z \cC(z)/\cC(z)$,
\begin{align}
\frac{\del_{z} \cC(z)}{\cC(z)}\label{e:C1}
&=\frac{1}{2\pi \ri}\oint_{\cout}\frac{\del_w\cB(w)}{\cB(w)}\frac{\rd w}{w-z}+\OO\left(\frac{1}{n}\right),
\end{align}
where the contour $\cout$ encloses $[\fa_1, \fb_1]\cup[\fa_2, \fb_2]\cup\cdots [\fa_r, \fb_r]$ and $z$.
This finishes the proof of Proposition \ref{p:firstod}.
\end{proof}

\begin{proposition}\label{p:difmeasure}
We assume Assumptions \ref{a:weight} and \ref{a:cancelB}, then the measure $\mu_n$ as defined in \eqref{e:mun} satisfies
\begin{align*}
\bE_\bme\left[
\int  \frac{\rd \mu_n}{z-x}\right]
=\frac{1}{2\pi \ri}\oint_{\cin}\frac{\del_w\cB}{\cB}\frac{\rd w}{w-z}+\OO\left(\frac{1}{n}\right),
\end{align*}
where the contour $\cin$ encloses $[\fa_1, \fb_1]\cup[\fa_2, \fb_2]\cup\cdots [\fa_r, \fb_r]$ but not $z$.
\end{proposition}
This implies that there exists a limiting profile $\delta \beta^*$ such that 
\begin{align}\label{e:defbeta}
\frac{1}{2\pi \ri}\oint_{\cin}\frac{\del_w\cB}{\cB}\frac{\rd w}{w-z}=\int \frac{\rd \delta \beta^*(x)}{z-x},
\end{align}
and $\rd \mu_n$ converges to $\rd\delta \beta^*$ in expectation.
\begin{proof}
Fix any $z$, we deform the weights $a_\bme$, 
\begin{align}\label{e:tae}
\tilde a_\bme=a_\bme \prod\frac{z-x_i}{z-x_i-e_i/n}
=a_\bme \prod_i\left(\frac{z-x_i}{z-x_i-1/n}\right)^{e_i}.
\end{align}
This is equivalent to change $\phi^+(z')$ by a factor $(z-z')/(z-z'-1/n)$, which can be absorbed in $\psi^+(z')$:
\begin{align*}
\tilde \phi^+(z')
=\phi^+(z')\frac{z-z'}{z-z'-1/n}
=\varphi^+(z')e^{\frac{1}{n}\left(\psi^+(z')+n\ln \left(\frac{z-z'}{z-z'-1/n}\right)\right)}
=:\varphi^+(z')e^{\frac{1}{n}\tilde \psi^+(z')}.
\end{align*}
Since the leading order term in $\tilde \phi^+(z')$ is still $\varphi^+(z')$, 
we can compute using Proposition \ref{p:firstod}
\begin{align}\label{e:tA}
\tilde \cA(z')=\sum_\bme \tilde a_\bme \prod_i \frac{z'-x_i-e_i/n}{z'-x_i},
\quad
\tilde \cB(z')=\cG(z')\varphi^+(z')+\varphi^-(z'),
\end{align}
and 
\begin{align}\begin{split}\label{e:dtA}
\frac{\del_{z'}\tilde \cA(z')}{\tilde \cA(z')}
&=\frac{\sum_\bme \tilde a_\bme \prod_i \left(\frac{z'-x_i-1/n}{z'-x_i}\right)^{e_i}\left(\sum_i\frac{1}{z'-x_i-e_i/n}-\frac{1}{z'-x_i}\right)}{\sum_\bme \tilde a_\bme \prod_i \left(\frac{z'-x_i-1/n}{z'-x_i}\right)^{e_i}}\\
&=\frac{1}{2\pi \ri}\oint_{\cin}\frac{\del_w\cB(w)}{\cB(w)}\frac{\rd w}{w-z'}+\OO\left(\frac{1}{n}\right),
\end{split}\end{align}
where the contour $\cin$ encloses $[\fa_1, \fb_1]\cup[\fa_2, \fb_2]\cup\cdots [\fa_r, \fb_r]$ but not $z,z'$.
Then we get 
\begin{align*}\begin{split}
&\phantom{{}={}}\bE_\bme\left[
\int  \frac{\rd \mu_n}{z-x}\rd x\right]
=\bE_\bme\left[ \sum_i\left(\frac{1}{z-x_i-e_i/n}-\frac{1}{z-x_i}\right)\right]\\
&
=\left.\frac{\del_{z'}\tilde \cA(z')}{\tilde \cA(z')}\right|_{z'=z}
=\frac{1}{2\pi \ri}\oint_{\cin}\frac{\del_w\cB(w)}{\cB(w)}\frac{\rd w}{w-z}+\OO\left(\frac{1}{n}\right),
\end{split}\end{align*}
by setting $z'=z$ in \eqref{e:dtA}. This finishes the proof of Proposition \ref{p:difmeasure}.
\end{proof}

\subsection{Higher Order Term}
In this section we derive higher order terms of the quantities $\cA(z), \cC(z)$ as defined in \eqref{e:ABC} in terms of $\cB(z)$, and the limiting profile $\delta\beta^*$ as defined in \eqref{e:defbeta}.
\begin{proposition}\label{p:order2}
We assume Assumptions \ref{a:weight} and \ref{a:cancelB}, then
\begin{align*}
\frac{\del_z \cA(z)}{\cA(z)}
=\frac{1}{2\pi \ri}\oint_{\cin}\frac{\del_w \cB}{\cB}\frac{\rd w}{w-z}+\frac{1}{n}\frac{1}{2\pi \ri}\oint_{\cin}\frac{\del_w \cE_0\rd w}{w-z}+\OO\left(\frac{1}{n^2}\right),
\end{align*}
where the contour $\cin$ encloses $[\fa_1, \fb_1]\cup[\fa_2, \fb_2]\cup\cdots [\fa_r, \fb_r]$ but not $z$, and
\begin{align*}
\frac{\del_{z} \cC(z)}{\cC(z)}
&=\frac{1}{2\pi \ri}\oint_{\cout}\frac{\del_w\cB}{\cB}\frac{\rd w}{w-z}+\frac{1}{n}\frac{1}{2\pi \ri}\oint_{\cout}\frac{\del_w \cE_0\rd w}{w-z}+\OO\left(\frac{1}{n^2}\right),
\end{align*}
where the contour $\cout$ encloses $z$ and $[\fa_1, \fb_1]\cup[\fa_2, \fb_2]\cup\cdots [\fa_r, \fb_r]$, and
\begin{align*}
\cE_0(z)=\frac{1}{\cB}\left(\varphi^-\psi^-+\varphi^+\del_z \cG+\varphi^+\cG\left(\psi^++\int\frac{\rd \delta\beta^*}{z-x}-
\int\kappa(x,z)\delta\beta^*\rd x\right)\right)+\OO\left(\frac{1}{n}\right).
\end{align*}
and $\delta \beta^*$ is as defined in \eqref{e:defbeta}.
\end{proposition}

\begin{proof}
Following the proof of Proposition \ref{p:firstod}, we can write $\del_z \cA(z)/\cA(z)$ and $\del_z \cC(z)/\cC(z)$ as contour integrals,
\begin{align*}
\frac{\del_z \cA(z)}{\cA(z)}&=\frac{1}{2\pi \ri}\oint_{\cin}\left(\frac{\del_w \cB}{\cB}+\frac{\del_w \cE_0}{n}\right)\frac{\rd w}{w-z}+\OO\left(\frac{1}{n}\right),
\end{align*}
and
\begin{align*}
\frac{\del_{z} \cC(z)}{\cC(z)}
&=\frac{1}{2\pi \ri}\oint_{\cout}\left(\frac{\del_w\cB}{\cB}+\frac{\del_w \cE_0}{n}\right)\frac{\rd w}{w-z}+\OO\left(\frac{1}{n}\right).
\end{align*}
To get the next order expansion, we need to estimate $\cE_0(z)$ as defined in \eqref{e:defcE}.
Using our assumption \eqref{e:defvphi}, the first term in \eqref{e:defcE} simplifies to
\begin{align}\begin{split}\label{e:fft}
\frac{n}{\cB(z)}\left((\cG(z)\phi^+(z)+\phi^-(z))-\cB(z)\right)
=\frac{1}{\cB(z)}\left(\cG(z)\varphi^+(z)\psi^+(z)+\varphi^-(z)\psi^-(z)\right)+\OO\left(\frac{1}{n}\right).
\end{split}\end{align}
By a Taylor expansion, the second term in \eqref{e:defcE} simplifies to
\begin{align}\begin{split}\label{e:sst}
&\phantom{{}={}}\frac{n}{\cA(z)\cB(z)}\left(\cG(z+1/n)\cA(z+1/n)-\cG(z)\cA(z)\right)\phi^+(z)
=\frac{\del_z(\cG(z)\cA(z))\varphi^+(z)}{\cA(z)\cB(z)}+\OO\left(\frac{1}{n}\right)\\
&=\left(\frac{\del_z \cG(z)}{\cB(z)}+\frac{\cG(z)}{\cB(z)}\frac{\del_z \cA(z)}{\cA(z)}\right)\varphi^+(z)+\OO\left(\frac{1}{n}\right)
=\left(\frac{\del_z \cG(z)}{\cB(z)}+\frac{\cG(z)}{\cB(z)}\int\frac{\rd\delta\beta^* }{z-x}\right)\varphi^+(z)+\OO\left(\frac{1}{n}\right),
\end{split}\end{align}
where we used Proposition \ref{p:firstod} and \eqref{e:defbeta}. For the last term in \eqref{e:defcE}, we have
\begin{align}\begin{split}\label{e:defcEz}
&\phantom{{}={}}\frac{n}{\cA(z)\cB(z)}\sum_{\bme}a_\bme\prod_{i=1}^m\frac{z-x_i+(1-e_i)/n}{z-x_i}\left(e^{\sum_{i=1}^m \frac{e_i}{n^2}\kappa(x_i, z)}-1\right)\phi^+(z)\\
&=\frac{n}{\cA(z)\cB(z)}\sum_{\bme}a_\bme\prod_{i=1}^m\frac{z-x_i+(1-e_i)/n}{z-x_i}\left(\sum_{i=1}^m \frac{e_i}{n^2}\kappa(x_i, z)+\OO\left(\frac{1}{n^2}\right)\right)\phi^+(z)\\
&=\frac{1}{\cA(z)\cB(z)}\sum_{\bme}a_\bme\prod_{i=1}^m\frac{z-x_i+(1-e_i)/n}{z-x_i}\left(\sum_{i=1}^m \frac{e_i}{n}\kappa(x_i, z)\right)\varphi^+(z)+\OO\left(\frac{1}{n}\right).
\end{split}\end{align}
The above expression can be viewed as the expectation of
\begin{align*}\begin{split}
\sum_{i=1}^m \frac{e_i}{n}\kappa(x_i, z)
&=\sum_i e_i \left(\int^{x_i+1/n} \kappa(y,z)-\int^{x_i} \kappa(y,z)\right)+\OO\left(\frac{1}{n}\right)\\
&=\int^x \kappa(y,z)\rd y\rd\mu_n(x)+\OO\left(\frac{1}{n}\right),
\end{split}\end{align*}
under the deformed measure 
\begin{align*}
\frac{a_\bme}{\cG(z+1/n)\cA(z+1/n)}\prod_{i=1}^m\frac{z-x_i+(1-e_i)/n}{z-x_i}
=\frac{a_\bme}{\cG(z+1/n)\cA(z+1/n)}\prod_{i=1}^m\left(1+\frac{1}{n}\frac{1}{z-x_i}\right)^{1-e_i}.
\end{align*}
The above measure is equivalent to change $\phi^-(x_i)$ by a factor $1+1/(n(z-x_i))$, which will not affect the first order term.
Therefore Proposition \ref{p:difmeasure} gives that
\begin{align}\begin{split}\label{e:thirdt}
&\phantom{{}={}}\frac{1}{\cA(z)\cB(z)}\sum_{\bme}a_\bme\prod_{i=1}^m\frac{z-x_i+(1-e_i)/n}{z-x_i}\sum_{i=1}^m \frac{e_i}{n}\kappa(x_i, z)\varphi^+(z)\\
&=\frac{\cG(z+1/n)\cA(z+1/n)}{\cA(z)\cB(z)}\bE_\bme\left[\int^x \kappa(y,z)\rd y\rd\mu_n(x) \right]\varphi^+(z)+\OO\left(\frac{1}{n}\right)\\
&=\frac{\cG(z)\varphi^+(z)}{\cB(z)}
\int\int^x\kappa(y,z)\rd y\rd\delta\beta^*
+\OO\left(\frac{1}{n}\right)
=-\frac{\cG(z)\varphi^+(z)}{\cB(z)}
\int\kappa(x,z)\delta\beta^*\rd x
+\OO\left(\frac{1}{n}\right),
\end{split}\end{align}
where $\delta\beta^*$ is as defined in \eqref{e:defbeta}.

The estimates \eqref{e:fft}, \eqref{e:sst} and \eqref{e:thirdt} all together imply
\begin{align*}\begin{split}
\cE_0(z)&=\frac{1}{\cB}\left(\cG\varphi^+\psi^++\varphi^-\psi^-+\varphi^+\del_z \cG+\varphi^+\cG\int\frac{\rd \delta\beta^*}{z-x}-\varphi^+\cG
\int\kappa(x,z)\delta\beta^*\rd x\right)+\OO\left(\frac{1}{n}\right)\\
&=\frac{1}{\cB}\left(\varphi^-\psi^-+\varphi^+\del_z \cG+\varphi^+\cG\left(\psi^++\int\frac{\rd \delta\beta^*}{z-x}-
\int\kappa(x,z)\delta\beta^*\rd x\right)\right)+\OO\left(\frac{1}{n}\right).
\end{split}\end{align*}
This finishes the proof of Proposition \ref{p:order2}.
\end{proof}

\begin{proposition}\label{p:second}
We assume Assumptions \ref{a:weight} and \ref{a:cancelB}, then the measure $\mu_n$ as defined in \eqref{e:mun} satisfies
\begin{align*}
\bE_\bme\left[
\int  \frac{\rd\mu_n}{z-x}\right]
=\frac{1}{2\pi \ri}\oint_{\cin}\frac{\del_w\cB}{\cB}\frac{\rd w}{w-z}+\frac{1}{n}\frac{1}{2\pi \ri}\oint_{\cin}\frac{\del_w \cE_1(w)\rd w}{w-z}+\OO\left(\frac{1}{n^2}\right),
\end{align*}
where the contour $\cin$ encloses $[\fa_1, \fb_1]\cup[\fa_2, \fb_2]\cup\cdots [\fa_r, \fb_r]$ but not $z$,
\begin{align*}
\cE_1(w)=
\frac{1}{\cB}\left(\varphi^-\psi^-+\varphi^+\del_{w} \cG+\varphi^+\cG\left(\psi^++\frac{1}{z-w}+\int\frac{\rd \delta\beta^*}{w-x}-
\int\kappa(x,w)\delta\beta^*\rd x\right)\right)+\OO\left(\frac{1}{n}\right).
\end{align*}
\end{proposition}

\begin{proof}
We recall the deformed weight $\tilde a_\bme$ and corresponding $\tilde \cA(z')$ as defined in \eqref{e:tae} and \eqref{e:tA}. 
They correspond to change $\phi^+(z')$ by a factor $(z-z')/(z-z'-1/n)$, which can be absorbed in $\psi^+(z')$:
\begin{align*}
\tilde \phi^+(z')
=\phi^+(z')\frac{z-z'}{z-z'-1/n}
=\varphi^+(z')e^{\frac{1}{n}\left(\psi^+(z')+n\ln \left(\frac{z-z'}{z-z'-1/n}\right)\right)}
=:\varphi^+(z')e^{\frac{1}{n}\tilde \psi^+(z')}.
\end{align*}
The new weight $\tilde \psi^+(z')$ is 
\begin{align*}
\tilde \psi^+(z')=\psi^+(z')+n\ln \left(\frac{z-z'}{z-z'-1/n}\right)
=\psi^+(z')+\frac{1}{z-z'}+\OO\left(\frac{1}{n}\right).
\end{align*}
Then thanks to Proposition \ref{p:order2}, 
\begin{align}\label{e:dtA2}
\frac{\del_{z'} \tilde \cA(z')}{\tilde \cA(z')}&=\frac{1}{2\pi \ri}\oint_{\cin}\frac{\del_w \cB}{\cB}\frac{\rd w}{w-z'}+\frac{1}{n}\frac{1}{2\pi \ri}\oint_{\cin}\frac{\del_w \cE_1\rd w}{w-z'}+\OO\left(\frac{1}{n^2}\right),
\end{align}
where  the contour $\cin$ encloses $[\fa_1, \fb_1]\cup[\fa_2, \fb_2]\cup\cdots [\fa_r, \fb_r]$ but not $z'$, and
\begin{align*}\begin{split}
\cE_1(w)=\frac{1}{\cB}\left(\varphi^-\psi^-+\varphi^+\del_{w} \cG+\varphi^+\cG\left(\tilde \psi^++\int\frac{\rd \delta\beta^*}{w-x}-
\int\kappa(x,w)\delta\beta^*\rd x\right)\right)+\OO\left(\frac{1}{n}\right)\\
=\frac{1}{\cB}\left(\varphi^-\psi^-+\varphi^+\del_{w} \cG+\varphi^+\cG\left(\psi^++\frac{1}{z-w}+\int\frac{\rd \delta\beta^*}{w-x}-
\int\kappa(x,w)\delta\beta^*\rd x\right)\right)+\OO\left(\frac{1}{n}\right).
\end{split}\end{align*}
Then we get 
\begin{align*}\begin{split}
&\phantom{{}={}}\bE_\bme\left[
\int  \frac{\rd \mu_n}{z-x}\right]
=\bE_\bme\left[ \sum_i\left(\frac{1}{z-x_i-e_i/n}-\frac{1}{z-x_i}\right)\right]
=\left.\frac{\del_{z'}\tilde \cA(z')}{\tilde \cA(z')}\right|_{z'=z}\\
&=\frac{1}{2\pi \ri}\oint_{\cin}\frac{\del_z \cB}{\cB}\frac{\rd w}{w-z}+\frac{1}{n}\frac{1}{2\pi \ri}\oint_{\cin}\frac{\del_w \cE_1\rd w}{w-z}+\OO\left(\frac{1}{n^2}\right),
\end{split}\end{align*}
by setting $z'=z$ in \eqref{e:dtA2} and this finishes the proof of Proposition \ref{p:second}.

\end{proof}

\section{Global Gaussian Fluctuation}\label{s:Gauss}

Fix a polygonal domain ${\mathfrak P}$ satisfying (1) the top side is connected,
(2) other horizontal boundary sides are lower boundaries of the polygon,
as in Definition \ref{def:oneend}.
For any time $s\in [0,T)$,  we recall the variational problem from \eqref{e:varWt}, and the corresponding complex Burgers equation from \eqref{e:burgereq4}
\begin{align}\label{e:bgcopy}
\del_t m_t(z;\beta,s)+\del_t \ln g_t(Z;\beta,s)+\del_z \ln (f_t(Z;\beta,s)+1)=0,\quad z=z(Z),\quad Z\in \gamma_t^{\beta,s}.
\end{align}
We recall the bottom boundary of $\mathfrak P\cap \bR\times[t,T]$ from \eqref{e:lbcopy}
\begin{align*}
\{ x :  (x,t+)\in {\mathfrak P}\}=[\fa_1(t), \fb_1(t)]\cup[\fa_2(t), \fb_2(t)]\cup\cdots \cup [\fa_{r(t)}(t),\fb_{r(t)}(t)].
\end{align*}
And for any $x\in [\fa_1(t), \fb_1(t)]\cup[\fa_2(t), \fb_2(t)]\cup\cdots \cup [\fa_{r(t)}(t),\fb_{r(t)}(t)]$, we recall the sets $A_{t}(x)$ and $B_{t}(x)$ from \eqref{e:defAB}, which can be extended to a constant function in a neighborhood of $z^{-1}(\{ x :  (x,t+)\in {\mathfrak P}\})$ on $\gamma_t^{\beta, s,+}$.
The poles $a=z(X_i^{(1)})\in A_t(Z)$ travel at rate $1$, i.e.
$\del_t a=1$, and the zeros $b=z(X_i^{(\infty)})\in B_t(Z)$ does not move, i.e. $\del_t b=0$.  We denote 
\begin{align}\label{e:defvarphi}
\varphi_t^-(Z;\beta,s)=\prod_{a\in A_{t}(Z)}(z-a), \quad \varphi_t^+(Z;\beta,s)= g_t'(Z;\beta,s)\prod_{b\in B_{t}(Z)}  (b-z),\quad z=z(Z),
\end{align}
where $g_t'(Z;\beta,s)$ is an extension of $g_t'(x;\beta,s)$ (as defined in \eqref{e:deftgt}) to a neighborhood of $z^{-1}(\{x:(x,t+)\in \fP\})$ on $\gamma_t^{\beta,s,+}$.
Then we have $\del_t \ln\varphi_t^-(Z;\beta,s)+\del_z \ln\varphi_t^-(Z;\beta,s)=0$. We add it to both sides of the complex Burgers equation \eqref{e:bgcopy}  and get
\begin{align}\label{e:bgcopy2}
\del_t m_t(z;\beta,s)+\del_t \ln  g_t'(Z;\beta,s)+\del_z \ln (e^{m_t(z;\beta,s)}\varphi_t^+(Z;\beta,s)+\varphi_t^-(Z;\beta,s))=0,
\end{align}

Thanks to Proposition \ref{p:branchloc}, in a neighborhood $\Upsilon_t^{\beta,s}$ of $z^{-1}(\{ x :  (x,t+)\in {\mathfrak P}\})$ over $\gamma_t^{\beta,s,+}$, the covering map  $\gamma_t^{\beta,s,+}\mapsto \bC\bP^1$ does not have a branch point. By taking $\Upsilon_t^{\beta,s}$ small enough,  the covering map  $\gamma_t^{\beta,s,+}\mapsto \bC\bP^1$ restricted on $\Upsilon_t^{\beta,s}$ is a bijection. In this way we can identify $\Upsilon_t^{\beta,s}$ with a neighborhood of $[\fa_1(t), \fb_1(t)]\cup[\fa_2(t), \fb_2(t)]\cup\cdots \cup [\fa_{r(t)}(t),\fb_{r(t)}(t)]$ over $\bC$. In the rest of this Section, we use this identification, and view $\Upsilon_t^{\beta,s}$ as a subset of $\bC$.

In this section we consider nonintersecting Bernoulli random walks with generator $\cL_t$ at time $t$ given by
\begin{align}\label{e:defLtnew}
\frac{a_t(\bme;\bmx)}{Z_t(\bmx)}\deq \frac{1}{Z_t(\bmx)}\frac{V(\bmx+\bme/n)}{V(\bmx)}\prod_{i=1}^m\phi_t^+(x_i;\beta,t)^{e_i}\phi_t^-(x_i;\beta,t)^{1-e_i}\exp\left\{\sum_{1\leq i,j\leq m}\frac{e_ie_j}{n^2}\kappa(x_i,x_j;\beta,t)\right\},
\end{align}
where 
\begin{align}\label{e:defvphit}
\phi_t^+(z;\beta,t)=\varphi_t^+(z;\beta,t)e^{\frac{1}{n}\psi_t^+(z;\beta,t)},\quad 
\phi_t^-(z;\beta,t)=\varphi_t^-(z;\beta,t),
\end{align}
$\varphi_t^{\pm}$ are given in \eqref{e:defvarphi}, $\psi_t^+(z;\beta,t)$ is analytic on $\Upsilon_t^{\beta,t}$, which will be chosen later, and $\kappa(z,w;\beta,t)=\kappa(w,z;\beta,t)$ is constructed in \eqref{e:defka}. From our construction of $\varphi_t^\pm(z;\beta,t)$ in \eqref{e:defvarphi}, they are analytic function over $\Upsilon_t^{\beta,s}$. {and so is $\kappa(z,w;\beta,t)$ analytic}. Therefore, the generator $\cL_t$ satisfies Assumption \ref{a:weight}.


We remark that the Markov process \eqref{e:newat0}, which we used to solve $E_{s}$ in Proposition \ref{p:FK} is a special case of \eqref{e:defLtnew}, with 
\begin{align*}
\psi_t^+(z;\beta,t)=-\sum_{b\in B_t(x)}n\ln\left(1-\frac{1}{n(b-z)}\right)+
\frac{1}{2}\del_z \ln  g_t'(z;\beta,t)+\del_t(\ln g_t'(z;\beta,t))+
\OO(1/n).
\end{align*}
As we will see in the proof, the $\OO(1/n)$ error contributes to an error of $\OO(1/n^2)$ for the Stieltjes transform of the empirical particle density, which is negligible for our analysis.
More importantly, later we will see in Section \ref{s:solve} that the nonintersecting Bernoulli walk \eqref{e:randomwalk}, which has the same law as random lozenge tilings of the domain $\fP$ is also in the form of \eqref{e:defLtnew}, with some suitably chosen $\psi_t^+(z;\beta,t)$.

We will study the generator $\cL_t$ in \eqref{e:defLtnew} using the discrete loop equations we developed in Section \ref{s:loopeq}. We recall the function $\cB(z)$ from \eqref{e:ABC}, in our setting it is 
\begin{align*}
\cB_t(z;\beta,t)=e^{m_t(z;\beta,t)}\varphi_t^+(z;\beta,t)+\varphi_t^-(z;\beta,t).
\end{align*}
With $\cB_t(z;\beta,t)$ we can rewrite the complex Burgers equation \eqref{e:bgcopy2} as
\begin{align}\label{e:bgcopy3}
\del_t m_t(z;\beta,t)+(\del_t \ln  g_t')(z;\beta,t)+\del_z \ln (\cB_t(z;\beta,t))=0.
\end{align}

We recall from the discussion after \eqref{e:deftgt}, $ \ln g_t'(z;\beta,t)$ is analytic over $\Upsilon_t^{\beta,s}$. We can do a contour integral on both sides of \eqref{e:bgcopy3} to get ride of $\del_t \ln  g_t'(Z;\beta,t)$,
\begin{align}\label{e:B1}
\del_t m_t(z;\beta,t)&=\frac{1}{2\pi \ri}\oint_{\cin}\frac{\del_z  \cB_t(w;\beta,t)}{\cB_t(w;\beta,t)}\frac{\rd w}{w-z},
\end{align}
where the contour $\cin$ encloses $z^{-1}(\{x:(x,t+)\in \fP\})$ but not $z$.  Since the derivative of the Stieltjes transform $\del_t m_s(z;\beta,t)$ behaves like $\OO(1)/z^2$ as $z\rightarrow \infty$, Assumption \ref{a:cancelB} holds in our setting. And we can also do a contour integral on both sides of \eqref{e:bgcopy3} to get ride of $\del_t m_t(z;\beta,t)$,
\begin{align}\label{e:B2}
(\del_t \ln g_t')(z;\beta,t)=(\del_t \ln g_t)(z;\beta,t)-\sum_{a\in A_t(Z)}\frac{1}{z-a}=-\frac{1}{2\pi \ri}\oint_{\cout}\frac{\del_z  \cB_t(w;\beta,t)}{\cB_t(w;\beta,t)}\frac{\rd w}{w-z},
\end{align}
where the contour $\cout$ encloses $z^{-1}(\{x:(x,t+)\in \fP\})$ and $z$.
%

\subsection{Dynamic Equation}
We denote the  Markov process starting from the configuration $\bmx=\bmx_s=\{x_1(s), x_2(s),\cdots, x_{m(s)}(s)\}\in \fM_s(\fP)$ with a time dependent generator $\cL_t$ given by \eqref{e:defLtnew} as $\{\bmx_t\}_{t\in[s, T]\cap \bZ_n}$.
In this section, we will study the dynamics of the  Stieltjes transform of its empirical particle density
\begin{align}\label{e:defrhot}
\rho(x;\bmx_t)=\sum_{i=1}^{m(t)}\bm1(x\in[x_i(t), x_i(t)+1/n]).
\end{align}
The  empirical particle density of the initial configuration is $\rho(x;\bmx)=\rho(x;\bmx_s)$. 
Let $\beta_t=\beta(x;\bmx_t)$ as in \eqref{e:defbeta0} with $\rho(x;\bmx_t)=\del_x \beta(x;\bmx_t)$, and simply write $\beta=\beta_s$. The  Stieltjes transform of its empirical particle density \eqref{e:defrhot} is
\begin{align*}
\tilde m_t(z;\beta,s)=\int \frac{\rho(x;\bmx_t)}{z-x}\rd x=\int^z_{z-1/n}\sum_{i=1}^{m(t)}\frac{1}{u-x_i(t)} \rd u.
\end{align*}
Then
\begin{align}\begin{split}\label{e:dm}
n(\tilde m_{t+1/n}(z;\beta,s)-\tilde m_t(z;\beta,s))&=n\int_{z-1/n}^z \left(\sum_{i=1}^{m(t)}\frac{1}{u-x_i(t+1/n)}-\sum_{i=1}^{m(t)}\frac{1}{u-x_i(t)}\right)\rd u\\
&=\Delta M_t +n\int_{z-1/n}^z \cL_t(n\tilde m_t(u))\rd u,
\end{split}\end{align}
where $\Delta M_t$ is a martingale difference, i.e. $\bE[\Delta M_t|\bmx_t]=0$, and 
the drift term is given by
\begin{align}\label{e:keyt}
\cL_t(n\tilde m_t(u))
=\sum_{\bme\in\{0,1\}^{m(t)}} \frac{a_t(\bme;\bmx)}{Z_t(\bmx)}\left(\sum_{i=1}^{m(t)}\frac{1}{u-x_i(t)-e_i/n}-\frac{1}{u-x_i(t)}\right).
\end{align}

We recall the characteristic flow from  \eqref{e:burgernew}, which solves the complex Burgers equation
\begin{align}\label{e:ccff}
\del_t f_t(Z_t;\beta,s)=0,\quad \del_t z_t=\frac{f_t(Z_t;\beta,s)}{f_t(Z_t;\beta,s)+1},\quad z_t=z(Z_t).
\end{align}
By plugging the characteristic flow \eqref{e:ccff} into \eqref{e:dm}, we have 
\begin{align}\begin{split}\label{e:Deltatmt}
&\phantom{{}={}} n\Delta_t (\tilde m_t(z_t))
\deq
 n(\tilde m_{t+1/n}(z_{t+1/n})-\tilde m_{t}(z_t))\\
& =
 n(\tilde m_{t+1/n}(z_{t+1/n})-\tilde m_{t}(z_{t+1/n}))
+ n(\tilde m_{t}(z_{t+1/n})-\tilde m_{t}(z_{t}))\\
& =\Delta M_t(z_{t+1/n})+n\int_{z_{t+1/n}-1/n}^{z_{t+1/n}} \cL_t(n\tilde m_t(u))\rd u
+\del_z \tilde m_t(z_t) \del_t z_t\\
&+\frac{1}{2n}\left(\del_z ^2 \tilde m_t(z_t)(\del_t z_t)^2+\del_z \tilde m_t(z_t)\del^2_t z_t\right)+\OO\left(\frac{1}{n^2}\right).
\end{split}\end{align}
Thanks to Proposition \ref{p:second}, we can compute explicitly $\cL_t(\tilde m_{t}(u))$ as
\begin{align}\begin{split}\label{e:bbt0}
&\phantom{{}={}}n\int_{z_{t+1/n}-1/n}^{z_{t+1/n}} \cL_t(n\tilde m_t(u))\rd u =n\int_{z_{t+1/n}-1/n}^{z_{t+1/n}}\cL_t\left(\sum_{i=1}^{m(t)}\frac{1}{u-x_i(t)}\right)\rd u\\
&=n\int_{z_{t+1/n}-1/n}^{z_{t+1/n}} \frac{1}{2\pi \ri}\oint_{\cin}\frac{\del_w\cB_t}{\cB_t}\frac{\rd w}{w-u}\rd u+\frac{1}{n}\frac{1}{2\pi \ri}\oint_{\cin}\frac{\del_w \cE_1\rd w}{w-z_{t}}+\OO\left(\frac{1}{n^2}\right),
\end{split}\end{align}
where the contour $\omega_-$ encloses $z^{-1}(\{x:(x,t+)\in \fP\})$ but not $z_t, z_{t+1/n}$, the function $\cG(w)$ in Proposition \ref{p:second} is $e^{\tilde m_t(w;\beta,s)}$, and
\begin{align}\label{e:bbt1}
\cE_1(w)=
\frac{e^{\tilde m_t}\varphi_t^+}{\cB_t}\left(\del_w \tilde m_t+\psi_t^++\frac{1}{z_t-w}+\int\frac{\rd \del_t h_t(x;\beta_t,t)}{w-x}-
\int K(x,w;\beta_t,t)\del_t h_t(x;\beta_t,t)\rd x\right).
\end{align}
We can rewrite the leading term in \eqref{e:bbt0} as
\begin{align}\begin{split}\label{e:bbt2}
&\phantom{{}={}}n\int_{z_{t+1/n}-1/n}^{z_{t+1/n}} \frac{1}{2\pi \ri}\oint_{\cin}\frac{\del_w\cB_t}{\cB_t}\frac{\rd w}{w-u}\rd u\\
&=\frac{1}{2\pi \ri}\oint_{\cin}\frac{\del_w\cB_t}{\cB_t}\frac{\rd w}{w-z_{t}}
+\frac{2\del_t z_t-1}{4 n\pi \ri}\oint_{\cin}\frac{\del_w\cB_t}{\cB_t}\frac{\rd w}{(w-z_{t})^2}+\OO\left(\frac{1}{n^2}\right).
\end{split}\end{align}
By plugging \eqref{e:bbt1} and \eqref{e:bbt2} into \eqref{e:bbt0}, we get 
\begin{align}\begin{split}\label{e:bbt3}
&\phantom{{}={}}n\int_{z_{t+1/n}-1/n}^{z_{t+1/n}} \cL_t(n\tilde m_t(u))\rd u
=\frac{1}{2\pi \ri}\oint_{\cin}\frac{\del_w\cB_t}{\cB_t}\frac{\rd w}{w-z_{t}}\\
&+\frac{1}{n}\frac{1}{2\pi \ri}\left((\del_t z_t-1/2)\oint_{\cin}\frac{\del_w\cB_t}{\cB_t}\frac{\rd w}{(w-z_{t})^2}+\oint_{\cin}\frac{\del_w \cE_1\rd w}{w-z_{t}}\right)+\OO\left(\frac{1}{n^2}\right).
\end{split}\end{align}
Using \eqref{e:bbt3}, we have the following more explicit expression of $n \Delta_t (\tilde m_t(z_t))$ from \eqref{e:Deltatmt}
\begin{align}\begin{split}\label{e:Deltatmt2}
&\phantom{{}={}}n \Delta_t (\tilde m_t(z_t))
 =\Delta M_t(z_{t+1/n})+\frac{1}{2\pi \ri}\oint_{\cin}\frac{\del_w\cB_t}{\cB_t}\frac{\rd w}{w-z_{t}}+\del_z \tilde m_t(z_t) \del_t z_t\\
&+\frac{1}{n}\frac{1}{2\pi \ri}\left((\del_t z_t-1/2)\oint_{\cin}\frac{\del_w\cB_t}{\cB_t}\frac{\rd w}{(w-z_{t})^2}+\oint_{\cin}\frac{\del_w \cE_1\rd w}{w-z_{t}}\right)\\&+\frac{1}{2n}\left(\del_z ^2 \tilde m_t(z_t)(\del_t z_t)^2+\del_z \tilde m_t(z_t)\del^2_t z_t\right)+\OO\left(\frac{1}{n^2}\right).
\end{split}\end{align}

\begin{remark}\label{r:leadingorder}
We can also directly simplify \eqref{e:dm} without plugging into the characteristic flow to get
\begin{align}\begin{split}\label{e:leadingorder}
n(\tilde m_{t+1/n}(z;\beta,s)-\tilde m_t(z;\beta,s))&
=\Delta M_t +\frac{1}{2\pi \ri}\oint_{\cin}\frac{\del_w\cB_t}{\cB_t}\frac{\rd w}{w-z} 
+\OO(1/n)\\
&=\Delta M_t +\del_t m_t(z;\beta,t) 
+\OO(1/n),
\end{split}\end{align}
where we used \eqref{e:B1} in the last line. It turns out the fluctuations of the Stieltjes transform is determined by the martingale difference $\Delta M_t$ and the leading order term $\del_t m_t(z;\beta,t)$. The $\OO(1/n)$ term affects only the mean of the Stieltjes transform.
\end{remark}

We denote the solution of the variational problem \eqref{e:varWt} starting at time $s$ with boundary condition $\beta$ as $h_t^*=h_t^*(x;\beta,s)$, then $g_t(Z;\beta,s)=g_t(Z;h^*_t,t)$ and $f_t(Z;\beta,s)=f_t(Z;h_t^*,t)$, and by identifying the neighborhood $\Upsilon_t^{h_t^*,t}$ of $z^{-1}([\fa_1(t), \fb_1(t)]\cup[\fa_2(t), \fb_2(t)]\cup\cdots \cup [\fa_{r(t)}(t),\fb_{r(t)}(t)])$ over $\gamma_t^{h_t^*,t}$ with the neighborhood of $[\fa_1(t), \fb_1(t)]\cup[\fa_2(t), \fb_2(t)]\cup\cdots \cup [\fa_{r(t)}(t),\fb_{r(t)}(t)])$ over $\fC$, we can rewrite \eqref{e:bgcopy} as 
\begin{align*}
\del_t m_t(z;\beta,s)+\del_t \ln g_t(z;h_t^*,t)+\del_z \ln (f_t(z;h_t^*,t)+1)=0,\quad \quad z\in \Upsilon^{h_t^*,t}.
\end{align*}
Similar to \eqref{e:B1}, we can express $\del_t m_t(z;\beta,s)$ as a contour integral
\begin{align}\label{e:dmtta}
\del_t m_t(z;\beta,s)&=\frac{1}{2\pi \ri}\oint_{\cin}\frac{\del_w  \cB_t^*}{\cB_t^*}\frac{\rd w}{w-z},
\end{align}
where the contour $\omega_-$ encloses $z^{-1}(\{x:(x,t+)\in \fP\})$ but not $z$, and
\begin{align*}
 \cB_t^*:=\cB_t(z;h_t^*,t)=\left(e^{m_t(z;\beta, s)} g_t'(z;h_t^*, t)+1\right)\prod_{b\in B_{t}(z)}(b-z).
\end{align*}
Using \eqref{e:dmtta}, and the characteristic flow \eqref{e:ccff}, it follows that 
\begin{align}\begin{split}\label{e:Deltamt}
&\phantom{{}={}}n\Delta_t  ( m_t(z_t;\beta,s))
 \deq n(m_{t+1/n}(z_{t+1/n};\beta,s)-m_t(z_t;\beta,s))\\
 &=\frac{1}{2\pi \ri}\oint_{\cin}\frac{\del_w  \cB_t^*}{\cB_t^*}\frac{\rd w}{w-z_{t}}+\del_z  m_t(z_t) \del_t z_t+\frac{\del_t z_t}{2n\pi \ri}\oint_{\cin}\frac{\del_w\cB_t^*}{\cB_t^*}\frac{\rd w}{(w-z_{t})^2}\\
 &
 +\frac{1}{2n}\left(\del_t^2 m_t(z_t;\beta, s)+\del_z ^2  m_t(z_t)(\del_t z_t)^2+\del_z  m_t(z_t)\del^2_t z_t\right)+\OO\left(\frac{1}{n^2}\right).
\end{split}\end{align}

By taking difference of \eqref{e:Deltatmt2} and \eqref{e:Deltamt} we get 
\begin{align}\label{e:leq1}\begin{split}
n\Delta_t (\tilde m_t(z_t;\beta,s)- m_t(z_t;\beta,s))
&=\Delta M_t+\frac{1}{2\pi \ri}\oint_{\cin}\left(\frac{\del_w \cB_t}{\cB_t}-\frac{\del_w  \cB_t^*}{\cB_t^*}\right)\frac{\rd w}{w-z_{t}}\\
&+(\del_z\tilde m_t(z_t)-\del_z m_t(z_t))\del_t z_t
+\frac{\cE_2(z_t)}{n}+\OO\left(\frac{1}{n^2}\right),
\end{split}\end{align}
where the contour $\omega_-$ encloses $z^{-1}(\{x:(x,t+)\in \fP\})$ but not $z_t$, and the correction term $\cE_2(z_t)$ is given by
\begin{align}\begin{split}\label{e:cE2}
    &\phantom{{}={}}\cE_2(z_t)=\frac{1}{2}\del_t^2 m_t(z_t;\beta, s)+\frac{1}{2\pi \ri}\oint_{\cin}\frac{\del_w \cE_1\rd w}{w-z_{t}}
    +\frac{\del_t z_t}{2\pi \ri}\oint_{\cin}\left(\frac{\del_w\cB_t}{\cB_t}-\frac{\del_w\cB_t^*}{\cB_t^*}\right)\frac{\rd w}{(w-z_{t})^2}\\
    &-\frac{1}{4\pi \ri}\oint_{\cin}\frac{\del_w\cB_t}{\cB_t}\frac{\rd w}{(w-z_{t})^2}+\frac{1}{2}\left(\del_z ^2  \tilde m_t(z_t)(\del_t z_t)^2+\del_z  \tilde m_t(z_t)\del^2_t z_t-\del_z ^2  m_t(z_t)(\del_t z_t)^2-\del_z  m_t(z_t)\del^2_t z_t\right).
\end{split}\end{align}
We remark that for the difference of $\tilde m_t(z_t;\beta,s)$ with $m_t(z_t;\beta,s)$, if there are new particles added at time $t$, i.e. $t\in \{T_0, T_1, T_2,\cdots, T_p\}$, the contributions from those new particles cancel out.

Let $\tilde f_t\deq f_t(w;\beta_t,t)$ $\tilde g_t\deq g_t(w;\beta_t,t)$, and $f_t\deq f_t(w;\beta,s)=f_t(z;h_t^*,t)$, $g_t\deq g_t(w;\beta,s)=g_t(z;h_t^*,t)$. For the second term on the righthand side of \eqref{e:leq1}, its integrand is
\begin{align}\begin{split}\label{e:bdidd}
\frac{\del_w \cB_t}{\cB_t}-\frac{\del_w  \cB_t^*}{\cB_t^*}
&=\del_w (\ln (\tilde f_t+1)-\ln ( f_t+1))=\frac{\del_w (\ln \tilde f_t-\ln f_t) f_t}{f_t+1}+\del_w \ln\tilde f_t\frac{(e^{\ln \tilde f_t-\ln f_t}-1)f_t}{(\tilde f_t+1)(f_t+1)}\\
&=\frac{\del_w ( \tilde m_t-m_t) f_t}{f_t+1}+\frac{\del_w (\ln  \tilde g_t-\ln g_t) f_t}{f_t+1}+\del_w \ln\tilde f_t\frac{(e^{\ln \tilde f_t-\ln f_t}-1)f_t}{(\tilde f_t+1)(f_t+1)}.
\end{split}\end{align}
Thanks to Proposition \ref{p:derg}, $\ln g_t(z;\beta,t)$ depends on $\beta$ smoothly. For the differences $\ln\tilde f_t-\ln f_t$ and $\ln\tilde g_t-\ln g_t$,  we can Taylor expand them around $h_t^*$, and use \eqref{e:ddlogg}
\begin{align}\begin{split}\label{e:gdidd}
\ln g_t(w)-\ln  \tilde g_t(w)
&=\int \del_\beta \ln g_t(w;h_t^*,t) (\beta_t-h_t^*)\rd x + \OO\left(\int \del^2_\beta \ln g_t(w;h_t^*,t) (\beta_t-h_t^*)(\beta_t-h_t^*)\rd x\rd y\right)\\
&= \int K(w,x;h_t^*,t)(\beta_t-h_t^*)\rd x
+\OO((\tilde m_t-m_t)^2),
\end{split}\end{align}
where for the integral with respect to $\beta_t-h_t^*$, we can rewrite them as contour integrals with respect to $\tilde m_t-m_t$. Later we will show that with probability $1-\oo(1)$, $\tilde m_t-m_t$ is uniformly small, i.e. of order $\OO(1/n) $. So we simply denote the error term as $\OO((\tilde m_t-m_t)^2)$, which is of order $\OO(1/n^2)$.
The difference $\ln\tilde f_t-\ln f_t$ is given by
\begin{align}\begin{split}\label{e:fdidd}
\ln f_t(w)-\ln \tilde f_t(w)
&=(m_t(w)-\tilde m_t(w))+\int \del_\beta \ln g_t(w;h_t^*,t) (\beta_t-h_t^*)\rd x
+\OO((\tilde m_t-m_t)^2)\\
&=(m_t(w)-\tilde m_t(w))+\int K(w,x;h_t^*,t)(\beta_t-h_t^*)\rd x
+\OO((\tilde m_t-m_t)^2).
\end{split}\end{align}
By plugging \eqref{e:gdidd} and \eqref{e:fdidd} into \eqref{e:bdidd}, we get
\begin{align*}\begin{split}
\frac{\del_w \cB_t}{\cB_t}-\frac{\del_w  \cB_t^*}{\cB_t^*}
=\del_w\left(\frac{ ( \tilde m_t-m_t) f_t}{f_t+1}+\frac{f_t}{f_t+1}\int K(w,x;h_t^*,t)(\beta_t-h_t^*)\rd x\right)+\OO((\tilde m_t-m_t)^2).
\end{split}\end{align*}
The integral with respect to $\beta_t-h_t^*$ can be written as a contour integral with respect to the difference of Stieltjest transforms $\tilde m_t-m_t$,
\begin{align*}
\int K(w,x;h_t^*,t)(\beta_t-h_t^*)\rd x
=-\frac{1}{2\pi\ri}\oint_\omega \int^u K(w,x;h_t^*,t)\rd x  (\tilde m_t(u)-m_t(u))\rd u,
\end{align*}
where the contour integral $\omega$ encloses $z^{-1}(\{x: (x,t+)\in \fP\})$.
The contour $\omega_-$ in \eqref{e:leq1} encloses $z^{-1}(\{x: (x,t+)\in \fP\})$ but not $z_t$, we can deform it to the contour $\omega_+$, which encloses both $z_t$ and $z^{-1}(\{x: (x,t+)\in \fP\})$, and subtract the residual at $z_t$. In this way we can rewrite the second term on the righthand side of \eqref{e:leq1} as
\begin{align}\begin{split}\label{e:BBdiff}
&\phantom{{}={}}\frac{1}{2\pi \ri}\oint_{\omega_-} \left(\frac{\del_w \cB_t}{\cB_t}-\frac{\del_w  \cB_t^*}{\cB_t^*}\right)\frac{\rd w}{w-z_{t}}\\
&=-\del_w \left.\left(\frac{(\tilde m_t- m_t)f_t}{(f_t+1)}\right)\right|_{w=z_t}+\frac{1}{2\pi\ri}\oint_{\omega_+} \sfK_t(w,z_t;\beta,s)(\tilde m_t(w)- m_t(w))\rd w+\OO((\tilde m_t-m_t)^2),
\end{split}\end{align}
where 
\begin{align*}
\sfK_t(w,z_t;\beta,s)=\frac{f_t}{f_t+1}\frac{1}{(w-z_t)^2}-\frac{1}{2\pi\ri}\oint_{\omega_-}\frac{f_t(u)}{f_t(u)+1}\frac{\int^w K(u,x;h_t^*,t)\rd x}{(u-z_t)^2}\rd u.
\end{align*}
For the first term on the righthand side of \eqref{e:BBdiff}, we recall that $f_t(z_t;\beta,s)/(f_t(z_t;\beta,s)+1)=\del_t z_t$ from \eqref{e:ccff}, and $\del_w(f_t/(f_t+1))|_{w=z_t}=\del_t \rd z_t/\rd z_t$. We can rewrite it as
\begin{align}\label{e:BBdiff2}
-\del_w \left.\left(\frac{(\tilde m_t- m_t)f_t}{(f_t+1)}\right)\right|_{w=z_t}
=-(\del_z\tilde m_t(z_t)- \del_z m_t(z_t))\del_t z_t-(\tilde m_t- m_t)\frac{\del_t \rd z_t}{\rd z_t}.
\end{align}

By plugging \eqref{e:BBdiff} and \eqref{e:BBdiff2} back into \eqref{e:leq1}, we get the following equation for the difference of the Stieltjes transforms
\begin{align}\begin{split}\label{e:mainEq}
n\Delta (\tilde m_t- m_t)
&=\Delta M_t-(\tilde m_t- m_t)\frac{\del_t \rd z_t}{\rd z_t}+\frac{1}{2\pi\ri}\oint_{\omega_+} \sfK_t(w,z_t;\beta,s)(\tilde m_t(w)-m_t(w))\rd w\\
&+\frac{1}{n}\sfE_t(z_t;\beta,s)+\OO\left((\tilde m_t-m_t)^2+\frac{(\tilde m_t-m_t)}{n}+\frac{1}{n^2}\right),
\end{split}\end{align}
where $\sfE_t(z_t;\beta,s)$ is obtained from $\cE_2(z_t)$ as in \eqref{e:cE2} by a Taylor expansion around $h_t^*$
\begin{align*}\begin{split}
    &\phantom{{}={}}\sfE_t(z_t;\beta,s)=\frac{1}{2}\del_t^2 m_t(z_t;\beta, s)\\
    &+\frac{1}{2\pi \ri}\oint_{\cin}\left(\frac{e^{ m_t}\varphi_t^+}{\cB^*_t}\left(\del_w  m_t+\psi_t^++\frac{1}{z_t-w}+\del_t m_t-
\int K(w,x;h_t^*,t)\del_t h^*_t\rd x\right) -\frac{\del_w\cB^*_t}{2\cB^*_t}\right)\frac{\rd w}{(w-z_{t})^2}.
\end{split}\end{align*}
The key point for \eqref{e:mainEq} is that both $\sfK_t(w, z_t;\beta,s)$ and $\sfE_t(z_t;\beta,s)$ are deterministic, depending only on the boundary condition $\beta$.

\subsection{Covariance Structure}\label{s:Cov}

In this section we study the covariance structure of the martingale difference terms $\Delta M_t(z)$ in \eqref{e:dm}, explicitly given by
\begin{align}\label{e:dMt}
\Delta M_t(z)
=n\int_{z-1/n}^z \sum_i\frac{1}{u-x_i(t)-e_i/n}\rd u-\bE_\bme\left[n\int_{z-1/n}^z \sum_i\frac{1}{u-x_i(t)-e_i/n}\rd u\right].
\end{align}
In this section, the time $t$ is fixed, for simplicity of notations, we simply write $x_i(t)$ as $x_i$ for $1\leq i\leq m(t)$. 
We can write $\Delta M_t(z)$ in terms of the measure $\mu_n=\sum_i e_i (\delta_{x_i+1/n}-\delta_{x_i})$, and it has the same  covariance structure as 
\begin{align*}
\left\{\sqrt n \int_{z-1/n}^z\int  \frac{\sum_i e_i(\delta_{x_i+1/n}-\delta_{x_i})}{u-x}\rd x\rd u\right\}_{z\in \bC\setminus [\fa_1(t), \fb_1(t)]\cup[\fa_2(t), \fb_2(t)]\cup\cdots [\fa_{r(t)}(t), \fb_{r(t)}(t)]}.
\end{align*}
Explicitly, after integrating out $u$ in the above expression, we get
\begin{align*}\begin{split}
\sqrt n \int_{z-1/n}^z\int  \frac{\sum_i e_i(\delta_{x_i+1/n}-\delta_{x_i})}{u-x}\rd x\rd u
&=-\frac{1}{\sqrt n}\sum_i e_i \ln \left(1-\frac{1}{(z-x_i-1/n)^2}\right)=:\sum_{i}e_i\theta(x_i;z).
\end{split}\end{align*}

To use Proposition \ref{p:firstod}, we need a deformed version of \eqref{e:defLtnew}. Fix any  numbers $\bms=\{s_1, s_2,\cdots, s_q\}\subset \bR$, and complex numbers $\bmv=\{v_1, v_2, \cdots, v_q\}\subset \bC\setminus [\fa_1(t), \fb_1(t)]\cup[\fa_2(t), \fb_2(t)]\cup\cdots [\fa_{r(t)}(t), \fb_{r(t)}(t)]$, we define a deformed version of \eqref{e:defLtnew}:
\begin{align}\label{e:defLtnew}
\tilde \bP_{\bms,\bmv}(\bme)\deq \frac{1}{\tilde Z_t(\bmx)}\frac{V(\bmx+\bme/n)}{V(\bmx)}\prod_{i=1}^m\tilde \phi_t^+(x_i;\beta,t)^{e_i}\phi_t^-(x_i;\beta,t)^{1-e_i}\exp\left\{\sum_{1\leq i,j\leq m}\frac{e_ie_j}{n^2}\kappa(x_i,x_j;\beta,t)\right\},
\end{align}
where 
\begin{align}\label{e:defvphitnew}
\tilde \phi_t^+(z;\beta,t)=\tilde \varphi_t^+(z;\beta,t)e^{\frac{1}{n}\psi_t^+(z;\beta,t)},\quad \tilde \varphi_t^+(z;\beta,t)= \varphi_t^+(z;\beta,t)e^{\sum_{k}s_k \theta(z; v_k)}.
\end{align}
Then the expectation with respect to the deformed measure $\tilde \bP_{\bms, \bmv}$, can be rewritten as the expectation with respect to the original measure \eqref{e:defLtnew} in the following way:
\begin{align}\label{e:rewrite}
\bE_{\tilde \bP_{\bms, \bmv}}\left[ \cdot\right]
=\frac{\bE_\bme\left[e^{\sum_{k}s_k \sum_i e_i\theta(x_i; v_k)}(\cdot)\right]}{\bE_\bme\left[e^{\sum_{k}s_k \sum_i e_i\theta(x_i;v_k)}\right]}.
\end{align}

Then for the deformed measure $\tilde \bP_{\bms, \bmv}$, Proposition \ref{p:difmeasure}  gives
\begin{align}\label{e:EPT}
\bE_{\tilde \bP_{\bms, \bmv}}\left[ \frac{\rd\mu_n}{u-x}\right]=
\bE_{\tilde \bP_{\bms, \bmv}}\left[ \frac{1}{n}\sum_{i}\frac{e_i}{(u-x_i-1/n)(u-x_i)}\right]=\frac{1}{2\pi \ri}\oint_{\cin} \frac{\del_w\ln \cB_t^{\bms,\bmv}}{w-u}\rd w +\OO\left(\frac{1}{ n}\right),
\end{align}
where the contour $\omega_-$ encloses $z^{-1}(\{x:(x,t+)\in\fP \})$, but not $u, v_1, v_2, \cdots, v_k$, and 
\begin{align}\begin{split}\label{e:bsvexp}
\ln \cB_t^{\bms,\bmv}(w)
&=\ln (e^{\tilde m_t}\varphi_t^+ e^{\sum_k s_k\theta(v_k)}+\varphi_t^-)\\
&=\ln \cB_t(w)+\ln \left(1+\frac{e^{\tilde m_t}\varphi_t^+ (e^{\sum_k s_k\theta(w; v_k)}-1)}{e^{\tilde m_t}\varphi_t^++\varphi_t^-}\right)\\
&=\ln \cB_t(w)+\frac{f_t(w;\beta,t)}{f_t(w;\beta,t)+1}\sum_k s_k\theta(w; v_k)+\OO\left(\frac{1}{ n}\right).
\end{split}\end{align}
By rewriting \eqref{e:EPT} in terms of the original measure \eqref{e:defLtnew} using \eqref{e:rewrite}, we get
\begin{align}\label{e:covar}
\frac{\bE_\bme\left[e^{\sum_{k}s_k \sum_i e_i\theta(x_i; v_k)} \frac{1}{\sqrt n}\sum_{i}\frac{e_i}{(u-x_i-1/n)(u-x_i)}\right]}{\bE_\bme\left[e^{\sum_{k}s_k \sum_i e_i\theta(x_i;v_k)}\right]}
=\frac{\sqrt n}{2\pi \ri}\oint_{\cin} \frac{\del_w\ln \cB_t^{\bms,\bmv}}{w-u}\rd w +\OO\left(\frac{1}{\sqrt n}\right).
\end{align}
Then we can plug \eqref{e:bsvexp} into \eqref{e:covar} and integrate both sides  from $v_0-1/n$ to $v_0$, and 
 \begin{align}\begin{split}\label{e:covar2}
\frac{\bE_\bme\left[e^{\sum_{k}s_k \sum_i e_i\theta(x_i; v_k)} \sum_{i}e_i \theta(x_i;v_0)\right]}{\bE_\bme\left[e^{\sum_{k}s_k \sum_i e_i\theta(x_i;v_k)}\right]}
=\frac{\sqrt n}{2\pi \ri}\oint_{\cin} \frac{\del_w\ln \cB_t}{w-v_0}\rd w\\
+\sum_k \frac{\sqrt n s_k}{2\pi \ri}\oint_{\cin} \frac{f_t(w;\beta,s)}{f_t(w;\beta,s)+1}\frac{\theta(w;v_k)}{(w-v_0)^2}\rd w
 +\OO\left(\frac{1}{\sqrt n}\right).
\end{split}\end{align}

The derivatives
\begin{align*}
\left.\del_{s_1, s_2,\cdots, s_q}\frac{\bE_\bme\left[e^{\sum_{k}s_k \sum_i e_i\theta(x_i; v_k)} \sum_{i}e_i \theta(x_i;v_0)\right]}{\bE_\bme\left[e^{\sum_{k}s_k \sum_i e_i\theta(x_i;v_k)}\right]}\right|_{s_1=s_2=\cdots=s_q=0}
\end{align*}
gives the joint cumulant of 
\begin{align*}
\left\{\sqrt n \Delta M_t(v_k)= \sum_i e_i\theta(x_i;v_k)-\bE_\bme\left[\sum_i e_i\theta(x_i;v_k)\right]\right\}_{k=0}^q.
\end{align*}
Using Taylor expansion, and \eqref{e:covar2} we have that
\begin{align*}\begin{split}
\bE_\bme\left[\sqrt n \Delta M_t(v_0), \sqrt n \Delta M_t(v_1)\right]
=\left.\del_{s_1}\frac{\bE_\bme\left[e^{s_1 \sum_i e_i\theta(x_i; v_1)} \sum_{i}e_i \theta(x_i;v_0)\right]}{\bE_\bme\left[e^{s_1\sum_i e_i\theta(x_i;v_1)}\right]}\right|_{s_1=0}\\
=\frac{1}{2\pi \ri}\oint_{\cin} \frac{f_t(w;\beta,t)}{f_t(w;\beta,t)+1}\frac{1}{(v_1-w)^2}\frac{1}{(v_0-w)^2}\rd w+\OO\left(\frac{1}{\sqrt n}\right),
\end{split}\end{align*}
where the contour $\omega_-$ encloses $z^{-1}(\{x:(x,t+)\in\fP \})$, but not $v_0, v_1$,
and higher order cumulants vanish
\begin{align*}
\left.\del_{s_1, s_2,\cdots, s_q}\frac{\bE_\bme\left[e^{\sum_{k}s_k \sum_i e_i\theta(x_i; v_k)} \sum_{i}e_i \theta(x_i;v_0)\right]}{\bE_\bme\left[e^{\sum_{k}s_k \sum_i e_i\theta(x_i;v_k)}\right]}\right|_{s_1=s_2=\cdots=s_q=0}=\OO\left(\frac{1}{\sqrt n}\right),
\end{align*}
for $q\geq 2$.

\subsection{Convergence to Gaussian Processes}

In this section, we prove that after proper normalization, $\{\tilde m_t(z_t;\beta,s)\}_{s\leq t\leq T}$ converges to a Gaussian Process, characterized by a stochastic differential equation, given by the limit of \eqref{e:mainEq}. To state the main result, we need to introduce more notations. For any $Z\in \gamma_s^{\beta,s}$, the characteristic flow $\{Z_t\}_{s\leq t\leq T}$ starting from $Z$ at time $s$ (as in \eqref{e:ccff}) will meet the bottom boundary of $\gamma_t^{\beta,s,+}$. This is when $z_t=z(Z_t)\in \bR$, then $\Im[z_t]$ will change sign. We denote the time $t(Z)$, when $Z_t\in \gamma_t^{\beta,s,+}$ gets  close to $z^{-1}(\{x: (x,t+)\in \fP\})$. Fix small $\fc>0$, let
\begin{align*}
t(Z)=\inf_{s\leq t\leq T}\{t: \dist(Z_t, z^{-1}(\{x: (x,t+)\in \fP\})\leq \fc\}.
\end{align*}
Then for any $s\leq t\leq t(Z)$, $Z_t\in \gamma_t^{\beta,s,+}$ stays away from $z^{-1}(\{x: (x,t+)\in \fP\})$.

\begin{theorem}\label{t:CLT}
Let $\{\bmx_t\}_{t\in[s,T]\cap\bZ_n}$ be the nonintersecting Bernoulli random walk with generator \eqref{e:defLtnew} $\{\bmx_t\}_{t\in[s,T]\cap\bZ_n}$ starting at $\bmx_s=\bmx$. We denote  its empirical particle density and Stieltjes transform as
\begin{align*}
\rho_t(x;\bmx_t)=\sum_{i=1}^{m(t)}x_i(t),\quad \tilde m_t(z;\beta,s)=\int\frac{\rho_t(x;\bmx_t)}{z-x}\rd x,
\end{align*}
where $\beta=\beta(x;\bmx)$.

The random processes $\{n(\tilde m_t(z_t(Z);\beta,s)-m_t(z_t(Z);\beta,s))\}_{s\leq t\leq t(Z), Z\in \gamma_s^{\beta,s}}$ converge weakly towards a Gaussian process $\{{\sf G}_t(Z)\}_{s\leq t\leq t(Z), Z\in \gamma_s^{\beta,s}}$ with zero initial data, which is the unique solution of the system of stochastic differential equations
\begin{align}\begin{split}\label{e:limitprocess}
\rd\sfG_t(Z)&= \rd\sfW_t(Z)-\sfG_t(Z)\frac{\del_t \rd z_t}{\rd z_t}+\frac{1}{2\pi\ri}\oint_{\omega} \sfK_t(z_t(W),z_t;\beta,s)\sfG_t(W)\rd z_t(W)+\sfE_t(z_t;\beta,s),
\end{split}\end{align}
where $z_t=z_t(Z)$, the contour $\{z_t(W): W\in \omega\subset \gamma_t^{\beta,s,+}\}$  encloses $z^{-1}(\{x:(x,t+)\in \fP\})$ and $z_t$,
$\sfK_t(w,z_t;\beta,s)$ is given by
\begin{align*}
\sfK_t(w,z_t;\beta,s)=\frac{f_t(w;\beta,s)}{f_t(w;\beta,s)+1}\frac{1}{(w-z_t)^2}-\frac{1}{2\pi\ri}\oint_{\omega_-}\frac{f_t(u;\beta,s)}{f_t(u;\beta,s)+1}\frac{\int^w K(u,x;h_t^*,t)\rd x}{(u-z_t)^2}\rd u, 
\end{align*}
 $\sfE_t(z_t;\beta,s)$ is give by
\begin{align*}\begin{split}
    &\phantom{{}={}}\sfE_t(z_t;\beta,s)=\frac{1}{2}\del_t^2 m_t(z_t;\beta, s)\\
    &+\frac{1}{2\pi \ri}\oint_{\cin}\left(\frac{e^{ m_t}\varphi_t^+}{\cB^*_t}\left(\del_w  m_t+\psi_t^++\frac{1}{z_t-w}+\del_t m_t-
\int K(w,x;h_t^*,t)\del_t h^*_t\rd x\right) -\frac{\del_w\cB^*_t}{2\cB^*_t}\right)\frac{\rd w}{(w-z_{t})^2},
\end{split}\end{align*}
and $\{\sfW_t(Z)\}_{s\leq t\leq t(Z), Z\in \gamma_s^{\beta,s}}$ is a centered Gaussian process with $\overline{W_t(Z)}=W_t(\bar Z)$, and covariance structure given by
\begin{align}\begin{split}\label{e:covW1}
&\langle \sfW(Z), \sfW(Z')\rangle_t=\int_s^t \frac{1}{2n\pi \ri}\oint_{\cin} \frac{f_r(w;\beta,s)}{f_r(w;\beta,s)+1}\frac{1}{(z_t(Z)-w)^2}\frac{1}{(z_t(Z')-w)^2}\rd w\rd r,
\end{split}\end{align}
$\omega_-$ encloses $z^{-1}(\{x:(x,t+)\in \fP\})$ but not $z_t(Z), z_t(Z')$.
\end{theorem}

We divide the proof of Theorem \ref{t:CLT} into two steps. 
In step one, we prove that the martingale term in \eqref{e:mainEq} converges weakly to a centered complex Gaussian process.
In step two we prove  the tightness of the processes $\{n(\tilde m_t(z_t(Z);\beta,s)-m_t(z_t(Z);\beta,s))\}_{s\leq t\leq t(Z), Z\in \gamma_s^{\beta,s}}$ as $n$ goes to infinity and  the subsequential limits of $\{n(\tilde m_t(z_t(Z);\beta,s)-m_t(z_t(Z);\beta,s))\}_{s\leq t\leq t(Z), Z\in \gamma_s^{\beta,s}}$ solve the stochastic differential equation \eqref{e:limitprocess}. Theorem \ref{t:CLT} follows from this fact and the uniqueness of the solution to \eqref{e:limitprocess}.

We denote the martingale starting at $M_s(Z)=0$,
\begin{align*}
M_t(Z)\deq \sum_{r\in [s, t)\cap \bZ_n}  \Delta M_r(z_{r+1/n}(Z)),\quad t\in [s,T]\cap \bZ_n,
\end{align*}
and its (predictable) quadratic variation as
\begin{align}\label{e:quadraticv}
\langle M(Z), M(Z')\rangle_t=\sum_{r\in [s,t)\cap \bZ_n} \bE[\Delta M_r(z_{r+1/n}(Z)) \Delta M_r(z_{r+1/n}(Z'))|\bmx_{r}].
\end{align}

In the following, we first show the weak convergence of the martingales $\{M_t(Z)\}_{s\leq t\leq t(Z), Z\in \gamma_s^{\beta,s}}$ to a centered Gaussian process.
We recall from Section \ref{s:Cov}, the covariance structure of the martingale difference terms.
\begin{proposition}\label{p:martingalebd}
The martingale difference terms $\Delta M_t(z)$ are asymptotically joint Gaussian, with covariance given by
\begin{align*}
n\bE[\Delta M_t(z)\Delta M_t(z')|\bmx_t]=\frac{1}{2\pi \ri}\oint_{\cin} \frac{\tilde f_t(w;\beta,s)}{\tilde f_t(w;\beta,s)+1}\frac{1}{(z-w)^2}\frac{1}{(z'-w)^2}\rd w+\OO\left(\frac{1}{\sqrt n}\right),
\end{align*}
where the contour encloses $z^{-1}(\{x:(x,t+)\in\fP\})$ but not $z,z'$, and for any even $p\geq 2$,
\begin{align}
\label{e:onepoint}&\bE[|\sqrt n\Delta M_t(z)|^p|\bmx_t]\leq \frac{\fC_p}{\langle z\rangle^{2p}},\quad \langle z\rangle \deq (1+|z|^2)^{1/2},\\
\label{e:twopoint}&\bE[|\sqrt n(\Delta M_t(z)-\Delta M_t(z'))|^p|\bmx_t]\leq \frac{\fC_p |z-z'|^p}{(\langle z\rangle \langle z'\rangle\min\{\langle z\rangle, \langle z'\rangle\})^{p}}.
\end{align}
\end{proposition}

We first show that the random processes $\{M_t(Z)\}_{s\leq t\leq t(Z), Z\in \gamma_s^{\beta,s}}$ are stochastically bounded: for any  $\varepsilon>0$ there exists a large constant $\fC>0$ such that 
\begin{align}\label{e:stochasticbd}
&\bP\left(\sup_{Z\in \gamma^{\beta,s}_s}\sup_{t\in [s,t(Z)]\cap \bZ_n}|M_{t}(Z)|\geq \fC\right)\leq \varepsilon.
\end{align}
Thanks to Doob's inequality and \eqref{e:onepoint}, for any $Z\in \gamma_s^{\beta,s}$ and $s\leq t\leq t(Z)$, we have
\begin{align}\label{e:ubbM}
\bE\left[\left|\sup_{t\in [s,t(Z)]\cap\bZ_n}M_t(Z)\right|^p\right]
\lesssim \fC_p\bE\left[\left|M_{t(Z)}(Z)\right|^p\right]
=\bE\left[\left|\sum_{t\in [s,t(Z))\cap \bZ_n}\Delta M_{t}(z_{t+1/n}(Z))\right|^p\right]
\lesssim (t(Z)-s)^{p/2},
\end{align}
where the implicit constant depends on $p$. Similarly, using Doob's inequality and \eqref{e:twopoint}, for any $Z,Z'\in \gamma_s^{\beta,s}$ with $t(Z)\leq t(Z')$, and $s\leq t\leq t(Z)$, we have
\begin{align}\begin{split}\label{e:diffubbM0}
&\phantom{{}={}}\bE\left[\left|\sup_{t\in [s,t(Z)]\cap\bZ_n}(M_t(Z)-M_t(Z'))\right|^p\right]
\lesssim \bE\left[\left|M_{t(Z)}(Z)-M_{t(Z)}(Z')\right|^p\right]\\
&=\bE\left[\left|\sum_{t\in [s,t(Z))\cap \bZ_n}\Delta M_{t}(z_{t+1/n}(Z))-\Delta M_{t}(z_{t+1/n}(Z))\right|^p\right]
\lesssim \sup_{s\leq t\leq t(Z)}\frac{(t(Z)-s)^{p/2} |z_t-z_t'|^p}{(\langle z_t\rangle \langle z_t'\rangle\min\{\langle z_t\rangle, \langle z_t'\rangle\})^{p}} ,
\end{split}\end{align}
where $z_t=z(Z_t)$ and $z_t'=z(Z_{t'})$.
To take a union bound for $Z\in \gamma_s^{\beta,s}$, we need to equip $\gamma_s^{\beta,s}$ a metric to make it compact. We recall that using the map \eqref{e:coveratr}, 
$\gamma_s^{\beta,s}$ can be identified with the gluing of two copies of $\Omega(\fP;\beta,s)$. This induces a metric on $\gamma_s^{\beta,s}$, using the metric of $\Omega(\fP;\beta,s)$ as a bounded subset of $\bR^2$. We denote this distance, by $\dist(Z,Z')$ for $Z,Z'\in \gamma_s^{\beta,s}$. It is easy to check, using the definition of characteristic flow, for any $s\leq t\leq t(Z)$, 
\begin{align}
\frac{ |z_t-z_t'|}{\langle z_t\rangle \langle z_t'\rangle\min\{\langle z_t\rangle, \langle z_t'\rangle\}} \lesssim \dist(Z,Z'),
\end{align}
and we can rewrite \eqref{e:diffubbM0} as
\begin{align}\label{e:diffubbM}
&\phantom{{}={}}\bE\left[\left|\sup_{t\in [s,t(Z)]\cap\bZ_n}(M_t(Z)-M_t(Z'))\right|^p\right]\lesssim (t(Z)-s)^{p/2}\dist(Z,Z')^p.
\end{align}

We take an infinite sequence of grids $\Pi_0\subset\Pi_1\subset \Pi_2\subset\cdots$ of $\gamma^{\beta,s}_s$, such that for any $Z\in \gamma^{\beta,s}_s$, there exists some $\pi_k(Z)\in \Pi_k$ with $\dist(Z,\pi_k(Z))\leq 2^{-k}$. We can take $\Pi_k$ to have cardinality $|\Pi_k|\lesssim 2^{2k}$. Let
\begin{align*}
U_0=\max_{\{Z\in \Pi_0\}}\left|\sup_{t\in[s,t(Z)]\cap \bZ_n}M_t(Z)\right|,\quad
U_k=\max_{\{Z,Z'\in \Pi_k: \dist(Z,Z')\leq 2^{-k}\}}\left|\sup_{t\in[s,t(Z)]\cap \bZ_n}(M_t(Z)-M_t(Z'))\right|,\quad k\geq 1.
\end{align*}
Using \eqref{e:ubbM}, and $|\Pi_0|\lesssim 1$, we have
\begin{align}\label{e:U0}
\bE[U_0^p]\leq \sum_{Z\in \Pi_0}\bE\left[\left|\sup_{t\in[s,t(Z)]\cap \bZ_n}M_t(Z)\right|^p\right]\lesssim (t(Z)-s)^{p/2}.
\end{align}
Using \eqref{e:diffubbM}, and noticing that $|\{Z,Z'\in \Pi_k: \dist(Z,Z')\leq 2^{-k}\}|\lesssim |\Pi_k|\lesssim 2^{2k}$, we have
\begin{align}\label{e:Uk}
\bE[U_k^p]\leq \sum_{\{Z,Z'\in \Pi_k: \dist(Z,Z')\leq 2^{-k}\}}\bE\left[\left|\sup_{t\in[s,t(Z)]\cap \bZ_n}(M_t(Z)-M_t(Z'))\right|^p\right]
\lesssim (t(Z)-s)^{p/2}2^{-(p-2)k}.
\end{align}
 For any $Z\in \gamma^{\beta,s}_s$, we can rewrite $M_t(Z)$ as the following dyadic sum
\begin{align}\begin{split}\label{e:dyadic}
\left|\sup_{t\in [s,t(Z)]\cap\bZ_n}M_t(Z)\right|&\leq
\left|\sup_{t\in [s,t(Z)]\cap\bZ_n}M_t(\pi_0(Z))\right|+\sum_{k\geq 0}\left|\sup_{t\in [s,t(Z)]\cap\bZ_n}(M_t(\pi_{k+1}(Z))-M_t(\pi_k(Z)))\right|\\
&
\leq U_0+\sum_{k\geq 1} U_k.
\end{split}
\end{align} 
By taking $L_p$ norm of \eqref{e:dyadic}, and using the triangle inequality of $L_p$ norm, we get
\begin{align}\begin{split}\label{e:pthm}
\bE\left[\sup_{Z\in \gamma^{\beta,s}_s}\left|\sup_{t(Z)\in [s,t(Z)]\cap\bZ_n}M_t(Z)(Z)\right|^p\right]^{1/p}
&\leq \bE[U_0^p]^{1/p}+\sum_{k\geq 1}\bE[U_k^p]^{1/p}\\
&\leq (t(Z)-s)^{1/2}\left(1+\sum_{k\geq 1}2^{-(1-2/p)k}\right)\lesssim(t(Z)-s)^{1/2},
\end{split}\end{align}
provided we take $p\geq 3$, where in the second line we used \eqref{e:U0} and \eqref{e:Uk}.
The claim \eqref{e:stochasticbd} follows from \eqref{e:pthm}, and a Markov's inequality.

For any $s\leq t\leq t(Z)$, we can rewrite \eqref{e:mainEq}
\begin{align}\begin{split}\label{e:mainEqcopy}
n (\tilde m_t- m_t)
&= M_t(Z)+\sum_{r\in [s,t]\cap \bZ_n}\left((\tilde m_r- m_r)\frac{\del_r \rd z_r}{\rd z_r}+\frac{1}{2\pi\ri}\oint_{\omega_+} \sfK_r(w,z_r;\beta,s)(\tilde m_t(w)-\del_z m_r(w))\rd w\right.\\
&\left.+\frac{1}{n}\sfE_r(z_r;\beta,s)+\OO(n(\tilde m_r-m_r)^2+(\tilde m_r-m_r)+1/n)\right).
\end{split}\end{align}
Using \eqref{e:stochasticbd}, and Gronwell's inequality, we can conclude that $n (\tilde m_t- m_t)$ is stochastically bounded,
for any  $\varepsilon>0$ there exists a large constant $\fC>0$ such that 
\begin{align}\label{e:stochasticbdm}
&\bP\left(\sup_{Z\in \gamma^{\beta,s}_s}\sup_{t\in [s,t(Z)]\cap \bZ_n}n|\tilde m_t(z_t(Z))-m_t(z_t(Z)|\geq \fC\right)\leq \varepsilon.
\end{align}

To show that the random processes $\{M_t(Z)\}_{s\leq t\leq t(Z), Z\in \gamma_s^{\beta,s}}$, and $\{\langle M(Z), M(Z)\rangle_t\}_{s\leq t\leq t(Z), Z\in \gamma_s^{\beta,s}}$ are tight, 
we apply the sufficient condition for tightness of \cite[Chapter 6, Proposition 3.26]{MR1943877}. We need to check the modulus conditions: for any $\varepsilon>0$ there exists a $\delta>0$ such that 
\begin{align}\begin{split}\label{e:modulus}
&\bP\left(\sup_{Z\in \gamma^{\beta,s}_s}\sup_{s\leq t\leq t'\leq t(Z), t'-t\leq \delta}|M_{t'}(Z)-M_t(Z)|\geq \varepsilon\right)\leq \varepsilon,\\
&\bP\left(\sup_{Z,Z'\in \gamma^{\beta,s}_s}\sup_{s\leq t\leq t'\leq t(Z)\wedge t(Z'), t'-t\leq \delta}\left|(\langle M(Z), M(Z')\rangle_{t'}-\langle M(Z), M(Z')\rangle_{t})\right|\geq \varepsilon\right)\leq \varepsilon.
\end{split}\end{align}
The statements \eqref{e:modulus} can be proven by a covering argument in the same way as \eqref{e:stochasticbd}. 

Thanks to , the quadratic variations $\langle M(Z), M(Z')\rangle_t$ as in \eqref{e:quadraticv}  converges weakly,
\begin{align}\begin{split}
&\phantom{{}={}}\label{e:cov} \langle M(Z), M(Z') \rangle_t
=\sum_{r\in [s,t)\cap \bZ_n} \bE[\Delta M_r(z_{r+1/n}(Z)) \Delta M_r(z_{r+1/n}(Z'))|\bmx_{r}]\\
&=\sum_{r\in[s,t)\cap \bZ_n}\frac{1}{2n\pi \ri}\oint_{\cin} \frac{\tilde f_r(w;\beta,s)}{\tilde f_r(w;\beta,s)+1}\frac{1}{(z_t(Z)-w)^2}\frac{1}{(z_t(Z')-w)^2}\rd w+\OO(1/n^{3/2})\\
&=\int_s^t \frac{1}{2n\pi \ri}\oint_{\cin} \frac{f_r(w;\beta,s)}{f_r(w;\beta,s)+1}\frac{1}{(z_t(Z)-w)^2}\frac{1}{(z_t(Z')-w)^2}\rd w\rd r+\OO(1/n^{1/2}),
\end{split}\end{align}
where in the last line, thanks to \eqref{e:stochasticbd}, we can replace $\tilde f_r(w;\beta,s)$ by $f_r(w;\beta,s)$.

We notice that $\overline{M_t(Z)}=M_t(\bar{Z})$. It follows from \cite[Chapter 7,  Theorem 1.4]{MR838085} and the modulus conditions \eqref{e:modulus},  the complex martingales $\{(M_t(Z))\}_{s\leq t\leq t(Z), Z\in \gamma_s^{\beta,s}}$ converge weakly towards a centered complex Gaussian process $\{(\sfW_t(Z))\}_{s\leq t\leq t(Z), Z\in \gamma_s^{\beta,s}}$, with quadratic variation given by \eqref{e:covW1}.

In step two we prove  the tightness of the processes $\{n(\tilde m_t(z_t(Z))-m_t(z_t(Z)))\}_{s\leq t\leq t(Z), Z\in \gamma_s^{\beta,s}}$ as $n$ goes to infinity and  the subsequential limits of $\{n(\tilde m_t(z_t(Z))-m_t(z_t(Z)))\}_{s\leq t\leq t(Z), Z\in \gamma_s^{\beta,s}}$ solve the stochastic differential equation \eqref{e:limitprocess}.
We apply the sufficient condition for tightness of \cite[Chapter 6, Proposition 3.26]{MR1943877}, and check the modulus condition: for any $\varepsilon>0$ there exists a $\delta>0$ such that 
\begin{align}\label{e:modulusm}
&\bP\left(\sup_{Z\in \gamma^{\beta,s}_s}\sup_{s\leq t\leq t'\leq t(Z)\wedge t(Z'), t'-t\leq \delta}(n(\tilde m_t(z_t(Z))-m_t(z_t(Z))-n(\tilde m_t(z_t(Z'))-m_t(z_t(Z')))\geq \varepsilon\right)\leq \varepsilon.
\end{align}
In \eqref{e:stochasticbdm} and \eqref{e:modulus}, we have proven that the process $\{n(\tilde m_t(z_t(Z))-m_t(z_t(Z))\}_{s\leq t\leq t(Z), Z\in \gamma_s^{\beta,s}}$ is stochastically bounded, and the martingale process $\{M_t(Z)\}_{s\leq t\leq t(Z), Z\in \gamma_s^{\beta,s}}$ is tight.
Therefore, on the event $\{\sup_{Z\in\gamma^{\beta,s}_s}\sup_{s\leq t\leq t(Z)}|n(\tilde m_t(z_t(Z))-m_t(z_t(Z))|\leq \fC\}$ and $\sup_{Z\in \gamma^{\beta,s}_s}\sup_{s\leq t\leq t'\leq t(Z), t'-t\leq \delta}|M_{t'}(Z)-M_t(Z)|\leq \varepsilon/2$, for any $s\leq t\leq t(Z)$ and $t\leq t'\leq t(Z)\vee t+\delta$, we have
\begin{align*}\begin{split}
&\phantom{{}={}}\left|n (\tilde m_{t'}- m_{t'})-n (\tilde m_{t}- m_{t})\right|\\
&= |M_{t'}(Z)-M_{t'}(Z)|+\sum_{r\in [t,t']\cap \bZ_n}\left|(\tilde m_r- m_r)\frac{\del_r \rd z_r}{\rd z_r}+\frac{1}{2\pi\ri}\oint_{\omega_+} \sfK_r(w,z_r;\beta,s)(\tilde m_t(w)-\del_z m_r(w))\rd w\right.\\
&\left.+\frac{1}{n}\sfE_r(z_r;\beta,s)+\OO(n(\tilde m_r-m_r)^2+(\tilde m_r-m_r)+1/n)\right|\\
&\leq |M_{t'}(Z)-M_{t'}(Z)|+\OO(\fC(t-t'))\leq \varepsilon.
\end{split}\end{align*}
The claim \eqref{e:modulus} follows provided we take $\delta$ small enough.

So far, we have proven that the random processes $\{n(\tilde m_t(z_t(Z))-m_t(z_t(Z))\}_{s\leq t\leq t(Z), Z\in \gamma_s^{\beta,s}}$ are stochastically bounded and tight. Without loss of generality, by passing to a subsequence, we assume that they weakly converge towards to a random process $\{\sfG_t(Z)\}_{s\leq t\leq t(Z), Z\in \gamma_s^{\beta,s}}$. The limit process satisfies the limit version of the stochastic difference equation \eqref{e:mainEq}. 
\begin{align*}\begin{split}
\rd\sfG_t(Z)&= \rd\sfW_t(Z)-\sfG_t(Z)\frac{\del_t \rd z_t}{\rd z_t}+\frac{1}{2\pi\ri}\oint_{\omega} \sfK_t(z_t(W),z_t;\beta,s)\sfG_t(W)\rd z_t(W)+\sfE_t(z_t;\beta,s),
\end{split}\end{align*}
where the contour $\{z_t(W): W\in \omega\subset \gamma_t^{\beta,s,+}\}$  encloses $z^{-1}(\{x:(x,t+)\in \fP\})$ and $z_t$.
This finishes the proof of Theorem \ref{t:CLT}.

The mean of the random process $\{n(\tilde m_t(z_t(Z);\beta,s)-m_t(z_t(Z);\beta,s))\}_{s\leq t\leq t(Z), Z\in \gamma_s^{\beta,s}}$ is nonzero. Instead of centering $\tilde m_t(z_t(Z);\beta,s)$ around $m_t(z_t(Z);\beta,s)$, which is from the solution of the complex Burgers equation, we can center it around its mean.  The following is an easy corollary of Theorem \ref{t:CLT}
\begin{corollary}\label{c:CLT}
Let $\{\bmx_t\}_{t\in[s,T]\cap\bZ_n}$ be the nonintersecting Bernoulli random walk with generator \eqref{e:defLtnew} $\{\bmx_t\}_{t\in[s,T]\cap\bZ_n}$ starting at $\bmx_s=\bmx$. We denote  its empirical particle density and Stieltjes transform as
\begin{align*}
\rho_t(x;\bmx_t)=\sum_{i=1}^{m(t)}x_i(t),\quad \tilde m_t(z;\beta,s)=\int\frac{\rho_t(x;\bmx_t)}{z-x}\rd x,
\end{align*}
where $\beta=\beta(x;\bmx)$.

The random processes $\{n(\tilde m_t(z_t(Z);\beta,s)-\bE[\tilde m_t(z_t(Z);\beta,s)])\}_{s\leq t\leq t(Z), Z\in \gamma_s^{\beta,s}}$ converge weakly towards a centered Gaussian process $\{{\sf G}_t^\circ(Z)\}_{s\leq t\leq t(Z), Z\in \gamma_s^{\beta,s}}$ with zero initial data, which is the unique solution of the system of stochastic differential equations
\begin{align}\begin{split}\label{e:limitprocess2}
\rd\sfG_t^\circ(Z)&= \rd\sfW_t(Z)-\sfG_t^\circ(Z)\frac{\del_t \rd z_t}{\rd z_t}+\frac{1}{2\pi\ri}\oint_{\omega} \sfK_t(z_t(W),z_t;\beta,s)\sfG_t^\circ(W)\rd z_t(W),
\end{split}\end{align}
where $z_t=z_t(Z)$, the contour $\{z_t(W): W\in \omega\subset \gamma_t^{\beta,s,+}\}$  encloses $z^{-1}(\{x:(x,t+)\in \fP\})$ and $z_t$,
$\sfK_t(w,z_t;\beta,s)$ is given by
\begin{align*}
\sfK_t(w,z_t;\beta,s)=\frac{f_t(w;\beta,s)}{f_t(w;\beta,s)+1}\frac{1}{(w-z_t)^2}-\frac{1}{2\pi\ri}\oint_{\omega_-}\frac{f_t(u;\beta,s)}{f_t(u;\beta,s)+1}\frac{\int^w K(u,x;h_t^*,t)\rd x}{(u-z_t)^2}\rd u, 
\end{align*}
 and $\{\sfW_t(Z)\}_{s\leq t\leq t(Z), Z\in \gamma_s^{\beta,s}}$ is a centered Gaussian process with $\overline{W_t(Z)}=W_t(\bar Z)$, and covariance structure given by
\begin{align*}\begin{split}
&\langle \sfW(Z), \sfW(Z')\rangle_t=\int_s^t \frac{1}{2n\pi \ri}\oint_{\cin} \frac{f_r(w;\beta,s)}{f_r(w;\beta,s)+1}\frac{1}{(z_t(Z)-w)^2}\frac{1}{(z_t(Z')-w)^2}\rd w\rd r,
\end{split}\end{align*}
$\omega_-$ encloses $z^{-1}(\{x:(x,t+)\in \fP\})$ but not $z_t(Z), z_t(Z')$.
%
%
\end{corollary}

\begin{remark}
The limiting equation \eqref{e:limitprocess2} depends only on the solution $f_t(Z;\beta,s)$ of the complex Burgers equation \eqref{e:bgcopy}, and independent of the correction term $\psi_t^+(z;\beta,t)$ as in \eqref{e:defvphit}. In Section \ref{s:solve}, we will solve for the error terms $E_s(\bmx)$ in our Ansatz \eqref{a:defE}. It turns out the random Bernoulli walk \eqref{e:randomwalk} which has the same law as the uniform lozenge tiling, is also of the 
form \eqref{e:defLtnew}. Therefore, we can conclude from Corollary \ref{c:CLT}, the centered height function has Gaussian fluctuation. 
\end{remark}
\section{Solve for the Error term}\label{s:solve}

In this section we solve for the error term $E_s(\bmx)$ using Proposition \ref{p:FK}.
We recall the Feynman-Kac formula from \eqref{e:FK}
\begin{align}\label{e:FKcopy}
E_s(\bmx)
=\bE\left[\left.\prod_{t\in [s,T)\cap \bZ_n}Z_t(\bmx_t)
\prod_{k:s<T_k<T}\frac{E_{T_k-}(\bmx_{T_k}')}{E_{T_k}(\bmx_{T_k})}
\right| \bmx_s=\bmx\right].
\end{align}

\subsection{Estimate the Partition Function}
Before analyzing \eqref{e:FKcopy}, we need to have an expression of those partition functions $Z_t(\bmx_t)$.
In this section we estimate those partition functions,
\begin{align}\label{e:defZtb}
Z_t(\bmx)=\sum_{\bme} a_t(\bme;\bmx),\quad \bmx=(x_1, x_2,\cdots, x_m)\in \fM_t(\fP).
\end{align}
Let  $\beta=\beta(x;\bmx)$, we recall from \eqref{e:newat0}, 
\begin{align}\label{e:newat1}
a_t(\bme;\bmx)&=\frac{V(\bmx+\bme/n)}{V(\bmx)}\prod_{j} \phi^+(x_j;\beta,t)^{e_i}\phi^-(x_j;\beta,t)^{1-e_i}
e^{\frac{1}{n^2}\sum_{i} e_i e_j \kappa(x_i,x_j;\beta,t) +\zeta(\beta,t)+\OO(1/n)},
\end{align}
where
\begin{align*}
\phi^+(x;\beta,t)=\prod_{b\in B_{t}(x)} (b-x-1/n)e^{ \chi(x;\beta,t)},\quad 
\phi^-(x;\beta,t)=\prod_{a\in A_{t}(x)}  (x-a).
\end{align*}
and
 \begin{align*}\begin{split}
\chi(x_i;\beta,t)&=
n^2\int \ln  g_t'(x;\beta,t)\Lambda(x-x_i)+n\int \del_t(\ln  g_t'(x;\beta,t))\Lambda(x-x_i)\rd x\\
\kappa(x_i,x_j;\beta,t)&=-\frac{1}{2}n^4\int K(x,y;\beta,t)\Lambda(x-x_i)\Lambda(y-x_j)\rd x\rd y,\\
\zeta(\beta,t)&=n\del_tY_t(\beta)+\frac{1}{2}\del_t^2Y_t(\beta),
\end{split}\end{align*}
the functions $ g_t'(x;\beta,t), K(x,y;\beta,t)$ are defined in \eqref{e:deftgt}, \eqref{e:deeKt}. 

We recall some notations from Section \ref{s:Gauss}, 
$\varphi_t^+=\varphi_t^+(z;\beta,t)=g_t'(z;\beta,t)\prod_{b\in B_t(z)} { (b-z)}$,
$\varphi_t^-=\varphi_t^-(z;\beta,t)=\prod_{a\in A_t(z)}(z-a)$, then 
$\phi^+(z;\beta,t)=\varphi_t^+ \exp(\psi_t^+(z;\beta,t)/n)$.
Let
$m_t=m_t(z;\beta,t)$ be the Stietjes transform of $\rho(x;\bmx)$, and 
$\cB_t=\cB_t(z;\beta,t)=e^{m_t}\varphi_t^++\varphi_t^-$. We prove that $Z_t(\beta)$ satisfies the following perturbation formula.

\begin{proposition}\label{p:lnZder}
Let  $\beta=\beta(x;\bmx)$, then $Z_t(\beta)$ as in \eqref{e:defZtb} satisfies
\begin{align*}
\ln Z_t(\beta+\delta \beta)
=\ln Z_t(\beta)+\int \del_\beta \ln Z_t(\beta)\delta\beta \rd x+\OO\left(\frac{1}{n^2}\right),
\end{align*}
where 
\begin{align*}\begin{split}
&\phantom{{}={}}\del_\beta \ln Z_t(\beta)=\frac{1}{2\pi\ri}\oint_\omega \ln \cB_t(z) \left( \frac{1}{2}\del_z \frac{\del\ln g_t(z;\beta,t)}{\del \beta }+\frac{\del (\del_t \ln  g_t'(z;\beta,t))}{\del \beta}\right)\rd z\\
&+\frac{1}{2\pi\ri}\oint_\omega \cE_1(z) \frac{\del\ln g_t(z;\beta,t)}{\del \beta }\rd z+\frac{1}{8\pi^2}\oint \ln \cB_t(z) \ln \cB_t(w)\frac{\del K(z,w;\beta,t)}{\del \beta}\rd z\rd w\\
&-
\frac{1}{4\pi \ri}\oint\frac{\del_w \cB_t}{ \cB_t}\frac{\rd w}{(w-x)^2}+
\frac{1}{2\pi \ri}\oint_{\omega}\del_w\left(\cE_0(w)-\frac{e^{ m_t}\varphi_t^+}{\cB_t(w-x)}\right)\frac{\rd w}{w-x}+\frac{1}{2}\frac{\del \del_t^2 Y_t(\beta)}{\del\beta}+\OO\left(\frac{1}{n}\right).
\end{split}\end{align*}
where the contour $\omega$ encloses $z^{-1}(\{x: (x,t+)\in \fP\})$,
and the error terms  depend continuously on $\beta$.
\end{proposition}

\begin{proof}
We can view $Z_t(\bmx)$ as a function of $x_j$ with $x_j$ for $1\leq j\leq m$, and $\beta$ fixed, and denote it by $\cC^{(j)}(z)$. Then we can use Proposition \ref{p:order2} to get $\OO(1/n)$ expansion of $\del_z \cC^{(j)}/\cC^{(j)}$.
Then Proposition \ref{p:order2} gives
\begin{align*}
\frac{\del_{z} \cC^{(j)}(z)}{\cC^{(j)}(z)}
&=\frac{1}{2\pi \ri}\oint_{\cout}\left(\frac{\del_w\cB_t^{(j)}}{\cB_t^{(j)}}+\frac{\del_w \cE_0}{n}\right)\frac{\rd w}{w-z}+\OO\left(\frac{1}{n^2}\right),
\end{align*}
where the contour $\omega_+$ encloses $z^{-1}(\{x:(x,t+)\in \fP\})$ and $z$, and
\begin{align}\label{e:bbt10}
\cE_0(w)=
\frac{e^{ m_t}\varphi_t^+}{\cB_t}\left(\del_w  m_t+\psi_t^++\int\frac{\del_x \del_t h_t(x;\beta_t,t)\rd x}{z-x}-
\int K(x,w;\beta_t,t)\del_t h_t(x;\beta_t,t)\rd x\right),
\end{align}
and 
\begin{align*}\begin{split}
\cB_t^{(j)}(w)
&=\prod_{i\neq j}\left(1+\frac{1/n}{w-x_j-1/n}\right)\varphi_t^+(w)+\varphi_t^-(w)\\
&=\prod_{i=1}^{m}\left(1+\frac{1/n}{w-x_i-1/n}\right)\varphi^+(w)\frac{w-x_j-1/n}{w-x_j}+\varphi^-(w)\\
&=e^{\tilde m_t}\varphi_t^+(w)\left(1-\frac{1}{n}\frac{1}{w-x_j}\right)+\varphi_t^-(w).
\end{split}\end{align*}
We can rewrite $\cB_t^{(j)}$ in terms of $\cB_t$, 
\begin{align*}
\frac{\del_w\cB_t^{(j)}}{\cB_t^{(j)}}
=\frac{\del_w\cB_t}{ \cB_t}+\del_w\left(\frac{\cB_t^{(j)}- \cB_t}{\cB_t}\right)+\OO\left(\frac{1}{n^2}\right)
=\frac{\del_w \cB_t}{ \cB_t}-\frac{1}{n}\del_w\left(\frac{e^{\tilde m_t}\varphi_t^+}{ \cB_t}\frac{1}{w-x_j}\right)+\OO\left(\frac{1}{n^2}\right).
\end{align*}
Putting them all together, we get
\begin{align}\begin{split}\label{e:dxjZ}
\frac{\del_{x_j} \cC^{(j)}(x_j)}{\cC^{(j)}(x_j)}
&=\frac{1}{2\pi \ri}\oint\frac{\del_w \cB_t}{ \cB_t}\frac{\rd w}{w-x_j}\\
&+
\frac{1}{2\pi \ri}\frac{1}{n}\oint_{\cout}\del_w\left(\cE_0(w)-\frac{e^{\tilde m_t}\varphi_t^+}{\cB_t(w-x_j)}\right)\frac{\rd w}{w-x_j}+\OO\left(\frac{1}{n^2}\right).
\end{split}\end{align}

We recall that $\beta$ also depends on $x_1, x_2,\cdots, x_m$, characterized by 
\begin{align}\label{e:defbetax}
\del_x \beta(x;\bmx)=\sum_{i=1}^m \bm1(x\in [x_i, x_i+1/n]).
\end{align} 
The partition function $
Z_t(\bmx)$ has another dependence on $x_j$ through $\beta$. The dependence on $\beta$ is from $\phi_t^+(x_i; \beta,t)$ ($\phi^-(x_i; \beta,t)$ does not depend on $\beta$), $\kappa(x_i,x_j;\beta,t)$ and $\zeta(\beta,t)$. We denote $``\del_\beta \ln Z_t(\bmx) "$, the derivative of $\ln Z_t(\bmx)$ with respect to $\beta$ with $\bmx$ fixed,
\begin{align}\begin{split}\label{e:dbZb}
{\phantom{\frac{1}{1}}}^{``}\frac{\del \ln Z_t(\bmx)}{\del \beta}^"
&=\bE_\bme\left[\sum_{i=1}^m e_i \frac{\del\ln \phi^+(x_i;\beta,t)}{\del \beta}+\sum_{i,j} \frac{e_ie_j}{n^2} \frac{\del \kappa(x_i, x_j;\beta,t)}{\del \beta}\right]+\frac{\del \ln \zeta(\beta,t)}{\del \beta}\\
&=\bE_\bme\left[\sum_{i=1}^m e_i \frac{\del\ln \chi(x_i; \beta,s)}{\del \beta}+\sum_{i,j} \frac{e_ie_j}{n^2} \frac{\del \kappa(x_i, x_j;\beta,t)}{\del \beta}\right]+\frac{\del \ln \zeta(\beta,t)}{\del \beta}.
\end{split}\end{align}

We can use Proposition \ref{p:second} to estimate \eqref{e:dbZb}. Under $\bE_\bme$, the measure $\mu_n$ as defined in \eqref{e:mun} satisfies
\begin{align}\label{e:bezx}
\bE_\bme\left[
\int  \frac{\rd\mu_n}{z-x}\right]
=\frac{1}{2\pi \ri}\oint_{\cin}\frac{\del_w\cB_t}{\cB_t}\frac{\rd w}{w-z}+\frac{1}{n}\frac{1}{2\pi \ri}\oint_{\cin}\frac{\del_w \cE_1(w)\rd w}{w-z}+\OO\left(\frac{1}{n^2}\right),
\end{align}
where 
\begin{align}\label{e:bbt10}
\cE_1(w)=
\frac{e^{ m_t}\varphi_t^+}{\cB_t}\left(\del_w  m_t+\psi_t^+ +\frac{1}{z-w}+\int\frac{\del_x \del_t h_t(x;\beta,t)\rd x}{z-x}-
\int K(x,w;\beta,t)\del_t h_t(x;\beta,t)\rd x\right).
\end{align}

For any function $\theta(x)$ which is analytic in a neighborhood of $z^{-1}(\{x: (x,t+)\in \fP\})$, we have
\begin{align*}
n\int \theta'(x)\Lambda(x-x_i)=
\left(\theta(x_i+1/n)+\frac{1}{2n}\theta'(x_i+1/n))-(\theta(x_i)+\frac{1}{2n}\theta'(x_i))\right)+\OO\left(\frac{1}{n^2}\right).
\end{align*}
Then we can use a contour integral to compute 
\begin{align}\begin{split}\label{e:ddL}
&\phantom{{}={}}\bE_\bme\left[\sum_{i=1}^m  e_i n^2\int \theta'(x)\Lambda(x-x_i)\rd x\right]
=n\bE_\bme\left[\frac{1}{2\pi\ri}\oint_\omega \frac{\theta(z)+\theta'(z)/(2n)}{z-x}\rd z\rd \mu_n \right]
+\OO\left(\frac{1}{n}\right)\\
&=-\frac{n}{2\pi \ri}\oint_{\omega}\frac{\del_w\cB_t}{\cB_t}\left(\theta(w)+\frac{\theta'(w)}{2n}\right)\rd w-\frac{1}{2\pi \ri}\oint_{\omega}\del_w \cE_1(w)\left(\theta(w)+\frac{\theta'(w)}{2n}\right)\rd w+\OO\left(\frac{1}{n}\right)\\
&=\frac{n}{2\pi \ri}\oint_{\omega} \ln\cB_t\left(\theta'(w)+\frac{\theta''(w)}{2n}\right)\rd w+\frac{1}{2\pi \ri}\oint_{\omega} \cE_1(w)\left(\theta'(w)+\frac{\theta''(w)}{2n}\right)\rd w+\OO\left(\frac{1}{n}\right),
\end{split}\end{align}
where the contour $\omega$ encloses $z^{-1}(\{x: (x,t+)\in \fP\})$.
By taking $\theta'(x)=n \del_\beta g_t'(x;\beta,t)$, $\del_\beta \del_t(\ln g_t'(x;\beta,t))$ and $\del_\beta K(x,y;\beta,t)$ in \eqref{e:ddL}, we can simplify \eqref{e:dbZb} as
\begin{align}\begin{split}\label{e:dZdb2}
{\phantom{\frac{1}{1}}}^{``}\frac{\del \ln Z_t(\bmx)}{\del \beta}^"
&=\frac{1}{2\pi\ri}\oint_\omega \ln \cB_t \left( n\frac{\del\ln g_t(z;\beta,t)}{\del \beta }+\frac{1}{2}\del_z \frac{\del\ln g_t(z;\beta,t)}{\del \beta }+\frac{\del (\del_t \ln  g_t'(z;\beta,t))}{\del \beta}\right)\rd z\\
&\frac{1}{2\pi\ri}\oint_\omega \cE_1(z) \frac{\del\ln g_t(z;\beta,t)}{\del \beta }\rd z+\frac{1}{8\pi^2}\oint_\omega \ln \cB_t(z) \ln \cB_t(w)\frac{\del K(z,w;\beta,t)}{\del \beta}\rd z\rd w\\
&+n\frac{\del \del_t Y_t(\beta)}{\del \beta}+\frac{1}{2}\frac{\del \del_t^2 Y_t(\beta)}{\del\beta}
+\OO(1/n),
\end{split}\end{align}
where the contour $\omega$ encloses $z^{-1}(\{x: (x,t+)\in \fP\})$.
The leading order term in \eqref{e:dZdb2} is
\begin{align}\label{e:leadingn}
&\phantom{{}={}}n\left(\frac{1}{2\pi\ri}\oint_\omega \ln \cB_t(z) \frac{\del\ln g_t(z;\beta,t)}{\del \beta }\rd z+\frac{\del \del_t Y_t(\beta)}{\del \beta}\right)
\end{align}
and other terms are of order $\OO(1)$. For the derivative of $Y_t(\beta)$, we can take time derivative on both sides of \eqref{e:deftg}, 
\begin{align}\begin{split}\label{e:dbdtY}
\frac{\del \del_t Y_t(\beta)}{\del \beta}
&=-\del_t \ln g_t(x;\beta,t)+\sum_{a\in A_{t}(x)} \frac{1}{x-a}\\
&=-\del_t \ln g_t(x;h_t^*(x;\beta,t),t)+\int \frac{\del \ln g_t(x; \beta,t)}{\del \beta}\del_t h_t^*(x;\beta,t)\rd x+\sum_{a\in A_{t}(x)} \frac{1}{x-a}\\
&=-(\del_t \ln g_t)(x;\beta,t)+\int \frac{\del \ln g_t(x; \beta,t)}{\del \beta}\del_t h_t^*(x;\beta,t)\rd x+\sum_{a\in A_{t}(x)} \frac{1}{x-a}\\
&=-(\del_t \ln g_t)(x;\beta,t)-\frac{1}{2\pi\ri}\oint_\omega \ln \cB_t(z) \frac{\del \ln g_t(z; \beta,t)}{\del \beta}\rd z+\sum_{a\in A_{t}(x)} \frac{1}{x-a},
\end{split}\end{align}
where the contour $\omega$ encloses $z^{-1}(\{x: (x,t+)\in \fP\})$.
By plugging \eqref{e:dbdtY} back into \eqref{e:leadingn} we get
\begin{align}\label{e:leadingn2}
\eqref{e:leadingn}=-n\left((\del_t \ln g_t)(x;\beta,t)+\sum_{a\in A_{t}(x)} \frac{1}{x-a}\right).
\end{align}
We recall the definition of $\beta(x;\bmx)$ from \eqref{e:defbetax}. For any function $U(\bmx)$ of $\bmx=(x_1, x_2, \cdots, x_m)$, we can view it as a function of $\beta$, i.e. $U(\beta)$. The derivative with respect to $\beta$ and the derivative with respect to $x_j$ are related 
\begin{align}\label{e:dbdU}
\del_{x_j}U(\bmx)=-\int_{x_j}^{x_j+1/n}\frac{\del U}{\del \beta} \rd x.
\end{align}
Using the relation \eqref{e:dbdU}, we can view $Z_t(\bmx)$ as a function of $\beta$, its functional derivative with respect to $\beta$ follows from combining \eqref{e:dxjZ}, \eqref{e:dZdb2} and \eqref{e:leadingn2}. Thanks to \eqref{e:B2}, the leading order term in \eqref{e:dxjZ} 
\begin{align*}
\frac{1}{2\pi \ri}\oint_\omega\frac{\del_w \cB_t}{ \cB_t}\frac{\rd w}{w-x_j}
=-(\del_t \ln g_t)(x_j;\beta,t)+\sum_{a\in A_{t}(x_j)} \frac{1}{x_j-a},
\end{align*}
cancels with \eqref{e:leadingn2}.
%
%
We conclude from combining \eqref{e:dxjZ}, \eqref{e:dZdb2},  \eqref{e:dbdtY} and \eqref{e:dbdU}
\begin{align*}\begin{split}
&\phantom{{}={}}\del_\beta \ln Z_t(\beta)=\frac{1}{2\pi\ri}\oint_\omega \ln \cB_t(z) \left( \frac{1}{2}\del_z \frac{\del\ln g_t(z;\beta,t)}{\del \beta }+\frac{\del (\del_t \ln  g_t'(z;\beta,t))}{\del \beta}\right)\rd z\\
&+\frac{1}{2\pi\ri}\oint_\omega \cE_1(z) \frac{\del\ln g_t(z;\beta,t)}{\del \beta }\rd z+\frac{1}{8\pi^2}\oint \ln \cB_t(z) \ln \cB_t(w)\frac{\del K(z,w;\beta,t)}{\del \beta}\rd z\rd w\\
&-
\frac{1}{4\pi \ri}\oint\frac{\del_w \cB_t}{ \cB_t}\frac{\rd w}{(w-x)^2}+
\frac{1}{2\pi \ri}\oint_{\omega}\del_w\left(\cE_0(w)-\frac{e^{ m_t}\varphi_t^+}{\cB_t(w-x)}\right)\frac{\rd w}{w-x}+\frac{1}{2}\frac{\del \del_t^2 Y_t(\beta)}{\del\beta}+\OO\left(\frac{1}{n}\right).
\end{split}\end{align*}
This finishes the proof of Proposition \ref{p:lnZder}.

\end{proof}

\subsection{Solve for the Error term}
Instead of solving \eqref{e:FKcopy} directly, we make another ansatz. 
The formula of the error term $E_s(\bmx)$ as in \eqref{e:FKcopy}, is expressed in terms of the nonintersecting Bernoulli random walk $\{\bmx_t\}_{t\in[s,T)\cap \bZ_n}$ starting from $\bmx_s=\bmx$. In Section \ref{s:Gauss}, we have proven that the empirical height functions of $\bmx_t$ have Gaussian fluctuation around $h_t^*(x;\beta,s)$. Their difference is of order $\OO(1/n)$. We make the following ansatz that the leader order term of $E_s(\bmx)$ is given by the expression \eqref{e:FKcopy} with $\bmx_t$ replaced by $h_t^*(x;\beta,s)$.

\begin{ansatz}\label{a:defE1}
For any time $s\in [0,T)\cap \bZ_n$, and any particle configuration $(x_1,x_2,\cdots,x_{m})\in {\mathfrak M}_{s}({\mathfrak P})$, let $\beta=\beta(x;\bmx)$. We make the following ansatz
\begin{align}\begin{split}\label{e:ansatz2}
&E_{s}(x_1, x_2, \cdots, x_m)=e^{F_s(\beta)}E_s^{(1)}(x_1,x_2,\cdots, x_m),\\
&F_s(\beta)\deq \ln\left(\prod_{t\in [s,T)\cap \bZ_n}Z_t(h_t^*)\prod_{k:s<T_k<T}\frac{E_{T_k-}(h_{T_k}^*)}{E_{T_k}(h_{T_k}^*)}\right),
\end{split}\end{align}
where $h_t^*=h_t^*(x;\beta,s)$.
\end{ansatz}


To rewrite the equation \eqref{e:cEeq} of $E_s(\bmx)$ in terms of $F_s(\beta)$, we also need $F_{s-}$ corresponding to $E_{s-}$. let  $(x'_1,x'_2,\cdots,x'_{m'})$ be the particle configuration from $\bmx$ by removing particles created at time $s$ and $\beta'=\beta(x;\bmx')$. The height function $\beta$ is an extension of $\beta'$. By slightly abuse of notation, we will not distinguish them.
 We define
\begin{align}\begin{split}\label{e:ansatz3}
&E_{s-}(x'_1, x'_2, \cdots, x'_{m'})=e^{F_{s-}(\beta)}E_{s-}^{(1)}(x'_1, x'_2, \cdots, x'_{m'}),\\
& F_{s-}(\beta)\deq \ln\left(\prod_{t\in [s,T)\cap \bZ_n}Z_t(h_t^*)\prod_{k:s\leq T_k<T}\frac{E_{T_k-}(h_{T_k}^*)}{E_{T_k}(h_{T_k}^*)}\right).
\end{split}\end{align}
Comparing \eqref{e:ansatz2} and \eqref{e:ansatz3}, we always have $E^{(1)}_s(\beta)=E_{s-}^{(1)}(\beta)$. In this section we show that \eqref{e:ansatz2} is a good approximation.
\begin{proposition}\label{p:Et1bb}
For any time $s\in [0,T)\cap \bZ_n$, and any particle configuration $(x_1,x_2,\cdots,x_{m})\in {\mathfrak M}_{s}({\mathfrak P})$, let $\beta=\beta(x;\bmx)$. 
Then $E_s^{(1)}(\beta)$ as in Ansatz \ref{e:defZtb} satisfies
\begin{align*}
\del_\beta \ln E_s^{(1)}(\beta)=\OO\left(\frac{1}{n}\right),
\end{align*}
and especially 
\begin{align}\label{e:Et1bb}
\frac{E_t^{(1)}(\bmx+\bme/n)}{E_t^{(1)}(\bmx)}=e^{\OO(1/n)}.
\end{align}
\end{proposition}

By using the ansatz \eqref{e:ansatz2}, we can rewrite the equation \eqref{e:cEeq} of $E_s(\bmx)$ as
\begin{align}\label{e:E1ecc}
E^{(1)}_{s}(\bmx)
=\sum_{\bme\in\{0,1\}^m}a_t(\bme;\bmx)e^{F_{(s+1/n)-}(\bmx+\bme/n)-F_s(\bmx)}
E^{(1)}_{(s+1/n)-}(\bmx+\bme/n).
\end{align}
Since $h_t^*(x;\beta,s)$ depends on $\beta$ smoothly, and thanks to Proposition \ref{p:lnZder}, $\del_\beta \ln Z_t(h_t^*)$ is of order $\OO(1)$.  Moreover, the ratio $E_{t-}((h_t^*)')/E_t(h_t^*)$ as in \eqref{e:split}, depends smoothly on $\beta$, and $\del_\beta \ln E_{t-}((h_t^*)')/E_t(h_t^*)$ is of order $\OO(1)$.  
We conclude that $\del_\beta F_s(\beta)$ and $\del_\beta F_{s-}(\beta)$ are of order $\OO(n)$.
Explicitly, we can expand the difference  $F_{(s+1/n)-}(\bmx+\bme/n)-F_s(\bmx)$ around $h_{s+1/n}^*$
\begin{align}\begin{split}\label{e:Fdidff}
&\phantom{{}={}}F_{(s+1/n)-}(\bmx+\bme/n)-F_s(\bmx)
=F_{(s+1/n)-}(\bmx+\bme/n)-F_{(s+1/n)-}(h^*_{s+1/n}) -\ln Z_s(\beta)\\
&=\int \del_\beta F_{(s+1/n)-}(h_{s+1/n}^*) (\beta(x;\bmx+\bme/n)-h_{s+1/n}^*)\rd x-\ln Z_s(\beta)+\OO\left(\frac{1}{n}\right)\\
&=\int \del_\beta F_{(s+1/n)-}(h_{s+1/n}^*) (\beta(x;\bmx+\bme/n)-\beta(x;\bmx))\rd x\\
&-\frac{1}{n}\int \del_\beta F_{(s+1/n)-}(h_{s+1/n}^*) \del_s h_s^*\rd x-\ln Z_s(\beta)+\OO\left(\frac{1}{n}\right)\\
\end{split}\end{align}

By plugging \eqref{e:Fdidff} into \eqref{e:E1ecc}, we can rewrite the recursion of $E^{(1)}_{s}(\bmx)$ as
%
\begin{align*}
E^{(1)}_{s}(\bmx)
=\sum_{\bme\in\{0,1\}^m}a_s^{(1)}(\bme;\bmx)
E^{(1)}_{(s+1/n)-}(\bmx+\bme/n),
\end{align*}
where 
\begin{align}\label{e:defas}
a_s^{(1)}(\bme;\bmx)=\frac{a_s(\bme;\bmx)e^{ \int \del_\beta F_{(s+1/n)-}(h_{s+1/n}^*) (\beta(x;\bmx+\bme/n)-\beta(x;\bmx))\rd x+\OO(1/n)}}{ Z_s(\bmx) e^{\frac{1}{n}\int \del_\beta F_{(s+1/n)-}(h_{s+1/n}^*) \del_s h_s^*\rd x}},
\end{align}
and the new partition function is given by
\begin{align}\begin{split}\label{e:Zt1}
Z_s^{(1)}(\bmx)
&=\sum_{\bme\in\{0,1\}^m}\frac{a_s(\bme;\bmx)e^{ \int \del_\beta F_{(s+1/n)-}(h_{s+1/n}^*) (\beta(x;\bmx+\bme/n)-\beta(x;\bmx))\rd x+\OO(1/n)}}{ Z_s(\bmx) e^{\frac{1}{n}\int \del_\beta F_{(s+1/n)-}(h_{s+1/n}^*) \del_s h_s^*\rd x}}\\
&=\frac{\bE_\bme \left[e^{ -\sum_i e_i \int \del_\beta F_{(s+1/n)-}(h_{s+1/n}^*) \Lambda(x-x_i)\rd x+\OO(1/n)}\right]}{e^{\frac{1}{n}\int \del_\beta F_{(s+1/n)-}(h_{s+1/n}^*) \del_s h_s^*\rd x}}.
\end{split}\end{align}
We can use the following Proposition, which follows from the loop equation Proposition \ref{p:difmeasure}, to get the $1/n$-expansion of $Z_s^{(1)}$.
\begin{proposition}\label{p:deform}
For any test function $\theta(x;\beta)$, which is analytic in a neighborhood of $z^{-1}(\{x: (x,s+)\in \fP\})$, we have 
\begin{align}\label{e:cdef}
\bE_\bme \left[e^{n\sum_{i }e_i \int \theta'(x;\beta)\Lambda(x-x_i)\rd x}\right]
=e^{-\int \theta'(x;\beta) \del_s h_s^*\rd x +\OO(1/n)},
\end{align}
where $\bE_\bme$ is with respect to the measure $\bP_\bme=a_s(\bme;\bmx)/Z_s(\bmx)$, and the error term $\OO(1/n)$ depends smoothly on $\beta$.
\end{proposition}

\begin{proof}
We can rewrite the exponent in the righthand side of \eqref{e:cdef} as
\begin{align*}
n\sum_{i }e_i \int \theta'(x;\beta)\Lambda(x-x_i)\rd x= \sum_{i }e_i (\theta(x_i+1/n;\beta)-\theta(x_i;\beta))+\OO\left(\frac{1}{n}\right).
\end{align*}
We use the interpolation to compute the expectation 
\begin{align}\label{e:taue0}
\ln\bE_\bme\left[e^{\sum_{i }e_i (\theta(x_i+1/n;\beta)-\theta(x_i;\beta)) }\right]
=\int_0^1 \del_\tau \ln\bE_\bme\left[ e^{\tau\sum_{i }e_i (\theta(x_i+1/n;\beta)-\theta(x_i;\beta)) }\right]\rd \tau.
\end{align}
The righthand side can be written as the expectation of $\sum_{i }e_i (\theta(x_i+1/n;\beta)-\theta(x_i;\beta))$ under the deformed measure 
\begin{align}\label{e:taue}
\del_\tau \ln\bE_\bme\left[e^{\tau\sum_{i }e_i (\theta(x_i+1/n;\beta)-\theta(x_i;\beta))}\right]
=\bE_\bme^\tau \left[\sum_i e_i (\theta(x_i+1/n;\beta)-\theta(x_i;\beta))\right],
\end{align}
where
\begin{align*}
\bE_\bme^\tau[\cdot]
=\frac{\bE_\bme\left[e^{\tau \sum_i e_i (\theta(x_i+1/n;\beta)-\theta(x_i;\beta)) }(\cdot)\right]}
{\bE_\bme\left[e^{\tau \sum_i e_i (\theta(x_i+1/n;\beta)-\theta(x_i;\beta)) }\right]}.
\end{align*}
The deformed measure $\bE_\bme^\tau$ corresponds to change the weight $\phi^+(z;\beta,s)$ by a factor $e^{(\theta(z+1/n;\beta)-\theta(z;\beta))}$.
Using Proposition \ref{p:second}, we have the $1/n$-expansion of \eqref{e:taue}, which gives us an $1/n$-expansion of  \eqref{e:taue0}
\begin{align*}\begin{split}
&\phantom{{}={}}\ln\bE_\bme\left[e^{\sum_i e_i (\theta(x_i+1/n;\beta)-\theta(x_i;\beta))}\right]\\
&=\int_0^1\bE_\bme^\tau \left[\sum_i e_i (\theta(x_i+1/n;\beta)-\theta(x_i;\beta)) \right]\rd \tau
=\int_0^1\bE_\bme^\tau \left[\frac{1}{2\pi\ri}\int\oint\frac{\theta(z;\beta)}{z-x}\rd z\rd\mu_n \right]\rd \tau\\
&=-\frac{1}{2\pi \ri}\oint \del_z\ln\cB_s \theta(z;\beta)\rd z+\OO\left(\frac{1}{n}\right)
=-\int \theta'(x;\beta) \del_s h_s^*\rd x +\OO\left(\frac{1}{n}\right),
\end{split}\end{align*}
and the error term $\OO(1/n)$ depends smoothly on $\beta$. This finishes the proof of Proposition \ref{p:deform}.
\end{proof}

\begin{proof}[Proof of Proposition \ref{p:Et1bb}]
To compute the numerator In \eqref{e:Zt1}, we can use Proposition \ref{p:difmeasure} with $\theta'(x;\beta)=F_{(s+1/n)-}(h_{s+1/n}^*)(x)/n$, 
then it follows that $Z_t^{(1)}(\beta)=\OO(1/n)$, and it depends on $\beta$ smoothly. Especially, we have $\del_\beta Z_t^{(1)}(\beta)=\OO(1/n)$.
The same as in Proposition \ref{p:FK}, we can construct a Markov process $\{\bmx_t\}_{t\in [s,T]\cap \bZ_n}$ with generator $\cL_t^{(1)}$ given by
\begin{align}\label{e:gener1}
\cL^{(1)}_{t} f(\bmx)=\sum_{\bme \in \{0,1\}^{m(t)}}\frac{a^{(1)}_t(\bme;\bmx)}{Z^{(1)}_t(\bmx)}(f(\bmx+\bme/n)-f(\bmx)),
 \quad \bmx=(x_1, x_2, \cdots, x_{m(t)})\in \fM_{t}(\fP),
\end{align}
where $a^{(1)}_t(\bme;\bmx), Z^{(1)}_t(\bmx)$ are as defined in \eqref{e:defas}, \eqref{e:Zt1}. The error term $E_s^{(1)}(\bmx)$ can be solved using Feynman-Kac formula
\begin{align}\label{e:esolve}
E_s^{(1)}(\bmx)
=\bE\left[\left.\prod_{t\in[s,T)\cap \bZ_n}Z^{(1)}_t(\bmx_t)\right|\bmx_s=\bmx\right].
\end{align}
By repeating the above analysis for $E_s(\bmx)$, i.e. plugging the characteristic flow $h_t^*=h_t^*(x;\beta,s)$ into \eqref{e:esolve}, we will have that
\begin{align}\label{e:esovle2}
E_s^{(1)}(\bmx)
\approx\prod_{t\in[s,T)\cap \bZ_n}Z^{(1)}_t(h^*_t).
\end{align}
We do not need explicit formula for $E_s^{(1)}(\bmx)$. Since $\del_\beta Z^{(1)}_t(\beta)=\OO(1/n)$, 
we will have that $\del_\beta E_s^{(1)}(\beta)=\OO(1)$, and especially,
\begin{align}\label{e:Et1bb}
\frac{E_t^{(1)}(\bmx+\bme/n)}{E_t^{(1)}(\bmx)}=e^{\OO(1/n)}.
\end{align}
\end{proof}

\subsection{Gaussian Fluctuations for Random Lozenge Tilings}
With the estimates for the error terms $E_s^{(1)}(\bmx)$ as in Proposition \ref{p:Et1bb}, in this section, we can check that the nonintersecting Bernoulli random walk \eqref{e:randomwalk}, which has the same law as random lozenge tiling,  is in the form of \eqref{e:defLtnew}. Then it follows from Corollary \ref{c:CLT},  the nonintersecting Bernoulli random walk \eqref{e:randomwalk} has Gaussian fluctuations.
\begin{proposition}\label{p:LozengeCLT}
Let $\{\bmx_t\}_{t\in[s,T]\cap\bZ_n}$ be the nonintersecting Bernoulli random walk \eqref{e:randomwalk} $\{\bmx_t\}_{t\in[s,T]\cap\bZ_n}$ starting at $\bmx_s=\bmx$, which corresponds to uniform random lozenge tilings. We denote  its empirical particle density and Stieltjes transform as
\begin{align*}
\rho_t(x;\bmx_t)=\sum_{i=1}^{m(t)}x_i(t),\quad \tilde m_t(z;\beta,s)=\int\frac{\rho_t(x;\bmx_t)}{z-x}\rd x,
\end{align*}
where $\beta=\beta(x;\bmx)$.

The random processes $\{n(\tilde m_t(z_t(Z);\beta,s)-\bE[\tilde m_t(z_t(Z);\beta,s)])\}_{s\leq t\leq t(Z), Z\in \gamma_s^{\beta,s}}$ converge weakly towards a centered Gaussian process $\{{\sf G}_t^\circ(Z)\}_{s\leq t\leq t(Z), Z\in \gamma_s^{\beta,s}}$ with zero initial data, which is the unique solution of the system of stochastic differential equations
\begin{align}\begin{split}\label{e:limitprocess2}
\rd\sfG_t^\circ(Z)&= \rd\sfW_t(Z)-\sfG_t^\circ(Z)\frac{\del_t \rd z_t}{\rd z_t}+\frac{1}{2\pi\ri}\oint_{\omega_+} \sfK_t(z_t(W),z_t;\beta,s)\sfG_t^\circ(W)\rd z_t(W)
\end{split}\end{align}
where $z_t=z_t(Z)$, the contour $\{z_t(W): W\in \omega\subset \gamma_t^{\beta,s,+}\}$  encloses $z^{-1}(\{x:(x,t+)\in \fP\})$ and $z_t$,
$\sfK_t(w,z_t;\beta,s)$ is given by
\begin{align*}
\sfK_t(w,z_t;\beta,s)=\frac{f_t(w;\beta,s)}{f_t(w;\beta,s)+1}\frac{1}{(w-z_t)^2}-\frac{1}{2\pi\ri}\oint_{\omega_-}\frac{f_t(u;\beta,s)}{f_t(u;\beta,s)+1}\frac{\int^w K(u,x;h_t^*,t)\rd x}{(u-z_t)^2}\rd u, 
\end{align*}
 and $\{\sfW_t(Z)\}_{s\leq t\leq t(Z), Z\in \gamma_s^{\beta,s}}$ is a centered Gaussian process with $\overline{W_t(Z)}=W_t(\bar Z)$, and covariance structure given by
\begin{align*}\begin{split}
&\langle \sfW(Z), \sfW(Z')\rangle_t=\int_s^t \frac{1}{2n\pi \ri}\oint_{\cin} \frac{f_r(w;\beta,s)}{f_r(w;\beta,s)+1}\frac{1}{(z_t(Z)-w)^2}\frac{1}{(z_t(Z')-w)^2}\rd w\rd r,
\end{split}\end{align*}
$\omega_-$ encloses $z^{-1}(\{x:(x,t+)\in \fP\})$ but not $z_t(Z), z_t(Z')$.
\end{proposition}

\begin{proof}

Using Proposition \ref{p:YYdif},  we can rewrite $a_s(\bme;\bmx)$ in the following form
\begin{align}\label{e:newat0}
a_s(\bme;\bmx)&=\frac{V(\bmx+\bme/n)}{V(\bmx)}\prod_{j} \phi^+(x_j;\beta,s)^{e_i}\phi^-(x_j;\beta,s)^{1-e_i}
e^{\frac{1}{n^2}\sum_{i} e_i e_j \kappa(x_i,x_j;\beta,s) +\zeta(\beta,s)+\OO(1/n)},
\end{align}
where $\beta=\beta(x;\bmx)$, and 
\begin{align*}
\phi^+(x;\beta,s)=\prod_{b\in B_{s}(x)} (b-x-1/n)e^{ \chi(x;\beta,s)},\quad 
\phi^-(x;\beta,s)=\prod_{a\in A_{s}(x)}  (x-a).
\end{align*}

More precisely, we recall $a_t(\bme,\bmx)$ from \eqref{e:newat0}. Using the defining relation \eqref{e:randomwalk}, Ansatz \eqref{a:defE}, \eqref{a:defE1} and Proposition \eqref{p:Et1bb}, we have
\begin{align}\begin{split}\label{e:firstP}
&\phantom{{}={}}\bP(\bmx((t+1/n)-)
=\bmx+\bme/n|\bmx(t)=\bmx)\propto a_t(\bme;\bmx)\frac{E_{(t+1/n)-}(\bmx+\bme/n)}{E_t(\bmx)}\\
&\propto a_t(\bme;\bmx)e^{F_{(t+1/n)-}(\bmx+\bme/n)-F_t(\bmx)}
\frac{E^{(1)}_{(t+1/n)-}(\bmx+\bme/n)}{E^{(1)}_t(\bmx)}=a_t(\bme;\bmx)e^{F_{(t+1/n)-}(\bmx+\bme/n)-F_t(\bmx)+\OO(1/n)},
\end{split}\end{align}
where $\propto$ means that there are some constant factor independent of $\bme$. Using \eqref{e:Fdidff}, we have
\begin{align}\begin{split}\label{e:secondP}
e^{F_{(t+1/n)-}(\bmx+\bme/n)-F_t(\bmx)}
&\propto 
e^{\int \del_\beta F_{(s+1/n)-}(h_{s+1/n}^*) (\beta(x;\bmx+\bme/n)-\beta(x;\bmx))\rd x+\OO(1/n)}\\
&=e^{\frac{1}{n}\sum_{i=1}^{m(t)}e_i \left.\del_ x\del_\beta F_{(s+1/n)-}(h_{s+1/n}^*) \right|_{x=x_i}+\OO(1/n)}.
\end{split}\end{align}
By plugging the expression \eqref{e:newat0} of $a_t(\bme;\bmx)$, and \eqref{e:secondP} into \eqref{e:firstP}, we get the following expression of $\bP(\bmx((t+1/n)-)
=\bmx+\bme/n|\bmx(t)=\bmx)$:
\begin{align*}
\frac{V(\bmx+\bme/n)}{V(\bmx)}\prod_{j} \left(\varphi_t^+(z;\beta,t)e^{\frac{1}{n}\psi_t^+(z;\beta,t)}\right)^{e_i}\phi_t^-(x_j;\beta,t)^{1-e_i}
e^{\frac{1}{n^2}\sum_{i} e_i e_j \kappa(x_i,x_j;\beta,t) +\OO(1/n)},
\end{align*}
where $\varphi_t^+(z;\beta,t), \varphi_t^-(x_j;\beta,t), \kappa(x_i,x_j;\beta,t)$ are defined in 
\eqref{e:defvarphi} and \eqref{e:defka}, and
\begin{align*}\begin{split}
\psi_t^+(z;\beta,t)&=-\sum_{b\in B_t(x)}n\ln\left(1-\frac{1}{n(b-z)}\right)+
\frac{1}{2}\del_z \ln  g_t'(z;\beta,t)\\
&+\del_t(\ln g_t'(z;\beta,t))+
\del_z\del_\beta F_{t+1/n}(h_{t+1/n}^*)+\OO(1/n).
\end{split}\end{align*}
We conclude that 
the nonintersecting Bernoulli random walk \eqref{e:randomwalk} is in the form of \eqref{e:defLtnew}, up to some $\OO(1/n)$ error. As we have seen in the proof of Section \ref{s:Gauss}, the $\OO(1/n)$ error contributes to an error of $\OO(1/n^2)$ for the Stieltjes transform of the empirical particle density, which is negligible for our analysis. Proposition \ref{p:LozengeCLT} follows from Corollary \ref{c:CLT}.

\end{proof}

\section{Identification as Gaussian Free Field}\label{s:GFF}
In Sections \ref{s:Gauss} and \ref{s:solve},  we have proven that the height fluctuations of random lozenge tilings of any polygonal domain ${\mathfrak P}$, which have exactly one horizontal upper boundary edge (as in Definition \ref{def:oneend}), are asymptotically Gaussian. In this section, we identify these fluctuations as a Gaussian Free Field on the liquid region, as predicted by  Kenyon and Okounkov \cite{MR2358053}. One way to do it is to directly analyze the limiting stochastic differential equation \eqref{e:limitprocess2} of the height fluctuations, and show it corresponds to a Gaussian Free Field.  This can be done, but will not help our understanding for the appearance of the Gaussian Free Field.
Instead, we will take another  more intuitive approach, based on the following heuristic  for the Kenyon-Okounkov conjecture \cite{MR2358053}
given by Vadim Gorin \cite[Lecture 12]{VG2020}. 
Imagine we have a probability density on the set of height functions, given by
\begin{align}\label{e:density}
\bP(h)\propto \exp\left(n^2 \int \int \sigma(\del_x h, \del_t h)\rd x\rd t\right)
\end{align}
Since we would like to study limiting fluctuations around $h^*$, we can let $h=h^*+\Delta h/n$, where under \eqref{e:density} $\Delta h$ is a random function supported in the liquid region which represents the fluctuations about the limiting height function $h^*$. Plugging this into \eqref{e:density}, we get that the probability density of $\Delta h$ is asymptotically given by
\begin{align}\label{e:Dhlaw}
\bP(\Delta h)\propto\exp\left(\pi \ri \int_{\Omega(\fP)}\frac{\rd \Delta h}{\rd f}\frac{\rd \Delta h}{\rd \bar f}\rd f\rd \bar f+\oo(1)\right),
\end{align}
where $f(x,t)$ is the complex slope on the liquid region $\Omega(\fP)$. The density \eqref{e:Dhlaw} is the Gaussian Free Field on the liquid region $\Omega(\fP)$, with complex structure given by the complex slope $f(x,t)$. Then we can conclude that the law of $\Delta h$ under \eqref{e:density} converges to as Gaussian Free Field on the liquid region $\Omega(\fP)$ with zero boundary condition 
\begin{align*}
\sqrt{\pi}\Delta h(x,t)\rightarrow {\rm GFF}, \quad (x,t)\in \Omega(\fP).
\end{align*} 

In our approach, we identify the lozenge tilings of the domain $\fP$ with nonintersecting Bernoulli walks. It is natural to view the probability density \eqref{e:density} on the set of height functions, as a dynamic $\{h_t\}_{0\leq t\leq T}$ of time slices of height functions.  In this way, the conditional distribution for $h_{t+\Delta t}=\beta+\delta \beta$, conditioned on $h_t=\beta$ is given by
\begin{align}\begin{split}\label{e:Pht}
&\phantom{{}={}}\bP(h_{t+\Delta t}=\beta+\delta \beta|h_t=\beta)\\
&\propto
\exp\left(n^2 \int_{t}^{t+\Delta t}\int \sigma(\del_x h, \delta\beta/\Delta t)\rd x\rd t\right)
\bE\left[\exp\left(n^2 \int_{t+\Delta t}^T \int \sigma(\del_x h, \del_t h)\rd x\rd t\right)\right]\\
&\propto
\exp\left(n^2 \int \sigma(\del_x \beta, \delta \beta/\Delta t)\rd x \Delta t\right)
\exp\left(n^2 W_{t+\Delta t}(\beta+\delta \beta)\right)\\
&\propto
\exp\left(n^2\left( \int \sigma(\del_x \beta, \delta \beta/\Delta t)\rd x \Delta t-\int\ln|f_t(x;\beta,t)| \delta \beta(x)\rd x\right)\right),
\end{split}\end{align}
where the expectation in the second line is over all height functions $\{h_{s}\}_{t+\Delta t\leq s\leq T}$ with $h_{t+\Delta t}=\beta+\delta \beta$, and it concentrates around $W_{t+\Delta t}(\beta+\delta \beta)$ as defined in \eqref{e:varWt} with negligible error term. To get the last line, we used Proposition \ref{p:Ws} that
\begin{align*}\begin{split}
W_{t+\Delta t}(\beta+\delta \beta)
&\approx W_{t}(\beta)+\int \del_\beta W_t(\beta)\delta \beta\rd x+\Delta_t \del_t W_t(\beta)\\
&= -\int \ln|f_t(x;\beta,t)|\delta \beta\rd x+ \text{constant independent of $\delta \beta$}.
\end{split}\end{align*}
To further simplify \eqref{e:Pht}, we need the following lemma on the higher derivatives of $\sigma(\del_x h, \del_t h)$.
\begin{lemma}
The derivatives of $\sigma(\del_x h, \del_t h)$ are given by
\begin{align*}
\frac{\del \sigma(\del_x h, \del_t h)}{\del (\del_t h)}
=\ln|f_t(x;\beta,t)|,\quad
\frac{\del^2 \sigma(\del_x h, \del_t h)}{\del (\del_t h)^2}
=\frac{\pi}{\Im[f_t(x;\beta,t)/(f_t(x;\beta,t)+1)]}.
\end{align*}
\end{lemma}
\begin{proof}
Let $\del_x h=\alpha$, and $\del_t h=\beta$. Then 
\begin{align*}
f=e^{-\ri\pi \alpha }\frac{\sin(-\beta \pi)}{\sin((\alpha+\beta)\pi)},
\quad 
f+1=e^{\ri\pi \beta }\frac{\sin(\alpha \pi)}{\sin((\alpha+\beta)\pi)},
\end{align*}
By plugging in
\begin{align*}
\frac{\del^2 \sigma(\alpha,\beta)}{\del \beta^2}
=\frac{\del \ln|f|}{\del \beta}
=\pi\left(\frac{\cos(\beta\pi)}{\sin(\beta \pi)}-\frac{\cos((\alpha+\beta)\pi)}{\sin((\alpha+\beta) \pi)}\right)
=\frac{\pi\sin(\alpha\pi)}{\sin(\beta\pi)\sin((\alpha+\beta)\pi)},
\end{align*}
and 
\begin{align*}
\frac{1}{\Im[f/(f+1)]}
=\frac{|f+1|^2}{\Im[f]}
=\frac{\sin(\alpha\pi)}{\sin(\beta\pi)\sin((\alpha+\beta)\pi)}.
\end{align*}
\end{proof}
Then we can do a Taylor expansion around $\del_t h^*_t(x;\beta,t)$
\begin{align}\begin{split}\label{e:sig}
&\phantom{{}={}}\sigma(\del_x \beta, \delta \beta/\Delta t)
=\sigma(\del_x \beta, \del_t h_t^*(x;\beta,t))
+\ln|f_t(x;\beta,t)|(\delta \beta/\Delta t-\del_t h_t^*(x;\beta,t))\\
&+\frac{1}{2}\frac{\pi}{\Im[f_t(x;\beta,t)/(f_t(x;\beta,t)+1)]}(\delta \beta/\Delta t-\del_t h_t^*(x;\beta,t))^2
+\OO(\delta \beta/\Delta t-\del_t h_t^*(x;\beta,t))^3.
\end{split}\end{align}
By plugging \eqref{e:sig} into \eqref{e:Pht}, we get
\begin{align*}
\bP(h_{t+\Delta t}=\beta+\delta \beta|h_t=\beta)\propto \exp\left(\frac{\pi}{2 \Delta t}\int\frac{n^2(\delta \beta-\Delta t\del_t h_t^*(x;\beta,t))^2}{\Im[f_t(x;\beta,t)/(f_t(x;\beta,t)+1)]}\rd x+\oo(1)\right).
\end{align*}
Therefore, the increase $\delta \beta$, after rescaling by a factor $n$, i.e. $n\delta \beta$ is asymptotically Gaussian. 

We denote the Stieltjes transform of the measure valued process $\{\del_x h_t\}_{0\leq t\leq T}$ as
\begin{align*}
\tilde m_t(z)=\int \frac{\del_x h_t}{z-x}\rd x,
\end{align*}
then its increase is
\begin{align}\label{e:increase}
\tilde m_{t+\Delta t}(z)-
\tilde m_t(z)
=\int \frac{\del_x h_{t+\Delta t}}{z-x}\rd x-\int \frac{\del_x h_t}{z-x}\rd x
=\int \frac{\del_x\delta \beta(x)\rd x}{z-x},
\end{align}
After normalize by a factor $n$, the mean of the increase \eqref{e:increase} is given by
\begin{align*}
\bE\left[\int \frac{n\del_x \delta \beta(x)\rd x}{z-x}\right]
=\Delta t \int \frac{\del_t \del_x h_t^*(x;\beta,t)\rd x}{z-x}
=\Delta t\del_t m_t(z;\beta,t),
\end{align*}
and covariance
\begin{align*}\begin{split}
\cov\left[\int \frac{\del_x\delta \beta(x)\rd x}{z-x}, \int \frac{\del_x\delta \beta(x)\rd x}{z'-x}\right]
&=-\frac{\Delta t}{\pi} \int \frac{\Im[f_t(x;\beta,t)/(f_t(x;\beta,t)+1)]}{(z-x)^2(z'-x)^2}\rd x\\
&=\frac{\Delta t}{2\pi\ri}\oint \frac{f_t(w;\beta,t)}{f_t(w;\beta,t)+1} \frac{1}{(z-w)^2(z'-w)^2}\rd w.
\end{split}\end{align*}

Therefore, we can rewrite the Stieltjes transform $\tilde m_t(z)$ of the measure valued process $\{\del_x h_t\}_{0\leq t\leq T}$ as
\begin{align}\label{e:diffeq9}
\frac{1}{\Delta t} (\tilde m_{t+\Delta t}(z)-
\tilde m_t(z))=\Delta M_t(z)+\del_tm_t(z;\beta,t)+\OO(\Delta t),
\end{align}
where the covariance of the Gaussian noise term $\Delta M_t(z)$ satisfies
\begin{align*}\begin{split}
\frac{1}{\Delta t} \cov\left[\Delta M_t(z),\Delta M_t(z')\right]
&=-\frac{\Delta t}{\pi} \int \frac{\Im[f_t(x;\beta,t)/(f_t(x;\beta,t)+1)]}{(z-x)^2(z'-x)^2}\rd x\\
&=\frac{1}{2\pi\ri}\oint \frac{f_t(w;\beta,t)}{f_t(w;\beta,t)+1} \frac{1}{(z-w)^2(z'-w)^2}\rd w.
\end{split}\end{align*}

By taking $\Delta t=1/n$, the leading order terms in the equation \eqref{e:diffeq9} is exactly the same as equation \eqref{e:leadingorder} for the Stieltjes transform of the nonintersecting Bernoulli random walks studied in Section \ref{s:Gauss}.
The fluctuations of the Stieltjes transform is determined by the Gaussian noise $\Delta M_t$ and the leading order term $\del_t m_t(z;\beta,t)$. The same analysis as in Section \ref{s:Gauss}, we will get that the statement of Corollary \ref{c:CLT} holds for the centered Stieltjes transform $(\tilde m_{t}(z)-
\bE[\tilde m_t(z)])/\Delta t$. As a consequence, the stochastic differential equation \eqref{e:limitprocess} is the Gaussian Free Field \eqref{e:Dhlaw}. Then Theorem \ref{c:GFF} follows from Corollary \ref{c:CLT}.

%
%
%
%


\bibliography{References.bib}
\bibliographystyle{abbrv}

\end{document}